\documentclass[11pt,leqno]{article}
\usepackage{amsmath, amscd, amsthm, amssymb, graphics, xypic, mathrsfs, setspace, fancyhdr, times, bm, pdfsync, enumitem}
\usepackage[usenames, dvipsnames, svgnames, table]{xcolor}
\usepackage[colorlinks=true,pagebackref=true]{hyperref} 
\hypersetup{backref}

\newcommand{\tensor}{\otimes}
\newcommand{\colim}{\operatorname{colim}}
\newcommand{\hofib}{\operatorname{hofib}}

\newcommand{\Spec}{\operatorname{Spec}}

\newcommand{\iso}{\cong}
\newcommand{\weq}{\simeq}
\newcommand{\isomto}{{\stackrel{\sim}{\;\longrightarrow\;}}}
\newcommand{\isomt}{{\stackrel{{\scriptscriptstyle{\sim}}}{\;\rightarrow\;}}}
\newcommand{\smallsim}{{\scriptscriptstyle{\sim}}}
\newcommand{\sma}{{\scriptstyle{\wedge}\,}}

\newcommand{\longhookrightarrow}{\lhook\joinrel\longrightarrow}
\newcommand{\longtwoheadrightarrow}{\relbar\joinrel\twoheadrightarrow}
\newcommand{\op}{\mathrm{op}}

\newcommand{\onto}{\twoheadrightarrow}

\renewcommand{\hom}{\operatorname{Hom}}
\newcommand{\Map}{\operatorname{Map}}
\newcommand{\id}{\iota}
\newcommand{\Id}{\mathrm{id}}

\newcommand{\real}{{\mathbb R}}
\newcommand{\cplx}{{\mathbb C}}
\newcommand{\Q}{{\mathbb Q}}
\newcommand{\Z}{{\mathbb Z}}

\newcommand{\aone}{{\mathbb A}^1}
\newcommand{\pone}{{\mathbb P}^1}

\newcommand{\gm}[1]{{{\mathbb G}_{m}^{#1}}}

\newcommand{\MW}{\mathrm{MW}}

\newcommand{\et}{\text{\'et}}
\newcommand{\ho}[2][]{\mathscr{H}_{#1}({#2})}

\newcommand{\M}{{\mathrm{M}}}

\newcommand{\bpi}{\bm{\pi}}
\newcommand{\piaone}{{\bpi}^{\aone}}

\newcommand{\Nis}{{\operatorname{Nis}}}

\newcommand{\Sigmas}{\mathfrak{S}}
\newcommand{\switch}{\wp}

\newcommand{\Sm}{\mathrm{Sm}}

\newcommand{\Ab}{\mathrm{Ab}}

\newcommand{\K}{{{\mathbf K}}}

\renewcommand{\H}{{{\mathbf H}}}

\newcommand{\cuppr}{\mathbin{\smile}}

\newcommand{\hocolim}{\operatornamewithlimits{hocolim}}

\newcommand{\Loc}{\mathrm{L}}
\newcommand{\Laone}{\Loc_{\aone}}
\newcommand{\Hoh}{\mathrm{H}}
\newcommand{\Eoh}{\mathrm{E}}
\newcommand{\ZZ}{\Z}

\newcommand{\Addresses}{{
 \bigskip
 \footnotesize

 A.~Asok, Department of Mathematics, University of Southern California, 3620 S. Vermont Ave.,
  Los Angeles, CA 90089-2532, United States; \textit{E-mail address:} \url{asok@usc.edu}

  \medskip

 K.~Wickelgren, Department of Mathematics, Georgia Institute of Technology, 686 Cherry Street
  Atlanta, GA 30308, United States; \textit{E-mail address:} \url{wickelgren@post.harvard.edu}

 \medskip

 T.B.\ Williams, Department of Mathematics, The University of British Columbia, 1984 Mathematics Road
Vancouver, B.C. Canada V6T 1Z2; \textit{E-mail address:} \url{tbjw@math.ubc.ca}

}}

\newcounter{intro}
\theoremstyle{plain}
\newtheorem{thm}{Theorem}[subsection]

\newtheorem{lem}[thm]{Lemma}
\newtheorem{cor}[thm]{Corollary}
\newtheorem{prop}[thm]{Proposition}
\newtheorem*{claim*}{Claim} 

\newtheorem{conj}[thm]{Conjecture}
\newtheorem*{thm*}{Theorem}
\newtheorem*{problem*}{Problem}

\newtheorem{thmintro}{Theorem}

\theoremstyle{definition}
\newtheorem{defn}[thm]{Definition}
\newtheorem{construction}[thm]{Construction}
\newtheorem{notation}[thm]{Notation}
\newtheorem{convention}[thm]{Convention}

\theoremstyle{remark}
\newtheorem{rem}[thm]{Remark}
\newtheorem{remintro}[thmintro]{Remark}

\newtheorem{entry}[thm]{}

\numberwithin{equation}{subsection}

\begin{document}
\pagestyle{fancy}
\renewcommand{\sectionmark}[1]{\markright{\thesection\ #1}}
\fancyhead{}
\fancyhead[LO,R]{\bfseries\footnotesize\thepage}
\fancyhead[LE]{\bfseries\footnotesize\rightmark}
\fancyhead[RO]{\bfseries\footnotesize\rightmark}
\chead[]{}
\cfoot[]{}
\setlength{\headheight}{1cm}

\author{Aravind Asok\thanks{Aravind Asok was partially supported by National Science Foundation Award DMS-1254892.} \and Kirsten Wickelgren\thanks{Kirsten Wickelgren was partially supported by National Science Foundation Award DMS-1406380} \and Ben Williams}

\title{{\bf The simplicial suspension sequence in $\aone$-homotopy}}
\date{}
\maketitle

\begin{abstract}
We study a version of the James model for the loop space of a suspension in unstable ${\mathbb A}^1$-homotopy theory. We
use this model to establish an analog of G.W.~Whitehead's classical refinement of the Freudenthal suspension theorem in
${\mathbb A}^1$-homotopy theory: our result refines F. Morel's ${\mathbb A}^1$-simplicial suspension theorem. We then
describe some $E_1$-differentials in the EHP sequence in ${\mathbb A}^1$-homotopy theory. These results are analogous to
classical results of G.W.~Whitehead's. Using these tools, we deduce some new results about unstable ${\mathbb
  A}^1$-homotopy sheaves of motivic spheres, including the counterpart of a classical rational non-vanishing result.
\end{abstract}

\begin{footnotesize}
\tableofcontents
\end{footnotesize}

\section{Introduction}
If $K$ is an $(n-1)$-connected pointed CW complex, then the suspension map ${\mathrm E} :
\pi_{q}(K) \to \pi_{q+1}(\Sigma K)$ fits into a long exact sequence of the form:
\[
\begin{split}
\xymatrix{
\pi_{3n-2}(K) \ar^-{{\mathrm E}}[r]& \pi_{3n-1}(\Sigma K) \ar^-{{\mathrm H}}[r] & \pi_{3n-1}(\Sigma K^{\sma 2}) \ar^-{{\mathrm P}}[r] & \pi_{3n-3}(K) \ar^-{{\mathrm E}}[r] & \cdots } \\
\xymatrix{
 \cdots \ar[r] & \pi_q(K) \ar^-{{\mathrm E}}[r] & \pi_{q+1}(\Sigma K) \ar^-{{\mathrm H}}[r] & \pi_{q+1}(\Sigma K^{\sma 2}) \ar^-{{\mathrm P}}[r] &
  \pi_{q-1}(K) \ar[r] & \cdots.} \end{split}
\]
Together with an elementary connectivity estimate for $\Sigma K^{\sma 2}$, this exact sequence may be viewed as a refinement of the Freudenthal suspension theorem. The exact sequence above was first constructed by G.W. Whitehead if $K = S^n$ \cite[Theorem 1 p. 211]{WhiteheadFreudenthalTheorems}, and by W.D. Barcus for $K$ as above \cite[Proposition 2.9]{Barcus} (see also \cite[Theorem XII.2.2 p. 543]{Whitehead} for a textbook treatment of the general statement).

The morphisms ${\mathrm H}$ and ${\mathrm P}$ appearing in the above exact sequence were also studied by Whitehead in great detail in the case where $K = S^n$ \cite[\S 10]{WhiteheadHopf}.  The morphism $\mathrm{H}$ is the Hopf invariant, and Whitehead linked the morphism ${\mathrm P}$ with Whitehead products.  In more detail, begin by observing that the $(n+1)$-fold suspension ${\mathrm E}^{n+1}: \pi_{q-n}(S^n) \to \pi_{q+1}(S^{2n+1})$ is an isomorphism for $q < 3n-1$. Define ${\mathrm P}': \pi_{q-n}(S^n) \to \pi_{q-1}(S^n)$ by ${\mathrm P}'(\alpha) = [\alpha,\id_n]$, where $\id_n$ is the identity map on the $n$-sphere and the bracket denotes Whitehead product. For ${\mathrm P}: \pi_{q+1}(S^{2n+1}) \to \pi_{q-1}(S^n)$, Whitehead observed that
\[
{\mathrm P} = {\mathrm P}' \circ ({\mathrm E}^{n+1})^{-1} \text{ if } q < 3n-1.
\]
While Whitehead established this result for spheres, it has been known for some time that morphism ${\mathrm P}$ is, for general $(n-1)$-connected spaces, still closely related to Whitehead products; see, e.g., \cite[\S 2]{SuspensionTriad} or \cite[Theorem 3.1 and p. 231]{Ganea} for a very general statement. In any case, these kinds of tools were used to great effect in early computations of unstable homotopy groups of spheres, e.g., by James and Toda \cite{JamesConstruction, Jamessuspensionsequence, Toda}.

The goal of this paper, whose title pays homage to the work of James \cite{Jamessuspensionsequence}, is to establish analogs of the above results in the Morel--Voevodsky unstable $\aone$-homotopy category \cite{MV} and to deduce some consequences of these results. The jumping-off point is to give a James-style model for the loop space of a suspension in $\aone$-homotopy theory (see Theorem \ref{thm:aonejamesconstruction}). Using this model, we deduce the following result, which can be thought of as a refinement of Morel's $\aone$-simplicial suspension theorem \cite[Theorem 6.61]{MField}.

\begin{thmintro}[See Theorem \ref{thm:EHP_range}, Remark \ref{rem:refinesfreudenthal} and Theorem \ref{P=piWhitehead}]
\label{thmintro:aoneehpsequence}
Assume $k$ is a perfect field. If $\mathscr X$ is a pointed $\aone$--$(n-1)$-connected simplicial presheaf on $(\Sm_k)_{\Nis}$, with $n \ge 2$, then there is an exact sequence of $\aone$-homotopy sheaves of the form:
\[
\begin{split}
\xymatrix{\piaone_{3n-2}(\mathscr X) \ar^-{\mathrm E}[r] &
\piaone_{3n-1}(\Sigma \mathscr{X}) \ar^-{{\mathrm H}}[r] & \piaone_{3n-1}(\Sigma \mathscr{X}^{\sma 2}) \ar^-{{\mathrm P}}[r] & \piaone_{3n-3}(\mathscr{X}) \ar^-{{\mathrm E}}[r] & \cdots
} \\
\xymatrix{
 \cdots \ar[r] & \piaone_q(\mathscr{X}) \ar^-{{\mathrm E}}[r] & \piaone_{q+1}(\Sigma \mathscr{X})\ar^-{{\mathrm H}}[r] & \piaone_{q+1}(\Sigma \mathscr{X}^{\sma 2}) \ar^-{{\mathrm P}}[r] &
  \piaone_{q-1}(\mathscr{X}) \ar[r] & \cdots; }
 \end{split}
\]
the map ${\mathrm E}$ is (simplicial) suspension, the map ${\mathrm H}$ is a James--Hopf invariant, and the map ${\mathrm P}$ is described, as above, in terms of Whitehead products.
\end{thmintro}

We go on to discuss various consequences of the existence of this exact sequence. We analyze the low-degree portion of this sequence in Theorem \ref{thm:lowdegree} and give a more explicit description of the sequence in the first degree where the suspension map fails to be an isomorphism. When $\mathscr{X}$ is a motivic sphere, it is shown in \cite{WickelgrenWilliams} that the exact sequences of Theorem \ref{thmintro:aoneehpsequence} can be extended to all degrees after localizing at $2$. By suitably varying the input sphere, these sequences can be strung together to obtain the EHP spectral sequence converging to the $2$-local $S^1$-stable $\aone$-homotopy sheaves of spheres.

By construction, the $E_1$-differentials in this spectral sequence arise from the composite map ${\mathrm H}{\mathrm
  P}$, which in certain degrees we can analyze integrally. To state the result, recall that Morel computed the first
non-vanishing $\aone$-homotopy sheaf of a motivic sphere in terms of Milnor--Witt K-theory \cite[Corollary
6.43]{MField}. He also showed that there is an isomorphism of rings $K^{\MW}_0(k) \iso GW(k)$, i.e., the zeroth Milnor--Witt K-theory group of a field $k$ is isomorphic to the Grothendieck-Witt ring of isomorphism classes of symmetric bilinear forms over $k$ \cite[Lemma 3.10]{MField}, defined to be the group completion of the monoid of isomorphism classes of non-degenerate symmetric bilinear forms. Given this terminology, the class of the composite ${\mathrm H}{\mathrm P}$ can be seen to correspond with a symmetric bilinear form, which we can describe. More precisely, we establish the following result (see the statement in the body of the text and Remark \ref{rem:zmod2equivariantinterpretation} for a more conceptual explanation of the formula).

\begin{thmintro}[See Theorem \ref{thm:HP}]
\label{thmintro:HP}
Assume $k$ is a perfect field, and let $p,q$ be integers with $p > 1, q \geq 1$. The map
\[
{\mathrm H}{\mathrm P}:\K^{\MW}_{2q} = \bpi_{2p+3}^{\aone}(\Sigma (S^{p + 1+ q \alpha})^{\wedge 2}) \longrightarrow \bpi_{2p+1}(\Sigma (S^{p + q \alpha})^{\wedge 2}) = \K^{\MW}_{2q}
\]
corresponds to the element $\langle 1 \rangle+ (-1)^{p+1+q}\langle -1 \rangle^q \in GW(k)$.
\end{thmintro}

One consequence of this result is the following analog of the classical fact, due to Hopf, that $\pi_{4n-1}(S^{2n})$ is non-trivial.

\begin{thmintro}[See Theorem \ref{thm:rationalized}]
Fix a base field $k$ assumed to be perfect and to have characteristic unequal to $2$. Let $n,q \geq 2$ be even integers, and let $j$ be an integer. There is a surjection
\[
\bpi^{\aone}_{2n-1+j\alpha} S^{n + q \alpha} \tensor \Q \stackrel{{\mathrm H} \tensor \Q}{\longrightarrow} \K^{\MW}_{2q-j} \tensor \Q,
\]
and the sheaf $\bpi^{\aone}_{2n-1+j\alpha} S^{n + q \alpha} \tensor \Q$ is non-trivial if either $k$ is formally real or if $j \leq 2n-1$.
\end{thmintro}

Relying on the computations of \cite{AsokFaselA3minus0}, we analyze the low-degree portion of the EHP sequence in great
detail in the special case were $X = {\mathbb A}^3 \setminus 0$. In particular, we give a description of the next
non-vanishing $\aone$-homotopy sheaf (i.e., beyond that computed by Morel) of $\Sigma {\mathbb A}^3 \setminus 0 \iso {\pone}^{\sma 3}$ in Theorem \ref{thm:pi45ponesmash3}. The following statement is an easy-to-state special case of a more general result.

\begin{thmintro}[See Theorem \ref{thm:pi45ponesmash3}]
\label{thmintro:computationpi45ponesmash3}
If $k$ is a field of characteristic $0$ and containing an algebraically closed subfield, then, for any integer $i \geq 0$, there is an isomorphism of sheaves of the form:
\[
\bpi_{4+i+ 5\alpha}^{\aone}(S^{3+i}_s \wedge \gm{\sma 3} ) \iso \Z/24.
\]
\end{thmintro}

\begin{remintro}
Cohomology of homotopy sheaves of spheres such as those above appears in concrete applications to problems in algebra via techniques of obstruction theory; see, e.g., \cite{AsokFaselThreefolds,AsokFaselA3minus0} for more details. Our description of the sheaf $\bpi_4^{\aone}(S^{3+3\alpha})$ is well-suited to such cohomology computations. Our computation allows us to state a precise conjecture (see Conjecture \ref{conj:structureofpinplus1}) regarding the structure of the sheaf $\bpi_n^{\aone}({\mathbb A}^n \setminus 0)$ for $n \geq 4$. An explicit description of the sheaf $\bpi_n^{\aone}({\mathbb A}^n \setminus 0)$ for $n = 2,3$ was a key step in \cite{AsokFaselThreefolds,AsokFaselA3minus0} in the resolution of
Murthy's conjecture regarding splitting of rank $n$ vector bundles on smooth affine $(n+1)$-folds over algebraically closed fields. A resolution of Conjecture \ref{conj:structureofpinplus1} would, similarly, imply Murthy's conjecture in general.
\end{remintro}

We close this introduction with some general comments regarding prerequisites. When working with the (unstable) $\aone$-homotopy category in general and Morel's $\aone$-algebraic topology in particular, with the goal of making this paper as self-contained as possible, we have labored to present the material in an axiomatic framework involving the ``unstable $\aone$-connectivity property", which is introduced in Section \ref{ss:unstableconnectivity}. All of the results in Sections \ref{s:jamesconstruction} and \ref{s:aoneehpsequence} are written from this axiomatic perspective. We hope this style of presentation makes the material accessible to people who have some familiarity with homotopy theory of simplicial presheaves and the constructions of \cite{MV}, but not, for example, all of the technical results about strongly and strictly $\aone$-invariant sheaves contained in the first five chapters of \cite{MField}. Moreover, we hope that our presentation also makes \cite{MField} itself more accessible to the non-expert.

For the most part, Section \ref{s:eonedifferential} is written in the same axiomatic framework. In contrast, Sections \ref{ss:HPKmwsubsection} and \ref{s:applications} require more background. In particular, this portion of the text requires familiarity with facts about strongly and strictly $\aone$-invariant sheaves (see Section \ref{ss:milnorwittktheory} for more precise statements), and known explicit computations of homotopy sheaves. In Section \ref{s:applications}, we also appeal to structural results from the theory of quadratic forms and both the affirmation of the Milnor conjecture on quadratic forms and the Bloch--Kato conjecture.

\subsubsection*{Acknowledgements}
The first-named author would like to thank Fr\'ed\'eric D{\'e}glise for his hospitality during a stay at ENS-Lyon in September 2012 when parts of this work were conceived and Fabien Morel for pointing out many years ago that a James-style model for (simplicial) loops on the suspension exists in $\aone$-homotopy theory. This project was initially conceived as joint with Jean Fasel, and the authors would also like to sincerely thank him for his collaboration in the formative stages. Some of this work was done while the second-named author was in residence at MSRI during the Spring 2014 semester supported by NSF grant 0932078 000.  The authors are grateful to have been afforded the opportunity to work together in several places including the University of Duisberg--Essen during a special semester on motivic homotopy theory organized by Marc Levine (funded by the Alexander von Humboldt Foundation and the German Research Foundation) and the American Institute of Mathematics.  Finally, we thank the referees for a very careful reading and a number of thoughtful suggestions that improved the clarity of the presentation.

\subsubsection*{Notation}
Throughout, the (undecorated) symbol $S$ will be used to denote a base scheme assumed Noetherian and of finite Krull dimension. We write $\Sm_S$ for the category of schemes that are separated, smooth and have finite type over $S$. Script letters, e.g., $\mathscr{X},\mathscr{Y}$, will typically be used to denote ``spaces", i.e., pointed simplicial presheaves on $\Sm_S$ (from Section \ref{ss:unstableconnectivity} onward), while capital roman letters will typically be used to denote simplicial presheaves on more general sites. Typically, boldface letters will be used to denote strongly $\aone$-invariant sheaves of groups (again, from Section \ref{ss:unstableconnectivity} onward), with the exception of $\mathbf{C}$, which will always mean a category (often equipped with the structure of a site) and $\mathbf{R}$, which will be used to denote right derived functors.

Sheaf cohomology will always be taken with respect to the Nisnevich topology. See \ref{notation:spheres} for our conventions regarding motivic spheres; unfortunately the letter $S$ appears in our notation for spheres, but since it will always be decorated with a superscript, we hope no confusion arises. See \ref{notation:homotopysheaves} for a summary of notation pertaining to homotopy sheaves, \ref{convention:connectivity} for some discussion of our connectivity conventions, \ref{convention:fibersequences} for some recollections on our use of the term fiber sequence, \ref{convention:relativeconnectivity} for conventions about relative connectivity, \ref{notation:aonehomology} for notation regarding $\aone$-homology sheaves, and \ref{notation:aonetensorproduct} for notation regarding the $\aone$-tensor product. Finally, our conventions for loop spaces change in Sections \ref{ss:pi4ponesmash3} and \ref{ss:miscellaneouscomputations}; see \ref{convention:loops} for more details.

\section{The James construction revisited}
\label{s:jamesconstruction}
The James construction on a CW complex was originally introduced by I.M.~James in \cite{JamesConstruction}. Milnor observed that the construction could be recast in the language of simplicial sets \cite[p. 120]{MilnorFK}. Using this translation, it is straightforward to develop a version of the James construction in the category of simplicial presheaves. Section \ref{ss:classicaljames} reviews the James construction in the category of simplicial sets and extends these constructions to simplicial presheaves; the main result in the context of simplicial presheaves is Proposition \ref{prop:J(X)=ho-nis=Omega_SigmaX}.

In this section, we aim to develop a version of the James model in $\aone$-homotopy category; this idea is due originally to F.~Morel. Section \ref{ss:unstableconnectivity} recalls a number of structural properties of the $\aone$-homotopy category that will be used throughout the work: we point the reader to Definition \ref{defn:unstableaoneconnectivity}, Lemma \ref{lem:unstableconnectivityconsequences} and Theorem \ref{thm:unstablenconnectivity}. Section \ref{ss:furtherproperties} studies some aspects of $\aone$-fiber sequences in the context of the axiomatic setup of Section \ref{ss:unstableconnectivity}. Section \ref{ss:loopspaces} proves the main result, i.e., Theorem \ref{thm:aonejamesconstruction}, which provides a James-style model for loops on the suspension of a space in $\aone$-homotopy theory; this result depends on results about the Kan loop group in $\aone$-homotopy theory, for which we refer the reader to Theorem \ref{thm:kanloopgroupmodelaonelocal}.

\subsection{The James construction in simplicial homotopy theory}
Textbook treatments of the James construction can be found in \cite[Chapter VII.2]{Whitehead}, for the category of CW complexes, and \cite[\S 3.3.3]{wu} in the category of simplicial sets.

\label{ss:classicaljames}
\subsubsection*{The James construction for simplicial sets}
Let $K$ be a pointed simplicial set. An injection $\alpha: (1,2,\ldots,m) \to (1,2,\ldots, n)$ induces a map $\alpha_* : K^m \to K^n$. Let $\smallsim$ denote the equivalence relation on $\coprod_{n=0}^{\infty} K^n$ generated by $x \sim \alpha_*(x)$ for all {\em order-preserving} injections $\alpha$. The
James construction on $K$ is defined by the formula
\[
J(K) := \coprod_{n=0}^{\infty} K^n / \sim,
\]
i.e., $J(K)$ is the free (pointed) monoid on the pointed simplicial set $K$. The assignment $K \mapsto J(K)$ is functorial in $K$ by definition. The James construction is filtered by pointed simplicial sets $J_n(X) \subset J(K)$, defined by
\[
J_n(K) := \coprod_{m=0}^{n} K^m / \sim.
\]
We consider also $F(K)$, the pointed free group functor as in \cite[p. 293]{GoerssJardine} or \cite[\S 3.2]{wu}. Because
$J(K)$ is the free pointed monoid on $K$, there is an evident inclusion map $J(K) \hookrightarrow F(K)$. If $\Sigma K$
denotes the Kan suspension \cite[p. 191]{GoerssJardine}, and $G(K)$ denotes the Kan loop group
\cite[p. 276]{GoerssJardine}, then Milnor showed that there is a weak equivalence $F(K) \weq G(\Sigma K)$ \cite[Theorem V.6.15]{GoerssJardine}. By \cite[Corollary V.5.11]{GoerssJardine}, since $\Sigma K$ is reduced, we conclude that $F(K)$ is a model for $\Omega \Sigma K$; here $\Omega$ is the derived loops (for a model, take naive loops on a fibrant model of the input). The following result details the main properties of $J(-)$.

\begin{thm}[{\cite[Theorems 3.24 and 3.25]{wu}}]
\label{thm:classicalJames}
Suppose $K$ is a pointed simplicial set.
\begin{enumerate}[noitemsep,topsep=1pt]
\item If $K$ is connected, there is a weak equivalence $J(K) \weq F(K)$.
\item For any integer $n \geq 1$ there is a cofiber sequence of the form
\[
J_{n-1}(K) \longhookrightarrow J_n(K) \longrightarrow K^{\sma n}.
\]
\item The canonical map $\colim_n J_n(K) \to J(K)$ is an isomorphism.
\end{enumerate}
\end{thm}

\begin{rem}
\label{rem:j1issuspension}
Consider the map $K = J_1(K) \to J(K)$. Under the weak equivalence $J(K) \weq \Omega \Sigma K$ of Point (1) of Theorem \ref{thm:classicalJames}, this map corresponds to the unit map $K \to \Omega \Sigma K$ of the loops-suspension adjunction.
\end{rem}

\subsubsection*{The James construction for simplicial presheaves}
Suppose $\mathbf{C}$ is a site. We will consider (pointed) simplicial presheaves on $\mathbf{C}$, though we do not introduce any special notation for this category. The category of (pointed) simplicial presheaves can be equipped with its injective local model structure \cite{Jardine}: cofibrations are given by monomorphisms, weak equivalences are defined locally with respect to the Grothendieck topology on $\mathbf{C}$ and fibrations are defined via the right-lifting property. Abusing terminology slightly, we will refer to the associated homotopy category as the (pointed) simplicial homotopy category.  The category of pointed simplicial presheaves is a pointed category, i.e., the canonical map from the initial to the final object is an isomorphism. We will typically write $\ast$ for the final object. With this terminology, one can extend the definition of the James construction to (pointed) simplicial presheaves in a straightforward fashion by applying the constructions above sectionwise.

\begin{defn}
\label{defn:J_etc_def}
Assume $X$ is a pointed simplicial presheaf on $\mathbf{C}$, and an $n \geq 0$ is an integer. Define pointed simplicial presheaves $G(X)$, $F(X)$, $J(X)$ and $J_n(X)$ by assigning to $U \in \mathbf{C}$ the following simplicial sets:
\[
\begin{split}
G(X)(U) &:= G(X(U)), \\
F(X)(U) &:= F(X(U)), \\
J(X)(U) &:= J(X(U)), \text{ and } \\
J_n(X)(U) &:= J_n(X(U)).
\end{split}
\]
We refer to the pointed simplicial presheaf $J(X)$ as the {\em James construction of $X$}, and $G(X)$ as the Kan loop group of $X$.
\end{defn}

\begin{rem}
Since we have not assumed $X$ to be reduced (i.e., having presheaf of $0$-simplices the constant presheaf $\ast$) in the above, the simplicial presheaf $G(X)$ will not in general have the homotopy type of the loop space of $X$ (e.g., take $X = \Delta_1$ the simplicial interval, in which case $G(\Delta^1)$ is the constant simplicial group on $\Z$, which is not contractible).
\end{rem}

The assignments $X \mapsto G(X)$, $X \mapsto F(X)$, $X \mapsto J(X)$, $X \mapsto J_n(X)$ are all evidently functorial in $X$. Moreover, there are morphisms $J_n(X) \to J(X)$ for any integer $n \geq 0$. We distinguish the morphism
\begin{equation}
\label{eqn:Emap}
{\mathrm E}: X = J_1(X) \longrightarrow J(X);
\end{equation}
we will refer to this morphism as {\em suspension} (cf. Remark \ref{rem:j1issuspension}).

\begin{prop}
\label{prop:simplicialKanloopgroup}
Suppose $X$ is a reduced pointed simplicial presheaf on $\mathbf{C}$, i.e., the presheaf of $0$-simplices is $\ast$. There is a weak equivalence
\[
G(X) \weq \Omega X.
\]
\end{prop}

\begin{proof}
This follows by observing that the induced map on sections over $U \in \mathbf{C}$ is a weak-equivalence by \cite[Corollary V.5.11]{GoerssJardine}.
\end{proof}

\begin{prop}
\label{prop:J(X)=ho-nis=Omega_SigmaX}
Suppose $X$ is a pointed simplicial presheaf on $\mathbf{C}$.
\begin{enumerate}[noitemsep,topsep=1pt]
\item The map $J(X) \to F(X)$ is a sectionwise equivalence.
\item The simplicial presheaf $J(X)$ is locally weakly equivalent to $\Omega \Sigma X$.
\end{enumerate}
Both of the weak equivalences just mentioned are functorial in $X$.
\end{prop}

\begin{proof}
Point (1) follows immediately from Theorem \ref{thm:classicalJames}(1): the inclusion maps $J(X(U) \to F(X(U))$ are weak equivalences; these maps are evidently functorial in $U$. For Point (2), observe that \cite[Theorem V.6.15]{GoerssJardine} guarantees that for any object $U \in \mathbf{C}$ there is a functorial weak equivalence
\[
G(\Sigma X(U)) \weq F(X(U)),
\]
where $\Sigma$ is the Kan suspension. Since $\Sigma X(U)$ is by construction reduced, the result then follows from Proposition \ref{prop:simplicialKanloopgroup}.
\end{proof}

\subsubsection*{The classifying space of the Kan loop group}
Suppose $H$ is a simplicial presheaf of groups. If $Y$ is a simplicial presheaf equipped with a right action $a: Y \times H \to Y$ of $H$, we will say that the action is {\em categorically free} if the morphism
\[
\xymatrix{
Y \times H \ar[r]^{(a,p_Y)}& Y \times Y
}
\]
is a monomorphism. If $Y$ carries a categorically free action of $H$, we write $Y/H$ for the quotient, i.e., the colimit of the diagram $Y \leftarrow Y \times H \rightarrow Y$.

\begin{entry}
\label{entry:barconstruction}
Following \cite[Chapter 7]{may1975}, if $H$ is a presheaf of groups, $Y$ is a space with a right $H$-action and $Z$ is a space with a left $H$-action, we may form the two-sided bar construction $B(Y,H,Z)$ as the ``geometric realization" of a certain functorially-constructed simplicial object $B(Y,H,Z)_{\bullet}$ having $n$-simplices of the form $Y \times H^{\times n} \times Z$.  In the present context, $B(Y,H,Z)_{\bullet}$ is a simplicial object in the category of simplicial presheaves and the ``geometric realization" is the homotopy colimit over $\Delta^{op}$, i.e.,
\[
B(Y,H,X) := \hocolim_{n \in \Delta^\op} B(Y, H, Z)_n.
\]
When $Y = X = \ast$, then we use the following special terminology: by the {\em simplicial classifying space $BH$} we mean $B(\ast,H,\ast)$, by the {\em universal bundle} $EH$ we mean $B(\ast,H,H)$, and by the {\em Borel construction} we mean $B(Y,H,\ast)$.
\end{entry}

The next result, which is a (pre)sheaf-theoretic variant of a classical fact about the Kan loop group (see, e.g., \cite[Point (5) p. 137]{Curtis}), follows immediately from Proposition \ref{prop:simplicialKanloopgroup}; we use this result in the next section.

\begin{prop}
\label{prop:classifyingspaceofkanloops}
For $X$ any reduced pointed simplicial presheaf on $\mathbf{C}$, there is a sectionwise weak equivalence $X \weq BG(X)$.
\end{prop}

\subsection{The unstable \texorpdfstring{$\aone$}{A1}-connectivity property}
\label{ss:unstableconnectivity}
Before discussing the James construction in $\aone$-homotopy theory, we will recall some facts from $\aone$-algebraic
topology. We take $\mathbf{C} = \Sm_S$, i.e., the category of smooth schemes over $S$.

This category will be endowed throughout with the Nisnevich topology, as in \cite[\S 3]{MV}, and the category of simplicial
presheaves on $\Sm_S$ may be equipped with a \textit{simplicial} model structure, \cite{Jardine}, local with respect to
this topology. That is, the cofibrations are the monomorphisms of simplicial presheaves, and the weak equivalences may
be detected on Nisnevich stalks.  We warn the reader that in \cite{MField} and \cite{MV}, contrary to our conventions, the motivic homotopy category is constructed using simplicial sheaves.  In \cite[Theorem 1.2]{JardineMSS}, Jardine shows that the sheafification and the forgetful functor define an adjoint equivalent between the two theories.

By a {\em pointed space}, we will mean pointed simplicial presheaf on $\Sm_S$. A model structure for the $\aone$-homotopy category $\ho{S}$ can be constructed by left Bousfield localization of the simplicial model structure of simplicial presheaves on $\Sm_S$.

We will adopt the convention, at variance with that of \cite{hirschhorn2003}, that homotopy limits will be calculated by
first applying a functorial fibrant replacement objectwise to the diagram in question.

\subsubsection*{$\aone$-localization}
Recall from \cite[\S 3.2]{MV} the notion of an $\aone$-local object. It will be useful to remember that the simplicial homotopy
limit of a diagram of $\aone$-local objects is again $\aone$-local, this is the case because the fibrant,
$\aone$-local objects are the fibrant objects of a model category, and we may use \cite[Theorem 18.5.2]{hirschhorn2003}. We begin by recalling the basic properties of ``the" $\aone$-localization functor.

\begin{prop}
\label{prop:aonelocalization}
There exists an endofunctor $\Laone$ of the category of simplicial presheaves on $(\Sm_S)_{\Nis}$ and a natural transformation $\theta: \Id \to \Laone$ such that, for any space $\mathscr{X}$ the following statements hold.
\begin{itemize}[noitemsep,topsep=1pt]
\item[(i)] The space $\Laone \mathscr{X}$ is fibrant, $\aone$-local and the map $\mathscr{X} \to \Laone\mathscr{X}$ is an $\aone$-weak equivalence;
\item[(ii)] If $\mathscr{Y}$ is any simplicially fibrant $\aone$-local space, and $f: \mathscr{X} \to \mathscr{Y}$ is a morphism, then $f$ factors as $\mathscr{X} \to \Laone \mathscr{X} \to \mathscr{Y}$.
\item[(iii)] The functor $\Laone$ commutes with the formation of finite limits.
\end{itemize}
\end{prop}

\begin{proof}[Comments on the proof]
 Most of this statement is contained in \cite{MV}; we slightly modify the $\aone$-localization functor given in
 \cite[p. 107]{MV}. The functor is constructed by repeated application of the singular construction and a fibrant
 replacement functor in the category of simplicial presheaves. The singular construction commutes with limits (see
 \cite[p. 87]{MV}). We use the Godement resolution functor of \cite[\S 2 Theorem 1.66]{MV} as our functorial fibrant
 replacement for simplicial presheaves; this commutes with formation of finite limits by construction.
\end{proof}

In Morel's analysis, a distinguished role is played by Eilenberg--MacLane spaces \cite[p. 56]{MV} or classifying spaces of Nisnevich sheaves of groups \cite[p. 128]{MV} that are $\aone$-local.

\begin{defn}
\label{defn:aoneinvarianceproperties}
A sheaf of groups $\mathbf{G}$ is called {\em strongly $\aone$-invariant} if $B\mathbf{G}$ is $\aone$-local. A sheaf of abelian groups $\mathbf{A}$ is called strictly $\aone$-invariant if $K(\mathbf{A},i)$ is $\aone$-local for every $i \geq 0$.
\end{defn}

\subsubsection*{Homotopy sheaves and the unstable $\aone$-connectivity property}
Suppose $\mathscr{X}$ and $\mathscr{Y}$ are pointed spaces. We write $[\mathscr{Y},\mathscr{X}]_s$ for morphisms in the homotopy category of the injective local model structure (we will refer to this category as the {\em simplicial homotopy category}) and $[\mathscr{Y},\mathscr{X}]_{\aone}$ for morphisms in the $\aone$-homotopy category. We now fix some conventions that will remain in force throughout the paper.

\begin{notation}[Spheres, suspension and looping]
\label{notation:spheres}
Write $S^i_s$ for the simplicial $i$-sphere and $\gm{\sma j}$ for the $j$-fold smash product of $\gm{}$ (pointed by $1$)
with itself. Following conventions of $\Z/2$-equivariant homotopy theory, we write $S^{i+j\alpha}$ for the sphere
$S^i_s \sma \gm{\sma j}$. If $j = 0$, the $\gm{}$-term shall be dropped from the notation. The undecorated symbol $\Sigma$
will be used for simplicial suspension. Likewise, we use $\Omega$ for the derived simplicial loops, a model for which is
obtained by first taking a functorial fibrant replacement of the input and then applying naive loops.
\end{notation}

\begin{notation}[Homotopy sheaves]
\label{notation:homotopysheaves}
We define homotopy sheaves $\bpi_i(\mathscr{X},x)$ and $\aone$-homotopy sheaves $\bpi_i^{\aone}(\mathscr{X},x)$ as Nisnevich sheaves associated with the presheaves $U \mapsto [S^i_s \wedge U_+,(\mathscr{X},x)]_{s}$ and $U \mapsto [S^i_s \wedge U_+,(\mathscr{X},x)]_{\aone} = [S^i_s \wedge U_+,(\Laone \mathscr{X},x)]_s$. Likewise, $\bpi_{i+j\alpha}^{\aone}(\mathscr{X},x)$ is the sheafification for the Nisnevich topology of the presheaf $U \mapsto [S^{i+j\alpha} \sma U_+,(\mathscr{X},x)]_{\aone}$. For notational compactness, base-points will typically be suppressed from notation.
\end{notation}

\begin{convention}[Connectivity]
\label{convention:connectivity}
We borrow various bits of terminology from classical homotopy theory: a pointed space $(\mathscr{X},x)$ is simplicially connected (resp. $\aone$-connected) if the sheaf $\bpi_0(\mathscr{X})$ (resp. $\bpi_0^{\aone}(\mathscr{X})$) is $\ast$. Similarly, if $n \geq 1$ is an integer, we will say that $(\mathscr{X},x)$ is simplicially $n$-connected (resp. $\aone$-$n$-connected) if $(\mathscr{X},x)$ is simplicially (resp. $\aone$-)connected and $\bpi_i(\mathscr{X},x)$ (resp. $\bpi_i^{\aone}(\mathscr{X})$) is trivial for $i \leq n$.
\end{convention}

Morel's approach to $\aone$-algebraic topology in \cite{MField} consists in studying $\aone$-local spaces via their Postnikov towers and, in doing this, it is important to understand the structural properties of $\aone$-homotopy sheaves. Our discussion is inspired by Morel's axiomatic discussion of the so-called stable $\aone$-connectivity property in \cite[\S 6]{MStable}.

\begin{defn}
\label{defn:unstableaoneconnectivity}
We will say that the {\em unstable $\aone$-connectivity property holds for $S$} if the following two properties hold.
\begin{itemize}[noitemsep,topsep=1pt]
\item For any pointed space $(\mathscr{X},x)$, $\bpi_1^{\aone}(\mathscr{X})$ is strongly $\aone$-invariant.
\item Any strongly $\aone$-invariant sheaf of abelian groups $\mathbf{A}$ is strictly $\aone$-invariant.
\end{itemize}
\end{defn}

For the most part, the results in this text will be proven assuming that the unstable $\aone$-connectivity holds over $S$. From this point of view, one of the key results of \cite{MField} is the following.

\begin{thm}[Morel]
\label{thm:unstableaoneconnectivityperfectfields}
If $S$ is the spectrum of a perfect field, then the unstable $\aone$-connectivity property holds for $S$.\footnote{The careful reader will want to assume that $k$ is, in addition, infinite, since the published proof of \cite[Lemma 1.15]{MField}, a result due to Gabber, requires this restriction.}
\end{thm}

\begin{proof}
See \cite[Theorems 5.46 and 6.1]{MField}.
\end{proof}

\begin{rem}
\label{rem:unstableoneconnectivityproperty}
The unstable $\aone$-connectivity property does not hold if $S$ is a Noetherian scheme of Krull dimension $\geq 2$; see Remark \ref{rem:stableaoneconnectivityproperty} for more details. Nevertheless, it is expected that the unstable $\aone$-connectivity property holds for the spectrum of an arbitrary field.  In fact, the perfection of the base field only intercedes in the verification of Point (2) of \textup{Definition \ref{defn:unstableaoneconnectivity}}.  There is some hope that it may hold for base schemes $S$ that are regular of dimension $\leq 1$.
\end{rem}

\subsubsection*{Fiber sequences}
\begin{convention}[Fiber and cofiber sequences]
\label{convention:fibersequences}
We use the terminology ``(co)fiber sequence" as in \cite[Definition 6.2.6]{Hovey}; we refer the reader there for more formal properties of (co)fiber sequences. We use the terminology {\em simplicial fiber sequence} for a fiber sequence in the injective local model structure on simplicial presheaves and {\em $\aone$-fiber sequence} for a fiber sequence in the $\aone$-local model structure. The theory of fiber sequences is simplified slightly by the fact that the injective local and $\aone$-local model structures are right proper.
\end{convention}

The following result, which is a version of \cite[Lemma 6.51]{MField}, studies the behavior of $\aone$-local objects in simplicial fiber sequences (cf. \cite[e.6 p. 5]{DrorFarjoun} for a completely analogous result for localizations of the classical homotopy category).  There appears to be a misprint in the statement of \cite[Lemma 6.51]{MField}; the statement concerns $\aone$-connectivity whereas the hypothesis used and the conclusion reached appear to concern $\aone$-locality.

\begin{lem}
\label{lem:Mfield6.51}
Suppose
\[
\xymatrix{ \mathscr{F} \ar[r] & \mathscr{E} \ar[r] & \mathscr{B} }
\]
is a simplicial fiber sequence of pointed spaces. If $\mathscr{B}$ and $\mathscr{F}$ are both $\aone$-local and $\mathscr{B}$ is simplicially connected, then $\mathscr{E}$ is $\aone$-local as well.
\end{lem}

\begin{proof}
The proof is that given in \cite[Lemma 6.51]{MField}, but we have added some details for the convenience of the reader and for the sake of completeness.  By means of the existence of functorial factorizations, we may replace $\mathscr{E} \to \mathscr{B}$ by a simplicial fibration between simplicially fibrant pointed spaces, and we may assume $\mathscr{F}$ is the fiber over the basepoint of this map.

Write $\mathbf{R}\hom(\aone,\cdot)$ for the derived internal mapping object in the category of presheaves with the simplicial model structure. The model structure is closed monoidal, by \cite[Section 4]{barwick_left_2010} for example, and the functor $\mathbf{R}\hom(\aone,\cdot)$ is a right Quillen functor, \cite[Chapter 4]{Hovey}.  In particular, this means that $\mathbf{R}\hom(\aone,\cdot)$ preserves simplicial fiber sequences.

Condition (2) of \cite[Lemma 2.2.8]{MV}, combined with the definition of internal mapping objects, allows us to say a space $\mathscr{X}$ is $\aone$-local if and only if the map $\mathscr{X} \to \mathbf{R}\hom(\aone,\mathscr{X})$ induced by the projection map $\aone \to \ast$ is a simplicial weak equivalence. In the case of pointed spaces, we forget the basepoint and then apply this test.

Combining the above observations, one concludes that there is a morphism of simplicial fiber sequences of the form:
\begin{equation}
\label{eq:ladder}
\xymatrix{\mathscr{F} \ar[r]\ar^{\sim}[d] & \mathscr{E} \ar[r] \ar[d] & \mathscr{B} \ar^{\sim}[d] \\
    \mathbf{R}\hom(\aone,\mathscr{F}) \ar[r] &\mathbf{R}\hom(\aone,\mathscr{E})  \ar[r] & \mathbf{R}\hom(\aone,\mathscr{B});
    } \end{equation}
the two indicated arrows are simplicial weak equivalences.

It is now possible to test whether the map $\mathscr{E} \to \mathbf{R}\hom(\aone,\mathscr{E})$ a simplicial weak equivalence by arguing at points of the Nisnevich site. We refer to \cite[Chapter 3 \& Lemma 5.12]{jardine_local_2015} for the required local homotopy theory. If $p^*$ is point, then from Diagram \eqref{eq:ladder} we obtain a morphims of fiber sequences of Kan complexes of the form:
\begin{equation*} \xymatrix{p^*\mathscr{F} \ar[r]\ar^{\sim}[d] & p^*\mathscr{E} \ar[r] \ar[d] & p^*\mathscr{B} \ar^{\sim}[d] \\
    p^*\mathbf{R}\hom(\aone,\mathscr{F}) \ar[r] &p^*\mathbf{R}\hom(\aone,\mathscr{E})  \ar[r] & p^*\mathbf{R}\hom(\aone,\mathscr{B})
    } \end{equation*}
We wish to show that the map in the middle is a weak equivalence of (possibly disconnected) simplicial sets; that is, it induces an isomorphism on $\pi_0$ and on all homotopy groups for all choices of basepoint in $p^*\mathscr{E}$. It suffices, since $p^* \mathscr{E} \to p^*\mathscr{B}$ is a fibration, to consider basepoints lying in $p^*\mathscr{F}$, the fiber over the canonical basepoint of $p^*\mathscr{B}$.

The required isomorphism on homotopy groups and $\pi_0$ follows from a $5$-lemma argument; the potentially problematic case of $\pi_0$ is handled by identifying $\pi_0(p^*\mathscr{E})$ with the orbit space of $\pi_0(p^*\mathscr{F})$ under the action of $\pi_1(p^*\mathscr{B}, b_0)$, and similarly for the derived mapping spaces.
\end{proof}

\subsubsection*{Basic consequences of the unstable $\aone$-connectivity theorem}
We very briefly recall the Postnikov tower in the form we will use. For any pointed simplicially connected space $\mathscr{X}$, there is a tower of fibrations of the form
\[
\xymatrix{
 & & \mathscr{X} \ar[dl]\ar[d] \ar[dr] & & \\
\cdots \ar[r]^{p_{i+2}}& \mathscr{X}^{(i+1)} \ar[r]^{p_{i+1}} & \mathscr{X}^{(i)} \ar[r]^{p_{i}} & \mathscr{X}^{(i-1)} \ar[r]^{p_{i-1}} & \cdots
}
\]
such that (a) for any integer $j \geq 0$, $\hofib(p_{j}) = K(\bpi_{j}(\mathscr{X}),j)$, and (b) $\mathscr{X} \weq \operatorname{holim}_i \mathscr{X}^{(i)}$ \cite[Definition 1.31 and Theorem 1.37]{MV}.
We now deduce some consequences of the unstable $\aone$-connectivity property.

\begin{lem}
\label{lem:unstableconnectivityconsequences}
Suppose the unstable $\aone$-connectivity property holds for $S$.
\begin{enumerate}[noitemsep,topsep=1pt]
\item For any pointed space $(\mathscr{X},x)$, and any integer $i \geq 2$, the sheaves $\bpi_i(\Laone\mathscr{X},x)$ are strictly $\aone$-invariant.
\item If $(\mathscr{X},x)$ is a pointed, simplicially connected space, then $\mathscr{X}$ is $\aone$-local if and only if $\bpi_1(\mathscr{X},x)$ is strongly $\aone$-invariant and for any integer $i \geq 2$, $\bpi_i(\mathscr{X},x)$ is strictly $\aone$-invariant.
\end{enumerate}
\end{lem}

\begin{proof}
For (1), $\Omega^i \Laone \mathscr{X}$ is $\aone$-local and simplicially fibrant for every $i \geq 1$. In particular,
the sheaf $\bpi_1(\Omega^{i-1} \Laone \mathscr{X})$ is strongly $\aone$-invariant. Since this is abelian when $i \geq 2$, we conclude by the assumption that the unstable $\aone$-connectivity property holds that $\bpi_i(\Laone\mathscr{X},x)$ is strictly $\aone$-invariant for $i \geq 2$.

For Point (2), we use the existence and convergence of the Postnikov tower, together with an induction argument in combination with the results of Point (1).  Using this tower, it suffices to show that if $\mathscr{X}$ is a pointed, simplicially connected $\aone$-local space, and we have a simplicial fiber sequence of the form
\[
K(\mathbf{A},n) \longrightarrow \mathscr{X}' \longrightarrow \mathscr{X},
\]
where $\mathbf{A}$ is strongly $\aone$-invariant if $n = 1$, and strictly $\aone$-invariant if $n \geq 2$, then $\mathscr{X}'$ is $\aone$-local as well.  Either assumption guarantees that $K(\mathbf{A},n)$ is $\aone$-local and the result then follows by appeal to Lemma \ref{lem:Mfield6.51}.
\end{proof}

The following result, which is called the unstable $\aone$-connectivity theorem, justifies our terminology; this result is an axiomatic form of \cite[Theorem 6.38]{MField}.

\begin{thm}[Unstable $\aone$-connectivity theorem]
\label{thm:unstablenconnectivity}
Suppose $n \geq 0$ is an integer, and $(\mathscr{X},x)$ is a pointed, simplicially $n$-connected space.  The space $\Laone \mathscr{X}$ is simplicially connected, and if the unstable $\aone$-connectivity property holds for $S$, then it is simplicially $n$-connected.
\end{thm}

\begin{proof}
The case $n = 0$ of the theorem follows from \cite[\S 2 Corollary 3.22]{MV} and does not require the unstable $\aone$-connectivity property to hold.  That result is presented without a proof in \cite{MV}, but follows from the properties of the $\aone$-localization functor.  For a detailed proof, see, e.g., \cite[Theorem 1.2.20]{Strunk}.

Now, we treat the case $n = 1$. We begin by establishing a general result. For any simplicially connected space $\mathscr{X}$, and any Nisnevich sheaf of groups $\mathbf{G}$, there is a functorial bijection
\[
\hom(\bpi_1(\mathscr{X},x),\mathbf{G}) \iso [(\mathscr{X},x),B\mathbf{G}]_s
\]
by obstruction theory \cite[Lemma B.7(1)]{MField}. If $\mathbf{G}$ is strongly $\aone$-invariant, then $B\mathbf{G}$ is $\aone$-local, and there are functorial bijections of the form:
\[
\begin{split}
[(\mathscr{X},x),B\mathbf{G}]_s &\iso [(\mathscr{X},x),B\mathbf{G}]_{\aone} \\
&\iso [(\Laone \mathscr{X},x),B\mathbf{G}]_{\aone} \\
&\iso [(\Laone \mathscr{X},x),B\mathbf{G}]_s.
\end{split}
\]
Since $\Laone \mathscr{X}$ is simplicially connected by \cite[\S 2 Corollary 3.22]{MV}, we conclude that there is a functorial bijection of the form
\[
\hom(\bpi_1(\Laone \mathscr{X},x),\mathbf{G}) \iso  [(\Laone \mathscr{X},x),B\mathbf{G}]_s.
\]
Now, $\bpi_1(\Laone \mathscr{X},x) = \bpi_1^{\aone}(\mathscr{X})$ by definition, so combining all of the above isomorphisms, we conclude that if $\mathbf{G}$ is strongly $\aone$-invariant, then
\begin{equation}
\label{eqn:homoutofpi1}
\hom(\bpi_1(\mathscr{X},x),\mathbf{G}) \iso \hom(\bpi_1^{\aone}(\mathscr{X}),\mathbf{G}).
\end{equation}
Having established this bijection, we can proceed to the proof of the main result.

If $\mathscr{X}$ is simplicially $1$-connected, then $\hom(\bpi_1(\mathscr{X},x),\mathbf{G}) = 0$ for any sheaf of groups $\mathbf{G}$. If $\mathbf{G}$ is furthermore strongly $\aone$-invariant, we conclude by the isomorphism of (\ref{eqn:homoutofpi1}) that $\hom(\bpi_1^{\aone}(\mathscr{X}),\mathbf{G}) = 0$. Since the unstable $\aone$-connectivity property holds for $k$, we know that $\bpi_1^{\aone}(\mathscr{X})$ is strongly $\aone$-invariant. Therefore, by the Yoneda lemma, we know that $\bpi_1^{\aone}(\mathscr{X})$ must be trivial.

For the general case, one proceeds by induction on $n$. If $\mathscr{X}$ is a simplicially $(n-1)$-connected space, $n \geq 2$, then for any sheaf of abelian groups $\mathbf{A}$,
\[
\hom(\bpi_n(\mathscr{X}),\mathbf{A}) \iso [(\mathscr{X},x),K(\mathbf{A},n)]_s;
\]
this follows from \cite[Lemma B.7(2)]{MField}. An argument completely analogous to the one above, this time using Lemma \ref{lem:unstableconnectivityconsequences}(1) to conclude that the higher $\aone$-homotopy sheaves are strictly $\aone$-invariant, shows that if $\mathbf{A}$ is any strictly $\aone$-invariant sheaf, then
\[
\hom(\bpi_n(\mathscr{X}),\mathbf{A}) \iso \hom(\bpi_n^{\aone}(\mathscr{X}),\mathbf{A}).
\]
If $\mathscr{X}$ is simplicially $n$-connected, then as before we can again conclude by appealing to the Yoneda lemma.
\end{proof}

\begin{rem}
\label{rem:unstable0connectivityproperty}
We add one comment about the $n = 0$ case of the above theorem.  In fact, \cite[\S 2 Corollary 3.22]{MV} establishes a more general statement that we will frequently use below: if $\mathscr{X} \to \mathscr{X}'$ is an $\aone$-weak equivalence with $\mathscr{X}'$ an $\aone$-local space, then the induced morphism $\bpi_0(\mathscr{X}) \to \bpi_0(\mathscr{X}')$ is an epimorphism.
\end{rem}

\subsection{Further consequences of the unstable \texorpdfstring{$\aone$}{A1}-connectivity property}
\label{ss:furtherproperties}
In this section, we study further consequences of the unstable $\aone$-connectivity property introduced in Section \ref{ss:unstableconnectivity}. In particular, we recast some results of F. Morel in our axiomatic framework. First, we provide an analog of Proposition \ref{prop:simplicialKanloopgroup} in the context of $\aone$-homotopy theory (see Theorem \ref{thm:kanloopgroupmodelaonelocal}); this result is a key technical tool in all that follows. In particular, it allows us to establish Theorem \ref{thm:MField6.53}, which is a statement about preservation of simplicial fiber sequences under $\aone$-localization. Consequences of this result include a relative version of the unstable connectivity theorem, which appears below as Corollary \ref{cor:relativehurewicz}, and Theorem \ref{thm:MField6.60}, which is a technical result about the interaction between Postnikov towers and $\aone$-localization.

\subsubsection*{On $\aone$-homotopy types of connected spaces}
We begin by establishing a result about the behavior of the classifying space of the Kan loop group under $\aone$-localization; this result is culled from the proof of \cite[Theorem 6.46]{MField}.  Suppose $G$ is a simplicial presheaf of groups. Since $\Laone$ preserves finite products, $\Laone G$ is again a simplicial presheaf of groups, the morphism $G\to L_{\aone}G$ is a homomorphism, and there is an induced morphism
\begin{equation}
\label{eqn:aonelocalmodel}
BG \longrightarrow BL_{\aone}G.
\end{equation}
Regarding this morphism, one has the following result.

\begin{lem} \label{lem:BGtoBLaoneG}
If $G$ is a simplicial presheaf of groups, then the functorial map $B G \to B \Laone G$ is an $\aone$-weak-equivalence.
\end{lem}

\begin{proof}
Recall from \ref{entry:barconstruction} that
\[
B G = \hocolim_{n \in \Delta^\op} B( \ast, G, \ast)_n.
\]
Since the map $G^{\times n} \to (\Laone G)^{\times n}$ is an $\aone$-weak equivalence, and since $\hocolim$ preserves such equivalences \cite[\S 2 Lemma 2.12]{MV}, it follows that $BG \to B(\Laone G)$ is an $\aone$-weak equivalence.
\end{proof}

\begin{thm}
\label{thm:kanloopgroupmodelaonelocal}
Suppose the unstable $\aone$-connectivity property holds for $S$ and $(\mathscr{X},x)$ is a pointed reduced space. If $\pi_0(\Laone G(\mathscr{X}))$ is strongly $\aone$-invariant, then the objects $\Laone \mathscr{X}$ and $B\Laone
G(\mathscr{X})$ are simplicially weakly equivalent.
\end{thm}

\begin{proof}
By Proposition \ref{prop:classifyingspaceofkanloops}, since $\mathscr{X}$ is reduced, we know that there is a simplicial
weak equivalence of the form $\mathscr{X} \weq BG(\mathscr{X})$. In particular, $\Laone\mathscr{X} \weq \Laone
BG(\mathscr{X})$. It will be sufficient to prove that there is a simplicial weak equivalence $\Laone B G( \mathscr{X})
\weq B \Laone G(\mathscr{X})$.

We show that $B\Laone G(\mathscr{X})$ is $\aone$-local. To this end, consider the sequence
\[
\Laone G(\mathscr{X}) \longrightarrow E\Laone G(\mathscr{X}) \longrightarrow B\Laone G(\mathscr{X}).
\]
We know this is a simplicial fiber sequence by reference to \cite{wendt2011} for instance.
This fiber sequence yields a long exact sequence of homotopy sheaves. Since $E\Laone G(\mathscr{X})$ is simplicially contractible, we conclude that
\[
\bpi_{i+1}(B\Laone G(\mathscr{X})) \iso \bpi_i(\Laone G(\mathscr{X}))
\]
for every integer $i \geq 0$. The space $B\Laone G(\mathscr{X})$ is simplicially connected as it is the classifying
space of a simplicial group. By assumption $\bpi_0(\Laone G(\mathscr{X}))$ is strongly $\aone$-invariant, so we conclude
that $\bpi_1(B\Laone G(\mathscr{X}))$ is strongly $\aone$-invariant. Since the unstable $\aone$-connectivity property
holds for $S$, we conclude from Lemma \ref{lem:unstableconnectivityconsequences}(1) that $\bpi_{j}(B\Laone
G(\mathscr{X}))$ is strictly $\aone$-invariant for $j \geq 2$. Therefore, applying Lemma
\ref{lem:unstableconnectivityconsequences}(2), we conclude that $B\Laone G(\mathscr{X})$ is itself $\aone$-local, and
the map $B \Laone G(\mathscr{X}) \to \Laone B \Laone G(\mathscr{X})$ is a simplicial weak equivalence. By Lemma
\ref{lem:BGtoBLaoneG}, the map $\Laone B G(\mathscr{X}) \to \Laone B \Laone G(\mathscr{X})$ is also a simplicial weak
equivalence. It follows that there is a map $B \Laone G(\mathscr{X}) \to \Laone B G (\mathscr{X})$ that is also a
simplicial weak equivalence, as required.
\end{proof}

\subsubsection*{On $\aone$-fiber sequences}
The following result is a slight variant of \cite[Theorem 6.53]{MField}, which is presented there without proof.

\begin{thm}
\label{thm:MField6.53}
Assume the unstable $\aone$-connectivity property holds for $S$. Suppose
\[
\mathscr{F} \longrightarrow \mathscr{E} \stackrel{f}{\longrightarrow} \mathscr{B}
\]
is a simplicial fiber sequence of pointed spaces and assume that $\mathscr{B}$ is simplicially connected. If, in addition, $\bpi_0^{\aone}(\Omega \mathscr{B})$ is strongly $\aone$-invariant, then the canonical map
\[
\Laone \mathscr{F} = \Laone \hofib(f) \longrightarrow \hofib (\Laone(f))
\]
is a simplicial weak equivalence. In particular, if $\bpi_1(\mathscr{B})$ is strongly $\aone$-invariant (e.g., trivial) then the canonical map of the previous display is a simplicial weak equivalence.
\end{thm}

\begin{proof}
The idea is to replace the simplicial fiber sequence in question by a ``principal" fibration under the Kan loop group of the base and use Theorem \ref{thm:kanloopgroupmodelaonelocal}. To this end, we begin with a reduction. \newline

\noindent {\bf Step 1.}  For any simplicial presheaf $\mathscr{X}$, the map $\mathscr{X}(U) \to Ex \mathscr{X}(U)$ is a (functorial) simplicial weak equivalence.  Thus, without loss of generality we can assume that $\mathscr{B}$ is objectwise fibrant.  Since $\mathscr{B}$ is connected, we can furthermore assume that $\mathscr{B}$ is also reduced.  To see this, write $\mathscr{B}(0)$ for the zeroth level of the (sectionwise) Moore-Postnikov factorization of $\mathscr{B}$, and set $\mathscr{B}^{(0)}$ to be the pullback of the diagram:
\[
\ast \longrightarrow \mathscr{B}(0) \longleftarrow \mathscr{B};
\]
the space $\mathscr{B}^{(0)}$ is called the $0$-th Eilenberg subcomplex of $\mathscr{B}$ (cf. \cite[p. 327, Proof of Lemma VI.3.6]{GoerssJardine}).  By the existence of functorial factorizations, we may assume that, replacing $\mathscr{E}$ by a simplicially weakly equivalent space if necessary, that $\mathscr{E} \to \mathscr{B}$ is a simplicial fibration. In that case, we replace $\mathscr{E}$ by the pullback of the diagram $\mathscr{B}^{(0)} \to \mathscr{B} \leftarrow \mathscr{E}$; this does not change the simplicial homotopy type of the fiber.   \newline

\noindent {\bf Step 2.} Since $\mathscr{B}$ is now assumed reduced, set $\mathscr{G} := G(\mathscr{B})$.  Next, we claim that the simplicial fiber sequence is equivalent to the simplicial fiber sequence associated with a principal fibration under the Kan loop group.  Indeed, since $\mathscr{B}$ is reduced, then $\mathscr{B} \weq B\mathscr{G}$ by Proposition \ref{prop:classifyingspaceofkanloops}.  Now, a priori there is an action of the $h$-group $\Omega \mathscr{B}$ on $\mathscr{F}$.  Since $\mathscr{E} \to \mathscr{B}$ is a simplicial fibration by assumption, if we set $\mathscr{F}'$ to be the pullback of $E\mathscr{G} \to B\mathscr{G} \weq \mathscr{B} \leftarrow \mathscr{E}$, then $\mathscr{F}'$ is simplicially weakly equivalent to $\mathscr{F}$ and carries an honest action of $\mathscr{G}$ (cf. \ref{entry:barconstruction} for our conventions regarding two-sided bar constructions).  One then checks that there is an induced
simplicial weak equivalence $B(\mathscr{F}',\mathscr{G},\ast) \to \mathscr{E}$ making the simplicial fiber sequence $\mathscr{F}' \to B(\mathscr{F}',\mathscr{G},\ast) \to B\mathscr{G}$ (see \cite[Proposition 7.9]{may1975}) weakly equivalent to the fiber sequence $\mathscr{F} \to \mathscr{E} \to \mathscr{B}$ with which we began. \newline

\noindent {\bf Step 3a.} We now study what happens under $\aone$-localization.  First, since the hypotheses of Theorem \ref{thm:kanloopgroupmodelaonelocal} are satisfied by assumption, and the map $\mathscr{G} \to \Laone \mathscr{G}$ induces a simplicial weak equivalence $\Laone \mathscr{B} \weq \Laone B\mathscr{G} \to B\Laone\mathscr{G}$.  In particular, $B \Laone\mathscr{G}$ is $\aone$-local.

On the other hand, there is a simplicial fiber sequence of the form
\[
\Laone \mathscr{F}' \longrightarrow B(\Laone \mathscr{F}',\Laone \mathscr{G},\ast) \longrightarrow B\Laone \mathscr{G}.
\]
Note that $\Laone \mathscr{F}'$ is $\aone$-local by assumption and $B\Laone \mathscr{G}$ is $\aone$-local and simplicially connected.  Therefore, by appeal to Lemma \ref{lem:Mfield6.51}, we conclude that $B(\Laone \mathscr{F}',\Laone \mathscr{G},\ast)$ is $\aone$-local as well.

Next, the maps $\mathscr{F}' \to \Laone \mathscr{F}'$ and $\mathscr{G} \to \Laone \mathscr{G}$ induce $\aone$-weak equivalences $\mathscr{F}' \times \mathscr{G}^{\times n} \to \Laone \mathscr{F'} \times (\Laone \mathscr{G})^{\times n}$ for all $n$.  Therefore, the induced morphism
\[
B(\mathscr{F}',\mathscr{G},\ast) \longrightarrow B(\Laone \mathscr{F}',\Laone \mathscr{G},\ast)
\]
is an $\aone$-weak equivalence by \cite[\S 2 Lemma 2.12]{MV}.

If $R_{\Nis}$ denotes a the simplicial fibrant replacement functor, then the map $B(\Laone \mathscr{F}',\Laone \mathscr{G},\ast) \to R_{\Nis}B(\Laone \mathscr{F}',\Laone \mathscr{G},\ast)$ is a simplicial weak equivalence and $R_{\Nis}B(\Laone \mathscr{F}',\Laone \mathscr{G},\ast)$ is $\aone$-fibrant.  Therefore, the composite map $B(\mathscr{F}',\mathscr{G},\ast) \to R_{\Nis}B(\Laone \mathscr{F}',\Laone \mathscr{G},\ast)$ is also an $\aone$-weak equivalence and factors through a simplicial weak equivalence of the form:
\[
\Laone B(\mathscr{F}',\mathscr{G},\ast) \longrightarrow R_{\Nis}B(\Laone \mathscr{F}',\Laone \mathscr{G},\ast).
\]

\noindent {\bf Step 3b.} The evident projections give morphisms
\begin{align*}
\Laone B(\mathscr{F}',\mathscr{G},\ast) & \longrightarrow \Laone B\mathscr{G}, \text{ and } \\
R_{\Nis}B(\Laone \mathscr{F}',\Laone \mathscr{G},\ast)&\longrightarrow R_{\Nis} B\Laone \mathscr{G}\end{align*}
that fit into the following commutative diagram:
\[
\xymatrix{
\Laone B(\mathscr{F}',\mathscr{G},\ast) \ar[r]\ar[d] & R_{\Nis}B(\Laone \mathscr{F}',\Laone \mathscr{G},\ast) \ar[d] \\
\Laone B\mathscr{G} \ar[r]& R_{\Nis} B \Laone \mathscr{G}.
}
\]
The simplicial homotopy fiber of the first column is $\hofib(\Laone f)$ by construction  while the simplicial homotopy fiber of the second column is $\Laone \hofib(f)$. Moreover, the diagram gives rise to a morphism of simplicial fiber sequences. Since the horizontal maps in the diagram are simplicial weak equivalences by the conclusions of the previous step, we conclude that the induced map of simplicial homotopy fibers is a simpicial weak equivalence.

The final statement is an immediate consequence of Lemma \ref{lem:simpliciallysimplyconnectedguaranteeshypothesis} below.
\end{proof}

\begin{lem}[{\cite[Lemma 6.54]{MField}}]
\label{lem:simpliciallysimplyconnectedguaranteeshypothesis}
If $\mathscr{X}$ is a pointed connected space, such that $\bpi_1(\mathscr{X}) = \bpi_0(\Omega\mathscr{X})$ is $\aone$-invariant, then the morphism
\[
\bpi_0(\Omega \mathscr{X}) \longrightarrow \bpi_0(\Laone \Omega \mathscr{X})
\]
is an isomorphism. In particular, if $\bpi_1(\mathscr{X})$ is strongly $\aone$-invariant, then so is $\bpi_0(\Laone \Omega \mathscr{X})$.
\end{lem}

\begin{proof}
By \cite[\S 2 Corollary 3.22]{MV}, the morphism $\bpi_0(\Omega \mathscr{X}) \to \bpi_0(\Laone \Omega \mathscr{X})$ is always an epimorphism (cf. Remark \ref{rem:unstable0connectivityproperty}). Since $\bpi_0(\Omega \mathscr{X})$ is $\aone$-invariant and has simplicial dimension $0$ it is necessarily simplicially fibrant and therefore $\aone$-local by \cite[\S 2 Proposition 3.19]{MV}. Therefore, the morphism $\Omega \mathscr{X} \to \bpi_0(\Omega \mathscr{X})$ necessarily factors through $\Laone \Omega \mathscr{X}$. Applying $\bpi_0$, we see that the identity map of $\bpi_0(\Omega \mathscr{X})$ factors through $\bpi_0(\Laone \Omega \mathscr{X})$. Using the epimorphism from the first sentence, we conclude that $\bpi_0(\Omega \mathscr{X}) \to \bpi_0(\Laone \Omega \mathscr{X})$ is an isomorphism.
\end{proof}

\subsubsection*{The relative unstable $\aone$-connectivity theorem}
Theorem \ref{thm:MField6.53} can be used to establish a relative version of the unstable connectivity theorem, Theorem \ref{thm:unstablenconnectivity}. In the form below, this result is a variant of \cite[Theorem 6.56]{MField} and will be used repeatedly in the sequel.

\begin{convention}[Relative connectivity]
\label{convention:relativeconnectivity}
If $f: \mathscr{E} \to \mathscr{B}$ is a morphism of pointed spaces, we will say that $f$ is simplicially $i$-connected or a simplicial $i$-equivalence if the simplicial homotopy fiber of $f$ is $(i-1)$-connected. Likewise, we will say that $f$ is $\aone$-$i$-connected or an $\aone$-$i$-equivalence if the $\aone$-homotopy fiber of $f$ is $\aone$-$(i-1)$-connected.
\end{convention}

\begin{cor}
\label{cor:relativehurewicz}
Assume the unstable $\aone$-connectivity property holds for $S$, and suppose $f: \mathscr{E} \to \mathscr{B}$ is a morphism of pointed spaces in which $\mathscr{B}$ is connected. If
\begin{itemize}[noitemsep,topsep=1pt]
\item[(i)] $\bpi_0^{\aone}(\Omega \mathscr{B}) = \bpi_0(\Laone \Omega \mathscr{B})$ is strongly $\aone$-invariant, and
\item[(ii)] $\hofib(f)$ is $n$-connected for some integer $n\geq 2$, then
\end{itemize}
the space $\hofib(\Laone(f))$ is $(n-1)$-connected as well. In particular, under hypothesis (i), if $f$ is a simplicial $n$-equivalence, then $f$ is also an $\aone$-$n$-equivalence.
\end{cor}

\begin{proof}
Since the unstable $\aone$-connectivity property holds for $S$, and since by hypothesis (ii) the space $\hofib(f)$ is assumed $(n-1)$-connected, we may apply Theorem \ref{thm:unstablenconnectivity} to conclude that $\Laone \hofib(f)$ is again $n$-connected. Again using the unstable $\aone$-connectivity property for $S$, the fact that $\mathscr{B}$ is connected, and hypothesis (i) we may apply Theorem \ref{thm:MField6.53} to conclude that the canonical map $\Laone \hofib(f) \to \hofib(\Laone(f))$ is a simplicial weak equivalence. Combining these observations, we conclude $\hofib(\Laone(f))$ is $n$-connected as well.
\end{proof}

\subsubsection*{$\aone$-localization of layers of Postnikov towers}
We can also deduce some stability properties for the layers of Postnikov towers under $\aone$-localization. To this end, assume $(\mathscr{X},x)$ is a pointed connected space.  If $\mathscr{X}^{(n)}$ is the $n$-th layer of the Postnikov tower for $\mathscr{X}$, then we write $\mathscr{X} \langle n \rangle$ for the space fitting into a simplicial fibration sequence of the form
\[
\mathscr{X}\langle n \rangle \longrightarrow \mathscr{X} \longrightarrow \mathscr{X}^{(n)}.
\]
The space $\mathscr{X}\langle n \rangle$ is the $n$-fold connective cover of $\mathscr{X}$, in particular it is $n$-connected.

\begin{lem}
\label{lem:connectivecoverofaonelocalizationstillaonelocal}
If $(\mathscr{X},x)$ is a pointed simplicially connected space, and the unstable $\aone$-connectivity property holds for our base $S$, then $(\Laone
\mathscr{X})\langle n \rangle$ is $\aone$-local.
\end{lem}

\begin{proof}
Assuming the unstable $\aone$-connectivity property holds for our base $S$, we conclude from Lemma \ref{lem:unstableconnectivityconsequences} that $(\Laone \mathscr{X})^{(n)}$ is $\aone$-local.  Moreover, $(\Laone \mathscr{X})^{(n)}$ is simplicially connected by Theorem \ref{thm:unstablenconnectivity} (though this does not require the unstable $\aone$-connectivity property).  It follows that, under these hypotheses, $(\Laone
\mathscr{X})\langle n \rangle$ is $\aone$-local since it is the simplicial homotopy fiber of the map $\Laone \mathscr{X} \to (\Laone \mathscr{X})^{(n)}$ which has a connected base.
\end{proof}

By functoriality of the Postnikov tower, there is an induced morphism $\mathscr{X} \langle n \rangle \to (\Laone \mathscr{X})\langle n \rangle$.   Assuming the unstable $\aone$-connectivity property holds for our base $S$, it follows from Lemma \ref{lem:connectivecoverofaonelocalizationstillaonelocal} that there is an induced morphism
\[
\Laone (\mathscr{X} \langle n \rangle) \longrightarrow (\Laone \mathscr{X})\langle n \rangle.
\]
Regarding this morphism, we have the following result, which is a variant of \cite[Theorem 6.59 and Corollary 6.60]{MField}.

\begin{thm}
\label{thm:MField6.60}
Assume the unstable $\aone$-connectivity property holds for $S$, and $\mathscr{X}$ is a pointed, connected space. Fix an integer $n \geq 1$. Suppose for each integer $i$ with $1 \leq i \leq n$ the sheaf $\bpi_i(\mathscr{X})$ is strongly $\aone$-invariant.
\begin{enumerate}[noitemsep,topsep=1pt]
\item The universal map $\bpi_i(\mathscr{X}) \to \bpi_i^{\aone}(\mathscr{X})$ is an isomorphism if $i \leq n$.
\item For each $i \leq n$, the morphism $\Laone (\mathscr{X} \langle i \rangle) \to (\Laone \mathscr{X})\langle i \rangle$ is a simplicial weak equivalence.
\item The universal map $\bpi_{n+1}(\mathscr{X}) \to \bpi_{n+1}^{\aone}(\mathscr{X})$ is the initial map from $\bpi_{n+1}(\mathscr{X})$ to a strictly $\aone$-invariant sheaf of groups.
\end{enumerate}
\end{thm}

\begin{proof}
Since $\bpi_1(\mathscr{X})$ is assumed strongly $\aone$-invariant, and for any integer $n \geq 1$ the map $\bpi_1(\mathscr{X}) \to \bpi_1(\mathscr{X}^{(n)})$ is an isomorphism, we conclude that $\bpi_1(\mathscr{X}^{(n)})$ is strongly $\aone$-invariant for any $n \geq 1$. By assumption, the unstable $\aone$-connectivity property holds for $S$. Thus, for $i \geq 2$ (or $i = 1$ if $\bpi_1^{\aone}(\mathscr{X})$ is abelian) we conclude that $\bpi_i(\mathscr{X})$ is strictly $\aone$-invariant.  We may also apply Theorem \ref{thm:MField6.53} to conclude that the sequence
\[
\Laone (\mathscr{X} \langle n \rangle) \longrightarrow \Laone \mathscr{X} \longrightarrow \Laone(\mathscr{X}^{(n)})
\]
is always a simplicial fiber sequence.

By Theorem \ref{thm:unstablenconnectivity}, we know that $\Laone (\mathscr{X} \langle n \rangle)$ is simplicially $n$-connected. Therefore, we conclude that
\begin{equation}
\label{eqn:piilaonexofn}
\begin{split}
\pi_i(\Laone \mathscr{X}) &\iso \bpi_i(\Laone(\mathscr{X}^{(n)})) \text{ if } i \leq n, \text{ and } \\
\bpi_{n+1}(\Laone \mathscr{X}) &\longrightarrow \bpi_{n+1}(\Laone(\mathscr{X}^{(n)})) \text{ is an epimorphism.}
\end{split}
\end{equation}
There are also simplicial fiber sequences of the form
\[
K(\bpi_i(\mathscr{X}),i) \longrightarrow \mathscr{X}^{(i)} \longrightarrow \mathscr{X}^{(i-1)}.
\]
Since $\mathscr{X}$ is connected, Point (2) of Lemma \ref{lem:unstableconnectivityconsequences} and the assumptions about the sheaves $\bpi_i(\mathscr{X})$ guarantee that $\mathscr{X}^{(n)}$ is $\aone$-local. Thus,
\begin{equation}
\label{eqn:piixofn}
\bpi_i(\mathscr{X}^{(n)}) \iso \bpi_i(\Laone \mathscr{X}^{(n)}) \text{ if } i \leq n.
\end{equation}
Now, we can put these facts together to prove the results.

For Point (1), notice that by combining the isomorphisms of (\ref{eqn:piilaonexofn}) and (\ref{eqn:piixofn}) we obtain for $i \leq n$ the following series of isomorphisms:
\[
\begin{split}
\bpi_i(\mathscr{X}) &\iso \bpi_i(\mathscr{X}^{(n)}) \\
&\iso \bpi_i(\Laone \mathscr{X}^{(n)}) \\
& \iso \bpi_i(\Laone \mathscr{X})\\
& \iso \bpi_i^{\aone}(\mathscr{X});
\end{split}
\]
this is precisely what we wanted to show.

For Point (2) we proceed as follows. From the isomorphisms established in Point (1), we conclude that the map $\mathscr{X}^{(i)} \to (\Laone \mathscr{X})^{(i)}$ is a simplicial weak equivalence. On the other hand, we already saw that $\mathscr{X}^{(i)}$ is $\aone$-local for $i \leq n$. Thus, the map $\Laone \mathscr{X}^{(i)} \to (\Laone \mathscr{X})^{(i)}$ is a simplicial weak equivalence for $i \leq n$. Since $(\Laone \mathscr{X})\langle i \rangle$ is by definition the simplicial homotopy fiber of $\Laone \mathscr{X} \to (\Laone \mathscr{X})^{(i)}$, it follows from the fiber sequence in the previous paragraph that the induced map $\Laone (\mathscr{X} \langle i \rangle) \to (\Laone \mathscr{X})\langle i \rangle$ is a simplicial weak equivalence for $i \leq n$.

Finally, for Point (3), begin by observing that if $\mathbf{A}$ is a strictly $\aone$-invariant sheaf, then since $\mathscr{X}\langle n \rangle$ is $n$-connected, obstruction theory (see \cite[Lemma B.7]{MField}) gives a bijection
\[
\hom(\bpi_{n+1}(\mathscr{X}),\mathbf{A}) \iso [\mathscr{X}\langle n \rangle,K(\mathbf{A},n+1)]_s.
\]
Since $K(\mathbf{A},n+1)$ is $\aone$-local, any map $\mathscr{X}\langle n \rangle \to K(\mathbf{A},n+1)$ factors through $\Laone (\mathscr{X}\langle n \rangle)$. However, since $\Laone (\mathscr{X} \langle n \rangle) \to (\Laone \mathscr{X})\langle n \rangle$ is a simplicial weak equivalence by Point (2), the result follows from the fact that $\bpi_{n+1}((\Laone \mathscr{X})\langle n \rangle) = \bpi_{n+1}^{\aone}(\mathscr{X})$.
\end{proof}

\subsection{James-style models for loop spaces in \texorpdfstring{$\aone$}{A1}-homotopy theory}
\label{ss:loopspaces}
In this section, we discuss the James model for loop spaces in $\aone$-homotopy theory. The construction involves comparing the James model and the Kan loop group model, as was the case in the setting of simplicial homotopy theory. If $(\mathscr{X},x)$ is a pointed simplicially connected space, then $\Omega \Laone \mathscr{X}$ is a model for the $\aone$-derived loop space of $\mathscr{X}$. The map $\Omega \mathscr{X} \to \Omega \Laone \mathscr{X}$ factors through a morphism
\[
\Laone \Omega \mathscr{X} \longrightarrow \Omega \Laone \mathscr{X},
\]
which need not be a simplicial weak equivalence. The following result, which is a variant of \cite[Theorem 6.46]{MField}, gives a necessary and sufficient condition for the above morphism to be a simplicial weak equivalence.

\begin{thm}
\label{thm:loopsaonelocal}
Assume the unstable $\aone$-connectivity property holds for $S$ and suppose $(\mathscr{X},x)$ is a pointed space. If $\bpi_0(\Laone \Omega \mathscr{X})$ is strongly $\aone$-invariant, then
\[
\Laone \Omega \mathscr{X} \longrightarrow \Omega \Laone \mathscr{X}
\]
is a simplicial weak equivalence.
\end{thm}

\begin{proof}
Since $\Omega \mathscr{X}$ only depends on the simplicial connected component of the base-point $x$, without loss of generality we can assume that $\mathscr{X}$ is simplicially connected.  In that case, the result follows immediately from Theorem \ref{thm:MField6.53} applied to the simplicial fiber sequence $\Omega \mathscr{X} \to \ast \to {\mathscr X}$.
\end{proof}

We now use the James construction in the category of simplicial presheaves (Proposition \ref{prop:J(X)=ho-nis=Omega_SigmaX}) together with the result just established about models for $\aone$-derived loop spaces to produce a James-style model for loops on the suspension in the $\aone$-homotopy category.

\begin{thm}
\label{thm:aonejamesconstruction}
Suppose the unstable $\aone$-connectivity property holds for $S$ and $f: \mathscr{X} \to \mathscr{Y}$ is a morphism of pointed simplicially connected spaces.
\begin{enumerate}[noitemsep,topsep=1pt]
\item There is a functorial simplicial weak equivalence $\Laone J(\mathscr{X}) \weq \Omega \Laone \Sigma \mathscr{X}$.
\item If $f$ is an $\aone$-weak equivalence, the map $J(f)$ is an $\aone$-weak equivalence.
\end{enumerate}
\end{thm}

\begin{proof}
By Proposition \ref{prop:J(X)=ho-nis=Omega_SigmaX}, there is a simplicial weak equivalence $J(\mathscr{X}) \weq \Omega \Sigma \mathscr{X}$. Thus, there is a simplicial weak equivalence
\[
\Laone J(\mathscr{X}) \weq \Laone \Omega \Sigma \mathscr{X}.
\]
Since $\mathscr{X}$ is connected, $\Sigma \mathscr{X}$ is necessarily $1$-connected (this follows by checking on stalks). Therefore, by Lemma \ref{lem:simpliciallysimplyconnectedguaranteeshypothesis}, we conclude that $\bpi_0(\Laone \Omega \Sigma \mathscr{X})$ is strongly $\aone$-invariant. Thus, we can apply Theorem \ref{thm:loopsaonelocal} to conclude that $\Laone \Omega \Sigma \mathscr{X} \weq \Omega \Laone \Sigma \mathscr{X}$.

For (2), it suffices to observe that if $f: \mathscr{X} \to \mathscr{Y}$ is an $\aone$-weak equivalence, then by \cite[\S 3 Lemma 2.13]{MV} the map $\Sigma \mathscr{X} \to \Sigma \mathscr{Y}$ is an $\aone$-weak equivalence. It follows immediately that the induced morphism $\Omega \Sigma \mathscr{X} \to \Omega \Sigma \mathscr{Y}$ is an $\aone$-weak equivalence. By Part (1), we conclude that $J(X) \to J(Y)$ is an $\aone$-weak equivalence.
\end{proof}

\section{The EHP sequence in \texorpdfstring{$\aone$}{A1}-homotopy theory}
\label{s:aoneehpsequence}
In this section, we study the analog in $\aone$-homotopy theory of Whitehead's EHP exact sequence from the introduction. We begin by recasting this exact sequence in the homotopy theory of simplicial sets (see Proposition \ref{prop:topoEHP}), and then explaining how to extend this result to simplicial presheaves on a site (see Proposition \ref{prop:simplicialsimplicialEHP}). For convenience, we will assume our site has enough points. In Section \ref{ss:aoneehp}, we construct a version of Whitehead's exact sequence in $\aone$-homotopy theory (see Theorem \ref{thm:EHP_range}). In Section \ref{ss:lowdegree}, we study the low-degree portion of the exact sequence of Theorem \ref{thm:EHP_range} and study very explicitly the first degree in which the suspension fails to be an isomorphism. The main result is Theorem \ref{thm:lowdegree}, which depends on various facts about $\aone$-homology.

\subsection{The EHP sequence in simplicial homotopy theory}
\label{ss:classicalEHP}
In this section, we recall Whitehead's refinement of the Freudenthal suspension theorem and adapt this result to the context of simplicial presheaves. This result appears as \cite[Chapter XII, Theorem 2.2]{Whitehead} and the main novelty of this section is that we give a different derivation of the exact sequence that we learned from Mike Hopkins; this version allows more precise control at the end of the sequence. The translation to the setting of simplicial presheaves is then straightforward.

\subsubsection*{The classical EHP sequence}
We begin by recalling the combinatorial construction of James--Hopf maps. We refer the reader to \cite[p. 169]{wu} for more details.

\begin{defn}
\label{defn:simplicialJamesHopf}
Suppose $K$ is a pointed simplicial set and $r \geq 1$ is an integer. Define a morphism of simplicial sets
\[
{\mathrm H}_r: J(K) \to J(K^{\sma r})
\]
that in each simplicial degree is given by the following formula:
\[
{\mathrm H}_r(x_1\ldots x_q) = \prod_{1 \leq i_1 < \cdots < i_r \leq q} x_{i_1} \wedge \cdots \wedge x_{i_r};
\]
the product on the right hand side is taken in (left to right) lexicographic order. We refer to ${\mathrm H}_r$ as a {\em simplicial James--Hopf invariant}.
\end{defn}

Note that ${\mathrm H}_r$ is, by definition, functorial in the input simplicial set $K$. Directly from the definition of ${\mathrm H}_r$ it follows that if $r \ge 2$, then the composite $K \overset{{\mathrm E}}{\to} J(K) \overset{{\mathrm H}_r}{\to} J(K^{\sma r})$ is trivial. We fix $r=2$, and write ${\mathrm H}$ for ${\mathrm H}_2$. There is a commutative diagram:
\begin{equation} \label{eq:referenceDiagram}
\xymatrix{ K \ar^{\mathrm E}[dr] \ar^\phi[d] \\ \hofib {\mathrm H} \ar[r] & J(K) \ar^{{\mathrm H}}[r] & J(K^{\sma 2}) }. \end{equation}

\begin{prop}
\label{prop:topoEHP}
 Suppose $K$ is $(n-1)$-connected where $n \ge 2$. Then $\phi: K \to \hofib {\mathrm H}$ is $(3n-2)$-connected. In particular, we
 obtain a long exact sequence of homotopy groups:
\begin{equation}
\label{eq:3}
\begin{split}
\xymatrix{
  \pi_{3n-2}(K) \ar@{->>}[d] & \pi_{3n-1}(\Sigma K) \ar@{-}^\iso[d] & \pi_{3n-1}(\Sigma (K^{\sma 2})) \ar@{-}^\iso[d] \\
  \pi_{3n-2}(\hofib {\mathrm H}) \ar[r] & \pi_{3n-2}(J(K)) \ar^-{{\mathrm H}}[r] & \pi_{3n-2}(J(K^{\sma 2})) \ar^-{{\mathrm P}}[r] & \pi_{3n-3}(K) \ar^-{{\mathrm E}}[r] & \cdots
} \\
\xymatrix{
 \cdots \ar[r] & \pi_q(K) \ar^-{{\mathrm E}}[r] & \pi_q(J(K)) \ar@{-}^\iso[d] \ar^-{{\mathrm H}}[r] & \pi_q(J(K^{\sma 2})) \ar@{-}^\iso[d]\ar^-{{\mathrm P}}[r] &
  \pi_{q-1}(K) \ar[r] & \dots
 \\ & & \pi_{q+1}(\Sigma K) & \pi_{q+1}(\Sigma (K^{\sma 2})).}
\end{split}
\end{equation}
\end{prop}

\begin{proof}
Since $K$ is $(n-1)$-connected, we conclude that $K^{\sma 2}$ is $(2n-1)$-connected. Therefore, $J(K) \weq \Omega \Sigma K$ is $(n-1)$-connected, and $J(K^{\sma 2}) \weq \Omega \Sigma K^{\sma 2}$ is $(2n-1)$-connected.

We consider the Serre spectral sequence in homology $\Hoh_*(\cdot, \ZZ)$ associated with the simplicial fiber sequence
\[
\hofib {\mathrm H} \longrightarrow J(K) \longrightarrow J(K^{\sma 2}).
\]
Since $n \geq 1$ by assumption, $J(K^{\sma 2})$ is simply connected.

By use of the Hilton--Milnor splitting \cite[Chapter VII, Theorem 2.10]{Whitehead} there are isomorphisms
\[
\tilde\Hoh_*(J(K), \ZZ) \iso \bigoplus_{i=1}^\infty \tilde \Hoh_*( K^{\sma i}, \ZZ), \quad \tilde \Hoh_*(J(K^{\sma 2}), \ZZ) \iso
 \bigoplus_{i=1}^\infty \tilde \Hoh_*( K^{\sma 2i}, \ZZ).
\]
 We remark in passing that the map ${\mathrm E}:K \to J(K)$ induces an isomorphism of $\tilde \Hoh_*(K, \ZZ)$ with the first
 summand of $\tilde \Hoh_*(J(K), \ZZ) \iso \bigoplus_{i=1}^\infty \tilde \Hoh_*( K^{\sma i}, \ZZ)$---this appears in the proof
 of \cite[Chapter VII, Theorem 2.10]{Whitehead}.

In the range where $p+q < 3n$, the $\Eoh^2$-page of the spectral sequence takes a particularly simple form. In this
 range $\Eoh^2_{0,q} = \Hoh_q(\hofib H, \ZZ)$ and $\Eoh^2_{p,0} = \Hoh_p(J(K^{\sma 2}), \ZZ) = \Hoh_p(K^{\sma 2} , \ZZ)$,
 and all other groups are necessarily $0$.

From \cite[Theorem 6.2]{kuhn87}, we know that the composite map
\[
\Hoh_*(K^{\sma 2}, \ZZ) \longrightarrow \Hoh_*(J(K), \ZZ) \overset{{\mathrm H}}{\longrightarrow} \Hoh_*(J(K^{\sma 2}), \ZZ) \longrightarrow
 \Hoh_*(K^{\sma 2}, \ZZ)
\]
is the identity. This observation implies that there are no nonzero differentials in our spectral sequence having source
 $\Eoh_{p,0}^*$ with $p < 4n$. For degree reasons, therefore, there can be no nonzero differentials the targets of which
 are the
 groups $\Eoh_{0,q}^*$ with $q < 3n-1$ either, and in the range where $p+q \le 3n-2$, the sequence collapses at the
 $\Eoh^2$ page. We obtain $\tilde \Hoh_{ \le 3n-2}(J(K), \ZZ) = \tilde \Hoh_{\le 3n-2}( K^{\sma 2}, \ZZ) \oplus \tilde
 \Hoh_{\le 3n-2}( \hofib {\mathrm H}, \ZZ)$.

In the given range, therefore, we have a commutative diagram of homology groups
 (with $\ZZ$ coefficients)
 \[ \xymatrix{ 0 \ar[r] & \tilde \Hoh_{\le 3n-2}(\hofib \mathrm{H}) \ar[r] & \tilde \Hoh_{\le 3n-2} (J(K)) \ar[r] & \tilde
  \Hoh_{\le 3n-2}(J(K^{\sma
   2})) \ar[r] & 0 \\
  0 \ar[r] &\tilde \Hoh_{\le 3n-2}(K) \ar[r] \ar^{\phi_*}[u]& \tilde \Hoh_{\le 3n-2}(K) \oplus \tilde \Hoh_{\le
   3n-2}(K^{\sma 2}) \ar[r] \ar@{=}[u]& \tilde \Hoh_{\le 3n-2}(K^{\sma 2}) \ar[r] \ar@{=}[u]& 0} \] from which it
 follows that the map $\phi_*$ is a homology isomorphism in the stated range. In particular, the map $\phi$ is $3n-2$
 connected. The long exact sequence \eqref{eq:3} now follows from the long exact sequence in homotopy associated with the simplicial
 fiber sequence
 \[ \hofib {\mathrm H} \longrightarrow J(K) \longrightarrow J(K^{\sma 2}). \]
\end{proof}

\subsubsection*{The EHP sequence for simplicial presheaves}
Using the results of the previous section, we can generalize Proposition \ref{prop:topoEHP} to the situation of pointed simplicial presheaves on a site $\mathbf{C}$ equipped with a local model structure; for simplicity, we assume that $\mathbf{C}$ has enough points. Functoriality of the simplicial James--Hopf invariants allows Definition \ref{defn:simplicialJamesHopf} to be extended to simplicial presheaves.

\begin{defn}
\label{defn:simplicialJamesHopfpresheaf}
If $X$ is a pointed simplicial presheaf on $\mathbf{C}$, define morphisms
\[
{\mathrm H}_r: J(X) \longrightarrow J(X^{\sma r})
\]
by ${\mathrm H}_r: J(X)(U) \longrightarrow J(X^{\sma r})(U)$. Set ${\mathrm H} := {\mathrm H}_2$.
\end{defn}

As before, the composite map $X \overset{{\mathrm E}}{\to} J(X) \overset{{\mathrm H}_r}{\to} J(X^{\sma r})$ is null. The next result extends Proposition \ref{prop:topoEHP} to simplicial presheaves.

\begin{prop}
\label{prop:simplicialsimplicialEHP}
Suppose $\mathbf{C}$ is a site that has enough points. Suppose, $n \geq 1$ is an integer, and $X$ is a pointed $(n-1)$-connected simplicial presheaf.  Let ${\mathrm E}: X \to J(X)$ be as in \textup{(\ref{eqn:Emap})}, ${\mathrm H}: J(X) \to J(X^{\sma 2})$ as in \textup{Definition \ref{defn:simplicialJamesHopfpresheaf}}, and let $\phi$ be a lift of the map ${\mathrm E} : X \to J(X)$ to a map $\phi: X \to \hofib {\mathrm H}$. The map $\phi$ is $(3n-2)$-connected and there is a long exact sequence of homotopy sheaves of the form:
 \begin{equation}
 \label{eq:4}
\begin{split} \xymatrix{
\bpi_{3n-2}(X) \ar[r]^-{{\mathrm E}} & \bpi_{3n-2}(J(X)) \ar^-{{\mathrm H}}[r] & \bpi_{3n-2}(J(X^{\sma 2})) \ar^-{\mathrm P}[r] & \bpi_{3n-3}(X) \ar^-{\mathrm E}[r] & \dots } \\
\xymatrix{
 \dots \ar[r] & \bpi_q(X) \ar^-{\mathrm E}[r] & \bpi_q(J(X))\ar^-{\mathrm H}[r] & \bpi_q(J(X^{\sma 2})) \ar^-{{\mathrm P}}[r] &
  \bpi_{q-1}(X) \ar[r] & \dots.}
 \end{split}
\end{equation}
\end{prop}

\begin{rem}
Proposition \ref{prop:J(X)=ho-nis=Omega_SigmaX} guarantees the existence of isomorphisms of homotopy sheaves of the form $\bpi_q(J(X))
\iso \bpi_{q+1}(X)$ and $\bpi_q(J(X^{\sma 2})) \iso \bpi_{q+1}(X^{\sma 2})$.
\end{rem}

\begin{proof}
In outline, we argue at points to reduce to the classical EHP sequence. In more detail, let $F$ denote the homotopy fiber of the map ${\mathrm H}: J(X) \to J(X^{\wedge 2})$ in the local model structure. Since the composite ${\mathrm H} \circ {\mathrm E}: X \to J(X)
 \to J(X^{\wedge 2})$ is null, there is a lift of ${\mathrm E}: X \to J(X)$ to a map $\phi: X \to F$ as follows:
 \begin{equation}
 \label{eq:1}
 \xymatrix{ X \ar_-\phi@{-->}[d] \ar^{\mathrm E}[dr] \\ F \ar[r] & J(X) \ar^-{\mathrm H}[r] & J(X^{\wedge 2}). }
 \end{equation}
 If $q^*$ is a point of the site $\mathbf{C}$, then $q^*$ preserves fiber sequences, and commutes with the formation of
 $J(\cdot)$ and ${\mathrm E}$, ${\mathrm H}$.
 In particular, applying $q^*$ throughout, we see using Proposition \ref{prop:topoEHP} that $q^*\phi$ is $(3n-2)$ connected.
Since this holds for all such $q^*$, we deduce that the map $\phi$ is itself $(3n-2)$ connected. The long exact sequence follows.
\end{proof}

\subsection{The construction of the EHP sequence in \texorpdfstring{$\aone$}{A1}-homotopy theory}
\label{ss:aoneehp}
We now transport the EHP sequence studied in the previous section to $\aone$-homotopy theory. The basic idea is to appeal to Proposition \ref{prop:simplicialsimplicialEHP} and use facts about when $\aone$-localization preserves simplicial fiber sequences from Section \ref{ss:furtherproperties}. If we $\aone$-localize the simplicial James--Hopf map ${\mathrm H}$ of Definition \ref{defn:simplicialJamesHopfpresheaf} (we abuse notation and write ${\mathrm H}$ for the resulting map), then we can consider the following sequence of morphisms
\begin{equation}
\label{eq:A1SJH}
\Laone {\mathscr X} {\longrightarrow} \Laone J({\mathscr X}) \stackrel{{\mathrm H}}{\longrightarrow} \Laone J({\mathscr X}^{\sma 2}).
\end{equation}
The next result gives an analog of Whitehead's classical exact sequence in $\aone$-homotopy theory.

\begin{thm}
\label{thm:EHP_range}
Assume the unstable $\aone$-connectivity property holds for $S$ and suppose $\mathscr X$ is a pointed $\aone$-$(n-1)$-connected space, with $n \ge 2$. There is an exact sequence of homotopy sheaves of the form:
\begin{equation}
\label{eq:5}
\begin{split}
\xymatrix{\piaone_{3n-2}(\mathscr X) \ar^-{{\mathrm E}}[r] &
\piaone_{3n-2}(J(\mathscr{X})) \ar^-{{\mathrm H}}[r] & \piaone_{3n-2}(J(\mathscr{X}^{\sma 2})) \ar^-{\mathrm P}[r] & \piaone_{3n-3}(\mathscr{X}) \ar^-{\mathrm E}[r] & \dots
} \\
\xymatrix{
 \dots \ar[r] & \piaone_q(\mathscr{X}) \ar^-{\mathrm E}[r] & \piaone_q(J(\mathscr{X}))\ar^-{\mathrm H}[r] & \piaone_q(J(\mathscr{X}^{\sma 2})) \ar^-{{\mathrm P}}[r] &
  \piaone_{q-1}(\mathscr{X}) \ar[r] & \dots.}
 \end{split}
\end{equation}
\end{thm}

\begin{rem}
Theorem \ref{thm:aonejamesconstruction} guarantees the existence of isomorphisms of sheaves $\piaone_q(J(\mathscr{X})) \iso \piaone_{q+1}(\Sigma \mathscr{X})$ and $\bpi_q(J(\mathscr{X}^{\sma 2})) \iso \piaone_{q+1}( \Sigma \mathscr{X}^{\sma 2})$.
\end{rem}

\begin{proof}
  We proceed as in the proof of Proposition \ref{prop:simplicialsimplicialEHP}, with the part of $\mathscr{X}$ played by
  $\Laone \mathscr{X}$. By hypothesis, $\Laone \mathscr{X}$ is simplicially $(n-1)$ connected. As before, set up the diagram
\begin{equation}
\label{eq:2}
\xymatrix{
\Laone \mathscr{X} \ar^\phi@{-->}[d] \ar^{\mathrm E}[dr] \\
\mathscr{F} \ar[r] & J(\Laone \mathscr X) \ar^-{\mathrm H}[r] & J((\Laone \mathscr X)^{\wedge 2}).
}
\end{equation}

Since $n \geq 1$, the space $J((\Laone \mathscr{X})^{\wedge 2})$ is simplicially $1$-connected. Then, using the unstable $\aone$-connectivity
property, we may apply Theorem \ref{thm:MField6.53} to conclude that applying $\Laone$ to the simplicial fiber sequence in \eqref{eq:1}
results in a simplicial fiber sequence of the form:
\begin{equation}
\label{eq:20}
\xymatrix{
\Laone \mathscr{F} \ar[r] & \Laone J(\Laone \mathscr{X}) \ar[r] & \Laone J((\Laone \mathscr{X})^{\wedge 2})
}.
\end{equation}
Since the map $\mathscr{X}\to \Laone \mathscr{X}$ is an $\aone$-weak equivalence, we conclude from Theorem \ref{thm:aonejamesconstruction}(2) that there are weak equivalences of the form $\Laone J( \Laone \mathscr{X}) \weq \Laone J(\mathscr{X})$ and $\Laone J((\Laone \mathscr{X})^{\wedge 2}) \weq \Laone J(\mathscr{X}^{\wedge 2})$.

Since the unstable $\aone$-connectivity property holds for $S$, the sheaves $\piaone_i(\mathscr{X})$ are strictly $\aone$-invariant by Lemma \ref{lem:unstableconnectivityconsequences}(1). Then, Theorem \ref{thm:MField6.60} implies that
\[
\bpi_i(\mathscr{F}) \iso \piaone_i(\mathscr{F}) \iso \bpi_i(\Laone \mathscr{F}) \iso \bpi_i(\Laone \mathscr X) =
 \piaone_i(\mathscr X) \quad 1 \le i \le 3n-3.
\]
These observations suffice to establish exactness everywhere except the leftmost part of the long exact sequence.

The map $\phi : \Laone \mathscr X \to \mathscr{F}$ is simplicially $(3n-2)$-connected and, since $n \geq 2$, the connectivity of $\mathscr X$ implies that $\bpi_0^{\aone}(\Omega \mathscr F) \simeq \ast$, by means of Theorem \ref{thm:unstablenconnectivity} for example. Thus, we can apply Corollary \ref{cor:relativehurewicz} to conclude that
$\Laone \phi: \Laone \mathscr X \to \Laone \mathscr{F}$ is also simplicially $(3n-2)$-connected. Therefore, there is a surjective map
$\phi_* : \piaone_{3n-2}(\mathscr X) \twoheadrightarrow \piaone_{3n-2}(\mathscr{F})$ factoring
${\mathrm E} : \piaone_{3n-2}(\mathscr X) \twoheadrightarrow \piaone_{3n-2}(J(\mathscr X))$, yielding the exactness of the long exact
sequence at the left as well.

\end{proof}

\begin{rem}
\label{rem:refinesfreudenthal}
Assume the unstable $\aone$-connectivity property holds for $S$. If $\mathscr{X}$ is a simplicially $(n-1)$-connected space, then
$J(\Laone \mathscr{X}^{\sma 2})$ is at least $\aone$-$(2n-1)$-connected by Theorem \ref{thm:unstablenconnectivity}. Theorem
\ref{thm:EHP_range} is therefore a refinement of Morel's suspension theorem, \cite[Theorem 6.61]{MField}.
\end{rem}

\subsection{Analyzing the \texorpdfstring{$\aone$}{A1}-EHP sequence in low degrees}
\label{ss:lowdegree}
The goal of this section is to study the low-degree portion of the EHP sequence in $\aone$-algebraic topology. To do this, given an $\aone$-$(n-1)$-connected space $\mathscr{X}$, we will show that $\mathscr{X}^{\sma 2}$ is at least $\aone$-$(2n-1)$-connected, identify the first non-vanishing $\aone$-homotopy sheaf of $\mathscr{X}^{\sma 2}$ and use this to give a more explicit form of the EHP sequence in the first degree in which the suspension map is not an isomorphism. Granted the results of previous sections, and some results about $\aone$-homology recalled below, the argument is a straightforward translation of a classical argument due to J.H.C. Whitehead \cite[Theorem 2]{Whiteheadrelations} in the case of spheres and more generally by P. Hilton \cite[Theorem 2.1]{Hiltonwedge}.

\subsubsection*{Some connectivity estimates}
Suppose $(\mathscr{X},x)$ and $(\mathscr{Y},y)$ are two pointed spaces. We will assume that $\mathscr{X}$ is $\aone$-$(m-1)$-connected, and $\mathscr{Y}$ is $\aone$-$(n-1)$-connected. Without loss of generality, we will assume that $m \leq n$.

\begin{lem}
\label{lem:connectivityofsmash}
Assume the unstable $\aone$-connectivity property holds over $S$. The wedge sum $\mathscr{X} \vee \mathscr{Y}$ is at least $\aone$-$(m-1)$ connected, and the smash product $\mathscr{X} \sma \mathscr{Y}$ is at least $\aone$-$(m+n-1)$-connected.
\end{lem}

\begin{proof}
For the first statement, observe that the map $\mathscr{X} \vee \mathscr{Y} \to L_{\aone} \mathscr{X} \vee L_{\aone} \mathscr{Y}$ is an $\aone$-weak equivalence by \cite[\S 2 Lemma 2.11]{MV}. Since taking stalks commutes with coproducts, we conclude that the stalks of $L_{\aone} \mathscr{X} \vee L_{\aone} \mathscr{Y}$ are at least $(m-1)$-connected. Under the hypotheses Theorem \ref{thm:unstablenconnectivity} implies that $\mathscr{X} \vee \mathscr{Y}$ is at least $(m-1)$-connected.

The second statement is established similarly. By two applications of \cite[\S 3 Lemma 2.13]{MV} we can conclude that the map $\mathscr{X} \sma \mathscr{Y} \to L_{\aone} \mathscr{X} \sma L_{\aone} \mathscr{Y}$ is an $\aone$-weak equivalence. Again by checking on stalks, and using the unstable $\aone$-connectivity theorem one concludes that $\mathscr{X} \sma \mathscr{Y}$ is at least $\aone$-$(m+n-1)$-connected.
\end{proof}

\subsubsection*{\texorpdfstring{$\aone$}{A1}-homology}
The $\aone$-derived category may be constructed as a left Bousfield localization of the derived category of presheaves of abelian groups on $\Sm_S$ with respect to a notion of $\aone$-quasi-isomorphism \cite[\S 6.2]{MField}.\footnote{Morel works with the derived category of Nisnevich sheaves of abelian groups, but the exact functor of Nisnevich sheafification induces a Quillen equivalence between the model we use and Morel's model.}  Morel gives a construction of an $\aone$-localization functor $\Laone^{ab}$ \cite[Lemma 6.18]{MField}; this functor is an endofunctor of the category of chain complexes of Nisnevich sheaves of abelian groups, and there is a natural transformation $\theta: \Id \to \Laone^{ab}$ such that for any complex $A$, there is a quasi-isomorphis $A \to \Laone^{ab}(A)$ with target that is fibrant and $\aone$-local.

\begin{notation}
\label{notation:aonehomology}
If $\mathscr{X}$ is a space, then we consider $C_*(\mathscr{X})$, the normalized chain complex associated with the simplicial presheaf of free abelian groups $\Z(\mathscr{X})$.  The $\aone$-singular chain complex of $\mathscr{X}$ is the complex $\Laone^{ab}C_*(\mathscr{X})$, which we may also denote $C_*^{\aone}(\mathscr{X})$.  The structure morphism $\mathscr{X} \to S$ induces a morphism $C_*^{\aone}(\mathscr{X}) \to C_*^{\aone}(S)$, and we define $\widetilde C^{\aone}_*(\mathscr{X})$ as the kernel of this morphism.

The $\aone$-homology sheaves of $\mathscr{X}$, denoted
$\H_i^{\aone}(\mathscr{X})$, are defined as the Nisnevich sheafifications of the homology presheaves
$H_i(\Laone^{ab}C_*(\mathscr{X}))$. If $A$ is a complex of presheaves of abelian groups, we will abuse notation and
define $\H_i^{\aone}(A)$ to be the Nisnevich sheafification of the homology presheaf $H_i(\Laone^{ab}A)$. We define
$\widetilde{\H}_i^{\aone}(\mathscr{X})$ as $ \ker(\H_i^{\aone}(\mathscr{X}) \to \H_i^{\aone}(S)$
\end{notation}

The Dold--Kan adjunction shows that the Eilenberg--MacLane space associated with an $\aone$-local complex is an $\aone$-local space \cite[Proposition 4 and (3.5)]{DeligneVoevodsky}. Note, however, that the ordinary singular chain complex of an $\aone$-local space can {\em fail} to be $\aone$-local (the standard counterexample is $\gm{}$). The following property is the analog of the unstable $\aone$-connectivity property of Definition \ref{defn:unstableaoneconnectivity} and was studied in \cite[\S 6.2]{MStable} in the closely related context of $S^1$-spectra.

\begin{defn}
\label{defn:stableaoneconnectivity}
The {\em stable $\aone$-connectivity property holds for $S$} if $\Laone^{ab}$ preserves $(-1)$-connected complexes.
\end{defn}

\begin{thm}[{\cite[Theorem 6.1.8]{MStable}}]
\label{thm:stableaoneconnectivitypropertyholds}
The stable $\aone$-connectivity property holds for the spectrum of a field.\footnote{Once more, the careful reader should assume that $k$ is infinite for the same reason as in the footnote to Theorem \ref{thm:unstableaoneconnectivityperfectfields}.}
\end{thm}

\begin{rem}
\label{rem:stableaoneconnectivityproperty}
Ayoub has shown that if $S$ is a Noetherian scheme of Krull dimension $d \geq 2$, then the stable $\aone$-connectivity
property may fail for $S$ in a very strong sense, \cite{Ayoub}. Ayoub's counter-example is constructed in
Voevodsky's derived category of motives. As noted above, if a complex of sheaves of abelian groups is $\aone$-local,
then the associated Eilenberg--MacLane space is $\aone$-local as well. Therefore, Ayoub's counterexample can be
transported to yield a counter-example to the unstable $\aone$-connectivity property over $S$.
\end{rem}

If the stable $\aone$-connectivity property holds, the $\aone$-derived category has a number of very nice
properties. Write $\Ab_S$ for the category of Nisnevich sheaves of abelian groups on $\Sm_S$, and $\Ab^{\aone}_S$ for
the full subcategory of strictly $\aone$-invariant sheaves. Before proceeding, we introduce the following notation.

\begin{notation}
\label{notation:aonetensorproduct}
Given two strictly $\aone$-invariant sheaves, set
\[
\mathbf{A} \tensor^{\aone} \mathbf{B} := \H_0^{\aone}(\mathbf{A} \tensor^{\mathrm{L}} \mathbf{B}).
\]
\end{notation}

\begin{rem}
The unit object for the $\aone$-tensor product is the strictly $\aone$-invariant sheaf $\Z$.
\end{rem}

With this notation, the following result holds.

\begin{lem}[{\cite[Lemma 6.2.13]{MStable}}]
\label{lem:stableaoneconnectivitystrictlyaoneinvariantabeliancategory}
If the stable $\aone$-connectivity property holds over $S$, then $\Ab^{\aone}_S$ is an abelian category and the inclusion functor $\Ab^{\aone}_S \to \Ab_S$ is an exact embedding. Moreover, the bifunctor $(\mathbf{A},\mathbf{B}) \mapsto \mathbf{A} \tensor^{\aone} \mathbf{B}$ equips the category $\Ab^{\aone}_S$ with a symmetric monoidal structure.
\end{lem}

The next result is closely related to \cite[Remark 6.2.6]{MStable} (apply that remark to shifted suspension spectra of suitably highly connected pointed spaces).

\begin{prop}
\label{prop:aonehomologyofproduct}
Assume the unstable and stable $\aone$-connectivity properties holds for $S$, and suppose $m,n$ are integers $\geq 1$. If $\mathscr{X}$ is $\aone$-$(m-1)$-connected, and $\mathscr{Y}$ is $\aone$-$(n-1)$-connected, there are canonical isomorphisms
\[
\widetilde{\H}_i^{\aone}(\mathscr{X} \times \mathscr{Y}) \isomto
\begin{cases} \widetilde{\H}_i^{\aone}(\mathscr{X})  \oplus \widetilde{\H}_i^{\aone}(\mathscr{Y}) & \text{ if } 0 \leq i \leq m+n-1, \\
  (\widetilde{\H}_m^{\aone}(\mathscr{X}) \tensor^{\aone} \widetilde{\H}_n^{\aone}(\mathscr{Y})) \oplus
  \widetilde{\H}_{m+n}^{\aone}(\mathscr{X}) \oplus \widetilde{\H}_{m+n}^{\aone}(\mathscr{Y}) & \text{ if } i =
  m+n. \end{cases}
\]
\end{prop}

\begin{proof}
Consider the inclusion map
\[
\mathscr{X} \vee \mathscr{Y} \longrightarrow \mathscr{X} \times \mathscr{Y}.
\]
The cone of this inclusion map is $\mathscr{X} \sma \mathscr{Y}$. Note also that after a single suspension, the inclusion map is split by the projection. As a consequence, there are direct sum decompositions of the form
\[
\widetilde{\H}_{i}^{\aone}(\mathscr{X} \times \mathscr{Y}) \iso \widetilde{\H}_i^{\aone}(\mathscr{X}) \oplus \widetilde{\H}_i^{\aone}(\mathscr{Y}) \oplus \widetilde{\H}_{i}^{\aone}(\mathscr{X} \wedge \mathscr{Y}).
\]
Under the assumption that the unstable and stable $\aone$-connectivity theorems hold for $S$, Lemma \ref{lem:connectivityofsmash} together with the usual $\aone$-Hurewicz theorem \cite[Theorem 6.57]{MField} immediately imply the result for $i \leq m+n-1$ (note: Morel's Hurewicz theorem holds in this context under the assumptions given: simply replace the appeal to \cite[Theorem 6.56]{MField} in Morel's proof by an appeal to Corollary \ref{cor:relativehurewicz}).

It remains to treat the case $i = m+n$. In that case, let $\widetilde{C}_*^{\aone}(\mathscr{X})$ and
$\widetilde{C}_*^{\aone}(\mathscr{Y})$ be the $\aone$-chain complexes of $\mathscr{X}$ and $\mathscr{Y}$; recall that
these are obtained by taking the chain complex associated with the free abelian group on $\mathscr{X}$ and then
$\aone$-localizing the result. By replacing $\widetilde{C}_*^{\aone}(\mathscr{X})$ by a shift, we may assume that $n=0$
and similarly for $\mathscr{Y}$ and $m=0$.  The complex $\widetilde{C}_*^{\aone}(\mathscr{X}) \tensor^{\mathrm{L}} \widetilde{C}_*^{\aone}(\mathscr{Y})$ is also concentrated in degrees $\geq 0$, and since the stable $\aone$-connectivity property holds for $S$ it follows that $\Laone^{ab}(\widetilde{C}_*^{\aone}(\mathscr{X}) \tensor^{\mathrm{L}} \widetilde{C}_*^{\aone}(\mathscr{Y}))$ is concentrated in degrees $\geq 0$. Since the zeroth homology is obtained by truncation with respect to the homotopy $t$-structure, it follows that
\[
\H_0^{\aone}(\widetilde{C}_*^{\aone}(\mathscr{X}) \tensor^{\mathrm{L}} \widetilde{C}_*^{\aone}(\mathscr{Y})) \iso \H_0^{\aone}(\mathscr{X}) \tensor^{\aone} \H_0^{\aone}(\mathscr{Y}).
\]
To conclude it remains to identify the left hand side in terms of the smash product $\mathscr{X} \sma \mathscr{Y}$.  For this, it suffices to observe that the Eilenberg-Zilber theorem implies the existence of an isomorphism of the form:
$\widetilde{C}_*^{\aone}(\mathscr{X}) \tensor^{\mathrm{L}} \widetilde{C}_*^{\aone}(\mathscr{Y})
\iso \widetilde{C}_*^{\aone}(\mathscr{X} \sma \mathscr{Y})$.
\end{proof}

\begin{cor}
\label{cor:homologyofsmashproducts}
Assume the unstable and stable $\aone$-connectivity properties hold for $S$ and suppose $m,n$ are integers $\geq 2$. If $(\mathscr{X},x)$ is a pointed $\aone$-$(m-1)$-connected space and $(\mathscr{Y},y)$ is a pointed $\aone$-$(n-1)$-connected space, then there is a canonical isomorphism
\[
\bpi_{n+m}^{\aone}(\mathscr{X} \wedge \mathscr{Y}) \isomto \bpi_m^{\aone}(\mathscr{X}) \tensor^{\aone} \bpi_n^{\aone}(\mathscr{Y}).
\]
\end{cor}

\begin{proof}
By Lemma \ref{lem:connectivityofsmash}, we know $\mathscr{X} \sma \mathscr{Y}$ is $\aone$-$(m+n-1)$-connected. By the $\aone$-Hurewicz theorem \cite[Theorem 6.57]{MField}, it suffices to prove the result in $\aone$-homology (again, as in the proof of Proposition \ref{prop:aonehomologyofproduct}, this holds under our assumption that the unstable and stable $\aone$-connectivity properties hold). In that case, it follows immediately from the proof of Proposition \ref{prop:aonehomologyofproduct}.
\end{proof}

The connection with the results of P. Hilton and J.H.C. Whitehead mentioned at the beginning of this section is contained in the next result, which describes the first ``non-linear" $\aone$-homotopy sheaf of a wedge sum. Given the above results, the proof is a direct consequence of the $\aone$-homotopy excision theorem (a.k.a. Blakers-Massey theorem\footnote{By inspecting the proof, one sees that this result is a direct consequence of the relative connectivity theorem (Corollary \ref{cor:relativehurewicz}) and therefore holds over any base scheme $S$ for which the unstable $\aone$-connectivity property holds.} see, e.g., \cite[Theorem 3.1]{AsokFaselComparison}, \cite[Theorem 2.3.8]{Strunk} or \cite[Proposition 2.20]{WickelgrenWilliams}) and is left to the reader.

\begin{cor}
Assume the unstable and stable $\aone$-connectivity properties hold for $S$ and suppose $m,n \geq 2$ are integers. Suppose $\mathscr{X}$ is a pointed $\aone$-$(m-1)$-connected space and $\mathscr{Y}$ is a pointed $\aone$-$(n-1)$-connected space. There are canonical isomorphisms
\[
\bpi_{i}^{\aone}(\mathscr{X} \vee \mathscr{Y}) \isomto \begin{cases} \bpi_{i}^{\aone}(\mathscr{X}) \oplus \bpi_i^{\aone}(\mathscr{X}) & \text{ if } 1 \leq i \leq m+n-2 \\ \bpi_{m+n-1}^{\aone}(\mathscr{X}) \oplus \bpi_{m+n-1}^{\aone}(\mathscr{X}) \oplus \bpi_m^{\aone}(\mathscr{X}) \tensor^{\aone} \bpi_{n}^{\aone}(\mathscr{Y}) &\text{ when $i=m+n-1$}.\end{cases}
\]
\end{cor}

\begin{rem}
As in classical homotopy theory, the computation of the homotopy of a wedge sum allows one to study homotopy operations. The first ``non-linear" summand in the homotopy of a wedge sum is closely related to the Whitehead product studied in Section \ref{ss:whiteheadproduct}, though we have not attempted to establish equivalence of the definitions.
\end{rem}

\subsubsection*{The EHP sequence in low-degree}
\begin{thm}
\label{thm:lowdegree}
Assume the unstable and stable $\aone$-connectivity properties hold for $S$. Let $n \ge 2$ be an integer. If $\mathscr{X}$ is an $\aone$-$(n-1)$-connected space, then there is an exact sequence of the form
\[
\bpi_{2n+1}^{\aone}(\Sigma \mathscr{X}) \stackrel{{\mathrm H}}{\longrightarrow} \bpi_{n}^{\aone}(\mathscr{X}) \tensor^{\aone} \bpi_{n}^{\aone}(\mathscr{X}) \stackrel{{\mathrm P}}{\longrightarrow} \bpi_{2n-1}^{\aone}(\mathscr{X}) \stackrel{{\mathrm E}}{\longrightarrow} \bpi_{2n}^{\aone}(\Sigma \mathscr{X}) \longrightarrow 0
\]
In particular, one has an exact sequence as above if $S$ is the spectrum of an (infinite) perfect field.
\end{thm}

\begin{proof}
Consider the exact sequence of Theorem \ref{thm:EHP_range}. Lemma \ref{lem:connectivityofsmash} implies that $\Sigma \mathscr{X} \sma \mathscr{X}$ is at least $\aone$-$2n$-connected, and thus $J(\mathscr{X} \sma \mathscr{X})$ is at least $\aone$-$(2n-1)$-connected. This immediately yields the surjectivity in the statement. Corollary \ref{cor:homologyofsmashproducts} then yields the identification of $\bpi_{2n}^{\aone}(J(\mathscr{X} \sma \mathscr{X}))$ with the $\aone$-tensor product term. The final statement is a consequence of Theorems \ref{thm:unstableaoneconnectivityperfectfields} and \ref{thm:stableaoneconnectivitypropertyholds}.
\end{proof}

\begin{rem}
The exact sequence of Theorem \ref{thm:lowdegree} in case $\mathscr{X} = {\mathbb A}^3 \setminus 0$ is precisely the one described in \cite[Theorem 4]{AsokFaselOberwolfach}. One notational benefit of the statement of Theorem \ref{thm:lowdegree} is that the quadratic nature of the James--Hopf invariants is apparent.
\end{rem}

\section{Some \texorpdfstring{$E_1$}{E1} differentials in the \texorpdfstring{$\aone$}{A1}-EHP sequence}
\label{s:eonedifferential}
The goal of this section is to analyze the morphisms in the $\aone$-EHP sequence.  As mentioned in the introduction, classically, the morphism ${\mathrm P}$ can be described in terms of Whitehead products. In Section \ref{ss:whiteheadproduct}, we extend the definition of Whitehead product to the theory of simplicial presheaves (see Definition \ref{defn:whiteheadproduct}) to make it available in the $\aone$-homotopy category as well. These results are written in the generality of simplicial presheaves on a site with enough points. We then use the results of Section \ref{ss:whiteheadproduct} to show that the map ${\mathrm P}$ in Theorem \ref{thm:EHP_range} can indeed be expressed in terms of the Whitehead product (see Theorem \ref{P=piWhitehead}); this requires that the unstable $\aone$-connectivity property holds for $S$.

The $\aone$-EHP spectral sequence is created by combining $\aone$-EHP sequences into an exact couple.  However, since the $\aone$-EHP sequences of Theorem \ref{thm:EHP_range} are truncated, some algebraic manipulation is required to form an exact couple (for example extending the sequences to the left by a kernel and then zeros), and the resulting spectral sequence will differ from the $\aone$-EHP spectral sequence.  Nevertheless, it is shown in \cite{WickelgrenWilliams} that after localizing at $2$, the exact sequences of Theorem \ref{thm:EHP_range} are ``low-degree portions" of suitable long exact sequences, and these long exact sequences yield the $2$-primary $\aone$-EHP sequence (with the expected convergence properties).

The analysis of morphisms in the $\aone$-EHP sequence described above can be used to describe some differentials on the $E_1$ page in the $\aone$-EHP spectral sequence.  The desired $E_1$-differentials (given by the composite ${\mathrm H}{\mathrm P}$ linking the EHP sequences of different spheres) are then determined by the James--Hopf invariant of a Whitehead product. The axiomatic approach to Hopf invariants of Boardman and Steer \cite{Boardman_Steer} determines these James--Hopf invariants. In Section \ref{ss:propertiesofh}, we recast some results of Boardman and Steer in the context of simplicial presheaves. The main result is Proposition \ref{prop:s*H2P} which holds in the generality of simplicial presheaves on a site with enough points. In contrast, the remaining Section \ref{ss:HPKmwsubsection} is more specific to $\aone$-homotopy theory; Theorem \ref{thm:HP} identifies an $E_1$-differential in the $\aone$-EHP sequence with multiplication by a given element of $GW(k)$.

\subsection{Whitehead products for simplicial presheaves}
\label{ss:whiteheadproduct}
In this section, we give a construction of Whitehead products in $\aone$-homotopy theory, the construction generalizes classical results of \cite{Cohen, Arkowitz} to the context of simplicial presheaves. Suppose $\mathbf{C}$ is a site and as in Section \ref{ss:classicaljames} consider the category of pointed simplicial presheaves on $\mathbf{C}$ with its injective local model structure. If $X$ and $Y$ are pointed simplicial presheaves, we write $[X,Y]$ for morphisms in the associated (pointed) homotopy category. Unfortunately, it is also standard to use the notation $[-,-]$ for Whitehead products, but we hope that context will ensure that no confusion arises.

Recall that the constant simplicial presheaf $S^1$ is an $H$-cogroup object in the category of pointed simplicial
presheaves on $\mathbf{C}$, and there is an induced $H$-co-group structure on $\Sigma W$ for any pointed space $W$. In
particular, for any space $Z$, the space $\Map(\Sigma W,Z)$ has the structure of an $H$-group, functorially in both $Z$
and $W$. We will write $\cdot$ for the product in $\Map(\Sigma W,Z)$ and $(-)^{-1}$ for the inversion map; the constant
map to the base-point $\ast$ serves as a unit.

Suppose given pointed simplicial presheaves $X,Y$, and $Z$. The product projections $p_X: X \times Y \to X$ and
$p_Y: X \times Y \to Y$ are pointed, and induce morphisms $\Sigma p_X: \Sigma(X \times Y) \to \Sigma X$ and
$\Sigma p_Y: \Sigma(X \times Y) \to \Sigma Y$. In addition, the canonical map $X \vee Y \to X \times Y$ induces a
morphism $\Sigma (X \vee Y) \weq \Sigma X \vee \Sigma Y \to \Sigma(X \times Y)$ that fits into a cofiber sequence with
cofiber $\Sigma (X \sma Y)$.

\begin{construction}[Whitehead product]
\label{construction:whiteheadproduct}
Given maps $\alpha: \Sigma X \to Z$ and $\beta: \Sigma Y \to Z$, composition with the projections yields morphisms $a := \alpha \circ \Sigma p_X$ and $b:= \beta \circ \Sigma p_Y$. With respect to the product structure on $\Map(\Sigma (X \times Y),Z)$, we may consider the map
\[
(a^{-1}\cdot b^{-1}) \cdot (a \cdot b): \Sigma (X \times Y) \longrightarrow Z.
\]
We embed the map $(a^{-1}\cdot b^{-1}) \cdot (a \cdot b)$ into the following diagram:
\[
\xymatrix{
\Sigma X \vee \Sigma Y \ar[r]\ar[dr]& \Sigma(X \times Y) \ar[r]\ar[d]^{(a^{-1}\cdot b^{-1}) \cdot (a \cdot b)}& \Sigma(X \wedge Y) \\
& Z. &
}
\]

The pull-back of $(a^{-1} \cdot b^{-1})\cdot(a \cdot b)$ to $\Sigma (X \vee Y)$ has a prescribed null-homotopy described as follows.  The composition of the inclusion $\Sigma (X \times \ast) \to \Sigma (X \times Y)$ with $\pi_Y$ is the constant map. Thus, if we pull-back $(a^{-1} \cdot b^{-1})\cdot(a \cdot b)$ to $\Sigma (X \times \ast)$ the result coincides with the pull-back of $(a^{-1} \cdot {\ast}^{-1})\cdot(a \cdot {\ast})$. There is a canonical homotopy between $(a^{-1} \cdot {\ast}^{-1})\cdot(a \cdot {\ast})$ and the constant map $\ast$. Switching the roles of $X$ and $Y$ and $a$ and $b$, the pull-back of $(a^{-1} \cdot b^{-1})\cdot(a \cdot b)$ to $\Sigma (\ast \times Y)$ also admits a specified null-homotopy. Thus, the pull-back of $(a^{-1} \cdot b^{-1})\cdot(a \cdot b)$ to $\Sigma (X \vee Y)$ comes equipped with a specified null-homotopy. By means of this null-homotopy, the map $(a^{-1} \cdot b^{-1})\cdot(a \cdot b)$ passes to a well-defined homotopy class of maps $\Sigma(X \sma Y) \longrightarrow Z$; we write $[\alpha,\beta]$ for any representative of this class.
\end{construction}

Since the sequence $1 \to [\Sigma (X \sma Y), Z] \to [\Sigma (X \times Y), Z] \to [\Sigma(X \vee Y)] \to 1$ is exact, the choice of null-homotopy does not affect the homotopy class of  $[\alpha,\beta]$.

\begin{defn}
\label{defn:whiteheadproduct}
Given maps $\alpha: \Sigma X \to Z$ and $\beta: \Sigma Y \to Z$, a representative for the homotopy class of maps
\[
[\alpha,\beta]: \Sigma(X \wedge Y) \longrightarrow Z.
\]
in \textup{Construction \ref{construction:whiteheadproduct}} (or the homotopy class itself) is a {\em Whitehead product} of $\alpha$ and $\beta$.
\end{defn}

Following classical conventions, we write $\id_X$ for the identity map on a (pointed) simplicial presheaf $X$, and by a slight abuse of notation, we also let $\id_{\Sigma X}$ denote the inclusion $\Sigma X \to \Sigma X \vee \Sigma Y$. The construction above with $\alpha = \id_{\Sigma X}$ and $\beta = \id_{\Sigma Y}$ also yields a canonical map
\[
[\id_{\Sigma X}, \id_{\Sigma Y}]: \Sigma(X \wedge Y) \longrightarrow \Sigma X \vee \Sigma Y
\]
that can be thought of as a universal Whitehead product in the sense that the Whitehead product of Definition \ref{defn:whiteheadproduct} can be obtained by composing $[\id_{\Sigma X}, \id_{\Sigma Y}]$ with the map $\alpha \vee \beta: \Sigma X \vee \Sigma Y \to Z$. Regarding this product, we have the following result, which is a straightforward consequence of \cite[Theorem 4.2]{Arkowitz}.

To state this result, introduce the following notation. For $X$ and $Y$ pointed spaces, let $\tau: \Sigma \Sigma(X \wedge Y)  \isomt  \Sigma \Sigma(X \wedge Y) $ denote the map which switches the two suspensions and is the identity on $X \wedge Y$. Let  $\switch: \Sigma \Sigma (X \sma Y)  \isomt  \Sigma X \sma \Sigma Y$ be the map which does not change the order of the suspensions and is the identity on $X$ and $Y$.

\begin{lem}
\label{lem:cofib_wh_iXiY}
If $X$ and $Y$ are pointed connected spaces, then there is a cofiber sequence of the form:
\[
\xymatrix{
\Sigma(X \wedge Y) \ar[rr]^-{[\id_{\Sigma X}, \id_{\Sigma Y}]}&& \Sigma X \vee \Sigma Y \ar[r]& \Sigma X \times \Sigma Y
}
\]
where the second map is the usual map from the sum to the product, and such that the induced weak equivalence
\[
 \Sigma^2(X \wedge Y) \isomto  \Sigma X \wedge \Sigma Y
\]
is homotopic to $\switch \tau$.
\end{lem}

\begin{proof}
Recall that if $X$ and $Y$ are two pointed spaces, their join, typically denoted $X \ast Y$, is the homotopy pushout of the diagram $X \leftarrow X \times Y \rightarrow Y$. There is a functorial sectionwise weak equivalence $X \ast Y \weq \Sigma X \sma Y$. Using this identification, the universal Whitehead product can be thought of as a map with source $X \ast Y \to \Sigma X \vee \Sigma Y$ (cf. \cite[Definition 2.3]{Arkowitz}).

We first treat the case of simplicial sets. Thus, suppose $A$ and $B$ are connected simplicial sets. Let $C(A \ast B \to \Sigma A \vee \Sigma B )$ be the reduced mapping cone of $[\id_{\Sigma A}, \id_{\Sigma B}]$. In the proof of \cite[Theorem 4.2]{Arkowitz}, one finds a natural map
\[
\Theta: C(A \ast B \to \Sigma A \vee \Sigma B ) \longrightarrow \Sigma A \times \Sigma B
\]
(this map is called $G$ in \cite{Arkowitz} and is originally due to D.E. Cohen \cite{Cohen}). We claim the map $\Theta$ is a homology isomorphism. In \cite[Theorem 4.2]{Arkowitz}, this is fact is established for $A$ and $B$ polyhedra with one of $A$ or $B$ compact. As a consequence, it is true for finite simplicial sets. Since homology commutes with filtered direct limits and since the map $\Theta$ is evidently compatible with passing to sub-simplicial-sets by inspection, it follows that $\Theta$ is a homology isomorphism for $A$ and $B$ arbitrary simplicial sets.

Now, we treat the general case of simplicial presheaves. Write $C([\id_{\Sigma X}, \id_{\Sigma Y}])$ for the cofiber of $[\id_{\Sigma X}, \id_{\Sigma Y}]: X \ast Y \to \Sigma X \vee \Sigma Y$. It follows that there is a map $\Theta: C([\id_{\Sigma X}, \id_{\Sigma Y}]) \to \Sigma X \times \Sigma Y $ defined sectionwise, i.e., for each object $U \in \mathbf{C}$ define $\Theta(U) : C([\id_{\Sigma X}, \id_{\Sigma Y}])(U) \to \Sigma X (U) \times \Sigma Y (U)$ to be the map above. For each such $U$, $\Theta(U)$ is a homology equivalence between simply connected simplicial sets, and therefore a sectionwise weak equivalence. Combining with the sectionwise weak-equivalence $X \ast Y \to \Sigma(X \sma Y)$ and using the compatibility of the definitions of the Whitehead product we have established the claimed cofiber sequence.

The cofiber sequence identifies the suspension of $\Sigma(X \wedge Y)$ with the homotopy cofiber of $\Sigma X \vee \Sigma Y \to \Sigma X \times \Sigma Y.$  To prove that the homotopy class of the resulting map
\begin{equation}
\label{desired_htyp_equiv}
\Sigma^2(X \wedge Y) \isomto  \Sigma X \wedge \Sigma Y
\end{equation}
is $\switch \tau$, by working sectionwise it suffices to establish the analogous claim in the context of simplicial sets, which in turn can be reduced to checking the claim in the context of CW complexes, as considered in \cite{Arkowitz}.

We recall the following constructions from \cite{Arkowitz}. Let $S$ denote the unreduced suspension, $T$ denote the unreduced cone, and $C$ denote the reduced cone. Let $A$ and $B$ be pointed, locally-finite, connected, CW-complexes. The map $\Theta$ induces a map \begin{equation}\label{theta_mod_SigmaXvSigmaY}C(A \ast B \to \Sigma A \vee \Sigma B ) / (\Sigma A \vee \Sigma B) \longrightarrow (\Sigma A \times \Sigma B)/(\Sigma A \vee \Sigma B).\end{equation} By construction \cite[Theorem 4.2, Lemma 4.1]{Arkowitz}, \eqref{theta_mod_SigmaXvSigmaY} is induced from a map of pairs
\begin{equation}
\label{stN}
(T(A \ast B), A \ast B ) \longrightarrow (\Sigma A \times \Sigma B, \Sigma A \vee \Sigma B),
\end{equation} constructed as follows. Define
\[
N: (T(A \ast B), A \ast B) \longrightarrow (TA \times TB, TA \times B \cup A \times TB)$$ by $$ N (u,(t, a,b)) =  \begin{cases} (u, a) \times (1-2t(1-u),b)&\mbox{if }0 \leq t \leq \frac{1}{2} \\
(1-2(1-t)(1-u),a) \times (u, b) & \mbox{if } \frac{1}{2} \leq t \leq 1 \end{cases}.
\]
For a space $W$, let $t_W: TW \to SW$ and $s_W: SW \to \Sigma W$ be the quotient maps. Then \eqref{stN} is defined to be the composite
\[
\xymatrix{ (T(A \ast B), A \ast B ) \ar[rr]^N && (TA \times TB, TA \times B \cup A \times TB) \ar[d]^{t_A \times t_B}\\ (\Sigma A \times \Sigma B, \Sigma A \vee \Sigma B) && \ar[ll]_{S_A \times S_B} (SA \times SB, SA \vee SB) }.
\]
There is a weak equivalence $\mu': A \ast B \to \Sigma A \wedge B$ given by quotienting by points of the form $(a, \ast, t)$ and $(\ast, b, t)$, where $\ast$ denotes the base points, and $t$ denotes the coordinate of the interval in the standard representation of the join. Let $\mu$ be a homotopy inverse. The map \eqref{desired_htyp_equiv} in the homotopy category is  $\overline{M} \circ \Sigma (\mu)$, where $\overline{M}$ denotes the map on quotient spaces associated to the map of pairs $(S_X \times S_Y)  \circ (t_X \times t_Y) \circ N$. It therefore suffices to show that
\[
(u,t) \longmapsto \begin{cases} (u, 1-2t(1-u))&\mbox{if }0 \leq t \leq \frac{1}{2} \\
(1-2(1-t)(1-u), u) & \mbox{if } \frac{1}{2} \leq t \leq 1 \end{cases}
\]
determines an endomorphism of $S^1 \sma S^1$ of degree $-1$. This is easily checked: For example, the fiber over $(\frac{1}{4}, \frac{3}{4})$ consists of the single point $(u,t) =(\frac{1}{4}, \frac{1}{6})$ and the Jacobian determinant at $(\frac{1}{4}, \frac{1}{6})$ is negative.

\end{proof}

\begin{prop}
If $Z$ is an $h$-space in the category of pointed simplicial presheaves on $\mathbf{C}$, then $[\alpha,\beta] = 0$ for all $\alpha \in [\Sigma X,Z]$ and $\beta \in [\Sigma Y,Z]$.
\end{prop}

\begin{proof}
If $Z$ is an $h$-space, then the group $[\Sigma(X \times Y),Z]$ is an abelian group by the Eckmann--Hilton argument. Therefore, the commutator must be zero.
\end{proof}

\subsection{On the map \texorpdfstring{${\mathrm P}$}{P} in the \texorpdfstring{$\aone$}{A1}-EHP sequence}
\label{ss:pmapinaoneehp}
In this section, we return to the setting of $\aone$-homotopy theory and we analyze the map
\[
{\mathrm P}: \bpi_{i+1}^{\aone}(J(\mathscr{X}^{\sma 2}) )\longrightarrow \bpi_i^{\aone}(\mathscr{X})
\]
in the exact sequence of Theorem \ref{thm:EHP_range} under the additional assumption that the pointed space $\mathscr{X}$ is itself a suspension $\mathscr{X}= \Sigma \mathscr{Z}$ (see the beginning of Section \ref{ss:unstableconnectivity} for a reminder regarding conventions). In direct analogy with the results mentioned in the introduction, the map ${\mathrm P}$ can be described in terms of the Whitehead product introduced in Definition \ref{defn:whiteheadproduct}; the main result is Theorem \ref{P=piWhitehead}.

The map ${\mathrm P}$ was defined as the connecting homomorphism in the exact sequence of Theorem
\ref{thm:EHP_range}. This exact sequence does not arise directly from a fiber sequence, however. If $\mathscr{X}$ is as
above, then we can recast the sequence of (\ref{eq:A1SJH}) as the homotopy commutative diagram:
\[
\xymatrix{
\Laone \mathscr{X} \ar[d]^-{\mathrm{E}} \ar[r]& \ast \ar[d] \\
\Laone J(\mathscr{X}) \ar[r]^-{\mathrm{H}} & \Laone J(\mathscr{X}^{\sma 2}).
}
\]
By functoriality of homotopy fibers, there is an induced morphism
\[
\hofib(\Laone \mathscr{X} \to \Laone J(\mathscr{X})) \longrightarrow \Omega \Laone J(\mathscr{X}^{\sma 2})
\]
and the connectivity of this map is what allows us to define the map $P$ in Theorem \ref{thm:EHP_range}.

In the range where the map above is connected, it makes sense to consider the composite map:
\begin{equation}
\label{eqn:Pcomparisonpartofsquare}
\bpi_i^{\aone}(\Sigma \mathscr{Z} \wedge \mathscr{Z}) \longrightarrow \bpi_{i+1}^{\aone}(J(\Sigma \mathscr{Z} \wedge \Sigma {\mathscr Z})) \stackrel{{\mathrm P}}{\longrightarrow} \bpi_i^{\aone}(\Sigma \mathscr{Z}).
\end{equation}
Precisely, if $\mathscr{Z}$ is $(n-2)$-connected, then this composite is defined for $i \le 3n-3$.
We shall furthermore see that, provided the $\aone$-connectivity property holds for $S$ and $n \geq 2$, the first map in
\eqref{eqn:Pcomparisonpartofsquare} is in fact an isomorphism in the range being considered.

On the other hand, we saw in Section \ref{ss:whiteheadproduct} that the Whitehead square of the identity $[\iota_{\Sigma \mathscr{Z}},\iota_{\Sigma \mathscr{Z}}]$ gives a morphism $\Sigma \mathscr{Z} \sma \mathscr{Z} \longrightarrow \Sigma \mathscr{Z}$. We will abuse notation and write $[\iota_{\Sigma \mathscr{Z}},\iota_{\Sigma \mathscr{Z}}]$ for the map
\[
[\iota_{\Sigma \mathscr{Z}},\iota_{\Sigma \mathscr{Z}}]: \Laone \Sigma \mathscr{Z} \wedge \mathscr{Z} \longrightarrow \Laone \Sigma \mathscr{Z}.
\]
This morphism induces a pushforward map on homotopy sheaves
\[
[\iota_{\Sigma \mathscr{Z}},\iota_{\Sigma \mathscr{Z}}]_*: \bpi_i^{\aone}(\Sigma \mathscr{Z} \wedge \mathscr{Z}) \longrightarrow \bpi_i^{\aone}(\Sigma \mathscr{Z})
\]
that we would like to compare with the map \eqref{eqn:Pcomparisonpartofsquare}. The next result, which gives precisely such a comparison, provides an analog of \cite[Theorem XII.2.4]{Whitehead} or, rather, its extension to general $(n-1)$-connected spaces in the spirit of \cite[Theorem 3.1 and p. 231]{Ganea}, in the context of unstable $\aone$-homotopy theory.

\begin{thm}
\label{P=piWhitehead}
Assume the unstable $\aone$-connectivity property holds for $S$ and suppose $n \geq 2$ is an integer. If $\mathscr{Z}$ is an $\aone$-$(n-2)$-connected pointed space, and $\mathscr{X} = \Sigma \mathscr{Z}$, then for any positive integer $i \leq 3n-4$ the composite morphism of \textup{(\ref{eqn:Pcomparisonpartofsquare})} fits into a commutative diagram of the form
\[
\xymatrix{
\bpi_i^{\aone}( \Sigma (\mathscr{Z} \wedge \mathscr{Z}))\ar[rr]^{\sim}\ar@/_2pc/[rrrr]_-{[\id_{\Sigma \mathscr{Z}}, \id_{\Sigma \mathscr{Z}}]_*} && \bpi_{i+1}^{\aone} (J(\mathscr{X}^{\wedge 2})) \ar[rr]^-{\mathrm P} && \bpi_i^{\aone}(\mathscr{X});
}
\]
the isomorphism of homotopy sheaves in the diagram is induced by two-fold suspension.
\end{thm}

\begin{proof}
Without loss of generality, assume $\mathscr{Z}$ is $\aone$-fibrant.

We begin with the cofiber sequence
\[\xymatrix{
\Sigma(\mathscr{Z} \wedge \mathscr{Z}) \ar[rr]^-{u}&& \Sigma
\mathscr{Z} \vee \Sigma \mathscr{Z} \ar[r]& \Sigma \mathscr{Z} \times \Sigma \mathscr{Z}}
\]
of Lemma \ref{lem:cofib_wh_iXiY}. Here $u$ denotes the ``universal'' Whitehead product; the map $[\id_{\Sigma
  \mathscr{Z}}, \id_{\Sigma \mathscr{Z}}]: \Sigma(\mathscr{Z} \wedge \mathscr{Z}) \to \Sigma \mathscr{Z}$ is obtained
by composing $u$ with a fold map.

By construction, the map $u$ is represented by $u: \mathscr{C} \to \Sigma
\mathscr{Z} \vee \Sigma \mathscr{Z}$, where $\mathscr{C}$ denotes the reduced mapping cone of $\Sigma (\mathscr{Z} \vee \mathscr{Z}) \to \Sigma (\mathscr{Z} \times \mathscr{Z})$.  We now consider the following commutative diagram:
\begin{equation}
\label{eqn:morphismofcofibersequences}
\xymatrix{
\mathscr{C} \ar^u[rr]\ar@{=}[d] && \Sigma \mathscr{Z} \vee \Sigma \mathscr{Z} \ar[r]\ar[d]& \Sigma \mathscr{Z} \times \Sigma \mathscr{Z} \ar[d] \ar[r] & \Sigma \mathscr{Z} \sma  \Sigma \mathscr{Z} \ar@{=}[d]\\
\mathscr{C} \ar[rr]^-{[\iota_{\Sigma \mathscr{Z}},\iota_{\Sigma \mathscr{Z}}]} && \Sigma \mathscr{Z} \ar[r]& J_2(\Sigma \mathscr{Z}) \ar[r]  & \Sigma \mathscr{Z} \sma  \Sigma \mathscr{Z};
}
\end{equation}
The vertical maps are, reading left to right, the identity, the fold map, the canonical map from the product to $J_2$,
and the identity. The horizontal arrows $\Sigma \mathscr{Z} \to J_2(\Sigma \mathscr{Z})$ in the center are the canonical
inclusions. We observe that the center square is a pushout, being the definition of $J_2(\Sigma \mathscr{Z})$, and
therefore the rightmost two horizontal maps are quotient maps.
Since $\Sigma \mathscr{Z} \vee \Sigma \mathscr{Z} \to
\Sigma \mathscr{Z} \times \Sigma \mathscr{Z}$ is a cofibration, this square is in fact a homotopy pushout square. The two rows are homotopy cofiber sequences.

We form the following diagram, where the upper row is a homotopy cofiber sequence and the lower row a fiber sequence
\begin{equation}
  \label{eq:11}
  \xymatrix{ \Sigma \mathscr{Z} \sma \mathscr{Z} \ar^-{[\iota_{\Sigma \mathscr{Z}},\iota_{\Sigma \mathscr{Z}}]}[r] \ar^-{g}@{-->}[d] & \Sigma \mathscr Z \ar[r] \ar^-{t}[d] & J_2 (\Sigma \mathscr Z) \ar[r] \ar[d] & \Sigma \mathscr{Z} \sma \Sigma \mathscr{Z} \ar^-{f}@{-->}[d] \\
    \Omega J((\Sigma \mathscr{Z})^{\sma 2}) \ar^-{s}[r] & \hofib \mathrm{H} \ar[r] & J(\Sigma \mathscr Z) \ar^-{\mathrm{H}}[r] & J((\Sigma \mathscr{Z})^{\sma 2} ) }
\end{equation}
The maps indicated by the dashed arrows are adjoint to one another: by \cite[Chapter I.3, proof of Proposition 6, 3.13]{Quillen_homotopical_algebra} the map $ \Sigma^2 \mathscr{Z} \sma \mathscr{Z}\to \Sigma \mathscr{Z} \sma \Sigma \mathscr{Z}\stackrel{f}{\to} J((\Sigma \mathscr{Z})^{\sma 2} )$ is inverse adjoint to $g$. By Lemma \ref{lem:cofib_wh_iXiY}, the map $ \Sigma^2 \mathscr{Z} \sma \mathscr{Z}\to \Sigma \mathscr{Z} \sma \Sigma \mathscr{Z}$ is simplicially homotopic to $\switch \tau$, which reverses the order of the two suspensions, inducing $-1$ in the homotopy category.

The map $f: \Sigma \mathscr{Z} \sma \Sigma
\mathscr{Z} \to J((\Sigma \mathscr{Z})^{\sma 2} )$ is given by the canonical map $\Sigma \mathscr{Z} \times \Sigma \mathscr{Z} \to  J(\Sigma
\mathscr{Z})$ followed by $\mathrm{H}$; an analysis of $\mathrm{H}$, Definition \ref{defn:simplicialJamesHopf}, now shows that $f$ is the suspension map $\mathrm{E}$ as
applied to $\Sigma \mathscr{Z} \sma \Sigma \mathscr{Z}$. Therefore, $g$ is the two-fold suspension map.

We may $\aone$ localize throughout. We do not draw Diagram \eqref{eq:11} a second time with
$\Laone$ prepended to all terms. Since all spaces appearing above are at least $1$-connected, Corollary \ref{cor:relativehurewicz} applies throughout and maps that are $n$-connected remain $n$-connected after $\aone$ localization. We apply $\bpi_i^{\aone}$ to the leftmost square to obtain a commuting square
\[
\xymatrix{ \bpi_i^{\aone}(\Sigma \mathscr{Z} \sma \mathscr{Z}) \ar^-{[\iota_{\Sigma \mathscr{Z}},\iota_{\Sigma \mathscr{Z}}]}[r] \ar^-{g_*}[d] & \bpi_i^{\aone}( \Sigma \mathscr Z )  \ar^-{t_*}[d]  \\
\bpi_{i+1}^{\aone}( J((\Sigma \mathscr{Z})^{\sma 2}) )\ar^-{s_*}[r] & \bpi_i^{\aone}(\hofib \mathrm{H}).
}
\]
The map $g_*$, the two-fold suspension map, is an isomorphism when $i \le 4n-3$.  When $i \le 3n-2$, the map $t_*$ is an isomorphism, and by definition $\mathrm{P} = t_*^{-1} \circ s_*$; the result follows.
\end{proof}

\subsection{The James--Hopf invariant of a Whitehead product}
\label{ss:propertiesofh}
In this section, we collect some properties of the James--Hopf invariant defined in Section \ref{ss:classicalEHP}. Probably due to proliferation of different definitions of ``Hopf invariants" made at the time, Boardman and Steer \cite{Boardman_Steer} made an axiomatic study of such invariants. We adapt some of their results to the context under consideration.

Let $\mathbf{C}$ be a site with enough points and consider simplicial presheaves on $\mathbf{C}$ equipped with the injective local model structure. For pointed simplicial presheaves $X$ and $Y$, we continue to use the notation $[X,Y]$ for morphisms in the associated (pointed) homotopy category, and rely on context to distinguish this notation from that for the Whitehead product.

\subsubsection*{Cup products}
Before passing to the main results, it will be necessary to recall some constructions from \cite{Boardman_Steer}.

\begin{construction}[{Cup product (cf. \cite[Definition 1.3]{Boardman_Steer})}]
\label{construction:cupproduct}
Let $X$, $Y$, and $Z$ be pointed simplicial presheaves. The reduced diagonal map $X \to X \sma X$ is the composite of the diagonal $X \to X \times X$ and the map $X \times X \to X \sma X$; this map is null-homotopic if $X$ is a suspension. The smash product induces a map
\[
[\Sigma^m X,Y] \times [\Sigma^n X,Z] \longrightarrow [\Sigma^m X \wedge \Sigma^n X,Y \wedge Z].
\]
The reduced diagonal then induces a morphism
\[
\Sigma^{m+n} X \longrightarrow \Sigma^{m+n} X \wedge X \weq \Sigma^m X \wedge \Sigma^n X;
\]
the isomorphism on the right does not permute the suspension factors. The composite of these two morphisms defines the cup product pairing
\[
\cuppr : [\Sigma^m X, Y] \times [\Sigma^n X, Z] \longrightarrow [\Sigma^{n+m} X, Y \wedge Z]
\]
\end{construction}

The following result, which is identical in form to \cite[Lemma 1.4]{Boardman_Steer} summarizes the properties of cup products we use.

\begin{lem}
\label{lem:cupproduct}
Suppose $X$, $Y$ and $Z$ are pointed simplicial presheaves on $\mathbf{C}$.
\begin{enumerate}[noitemsep,topsep=1pt]
\item The cup product pairing is bilinear and associative.
\item If $X$ is itself a suspension, the cup-product pairing is trivial.
\end{enumerate}
\end{lem}

\begin{proof}
For Point (1), note that the smash product is bilinear and associative, and the pull-back is a homomorphism. Thus, the bilinearity and associativity of the cup product follow immediately.
As regards Point (2), if there exists a pointed simplicial presheaf $W$ such that $X = \Sigma W$, then the reduced diagonal map $\Sigma W \to \Sigma W \sma \Sigma W$ is simplicially null-homotopic, which means the cup-product pairing is trivial.
\end{proof}

\subsubsection*{Hopf invariants after Boardman--Steer}
Now, fix a pointed space $Z$ and consider the James--Hopf invariant ${\mathrm H}: J(Z) \to J(Z^{\sma 2})$ from Definition \ref{defn:simplicialJamesHopfpresheaf}. Applying $[X,-]$ to this morphism, the identification of $J(-)$ with $\Omega \Sigma(-)$ of Proposition \ref{prop:J(X)=ho-nis=Omega_SigmaX} and the loops suspension adjunction, ${\mathrm H}$ determines a map
\[
{\mathrm H}: [\Sigma X, \Sigma Z] \longrightarrow [\Sigma X, \Sigma Z \wedge Z];
\]
in an abuse of notation, we denote this map also by $\mathrm H$, the context should make clear which version we mean.

Following Boardman and Steer \cite[Definition 2.1]{Boardman_Steer}, define
$\lambda_2: [\Sigma X, \Sigma Z] \to [\Sigma^2 X, \Sigma^2 (Z \sma Z)]$ by
$\lambda_2 = \Sigma \circ {\mathrm H}$. Note that $\lambda_2$ vanishes on suspensions. The classical interaction between
$\lambda_2$ and the group operation $\cdot$ induced by the co-$h$-group structure of a suspension generalizes to the
context under consideration; this is a special case of the Cartan formula \cite[Definition 2.1 (c) and Theorem
3.15]{Boardman_Steer}, which uses the notion of cup products just introduced. To state the result, let
$\switch: \Sigma Z \sma \Sigma Z \isomt \Sigma^2 Z \sma Z$ be the map which does not change the order of the
suspensions or the order of the $Z$'s. The next result provides a direct analogue of \cite[Formula 3.14]{Boardman_Steer}, but we include the proof
for the convenience of the reader.

\begin{lem}
\label{CartanSuspensions}
Suppose $X$ and $Z$ are pointed simplicial presheaves on $\mathbf{C}$. Given $\alpha_1,\ldots, \alpha_n$ in $[X,Z]$,
consider $\Sigma \alpha_i \in [\Sigma X,\Sigma Z]$. The following formula holds in $[\Sigma^2 X, \Sigma^2 Z \sma Z]$:
\[
\lambda_2(\Sigma \alpha_1 \cdots \Sigma \alpha_n) = \prod_{1 \leq i<j \leq n} \switch (\Sigma \alpha_{i} \cuppr \Sigma \alpha_{j}),
\]
where the product $\Sigma \alpha_1 \cdots \Sigma \alpha_n$ is taken with respect to the group structure on $[\Sigma X,\Sigma Z]$ and $\prod$ denotes the group operation $\cdot$ ordered lexicographically from left to right.
\end{lem}

\begin{proof}
We abuse notation slightly and write $\alpha_i$ for a specified representative of the homotopy class of maps $\alpha_i \in [X,Z]$. In that case, $\Sigma \alpha_1 \cdots \Sigma \alpha_n$ is adjoint to the element of $[X, J(Z)]$ determined by $x \mapsto \alpha_1(x)\alpha_2(x)\cdots \alpha_n(x)$ for $x \in X(U)$.

Then, ${\mathrm H} (\Sigma \alpha_1 \cdots \Sigma \alpha_n)$ is adjoint to the element of $[X, J(Z \sma Z)]$ represented by
\[
x \mapsto \prod_{1 \leq i<j \leq n} \alpha_i(x) \wedge \alpha_j(x) = \prod_{1 \leq i<j \leq n} (\alpha_i \wedge \alpha_j)(x \wedge x),
\]
where $\prod$ denotes the group operation on $J(Z \sma Z)$ ordered lexicographically from left to right. Thus, $\lambda_2(\Sigma \alpha_1 \cdots \Sigma \alpha_n) = \prod_{1 \leq i<j \leq n} \switch(\Sigma \alpha_{i} \cuppr \Sigma \alpha_{j}),$ as claimed.
\end{proof}

Let
\[
\begin{split}
s: \Sigma (X \times Y) &\longrightarrow \Sigma(X \wedge Y) \text{, and } \\
\Sigmas: \Sigma^2 (X \times Y) &\longrightarrow \Sigma^2(X \wedge Y)
\end{split}
\]
be the suspension and two-fold suspension of the usual map from the product to the smash. The group operation on $[\Sigma^2(X \times Y), \Sigma^2 Z \sma Z]$ is abelian, so we use the symbol $+$ instead of $\cdot$ when writing this operation.

\begin{prop}\label{lambda2WhiteheadProd_suspensions}
Let $X$ and $Y$ be pointed simplicial presheaves on $\mathbf{C}$ and assume both are suspensions. If $\alpha \in [\Sigma X, \Sigma Z]$ and $\beta \in [\Sigma Y, \Sigma Z]$ are suspensions, then
\[
\Sigmas^* \lambda_2[\alpha,\beta] = - \switch(\pi_Y^* \beta \cuppr \pi_X^* \alpha) + \switch (\pi_X^* \alpha \cuppr \pi_Y^* \beta).
\]
\end{prop}

We remind the reader that $[\alpha, \beta]$ denotes the Whitehead product of $\alpha$, $\beta$. Here
$\pi^*_X : [ \Sigma X, \Sigma Z] \to [ \Sigma (X \times Y), \Sigma Z]$ is the map induced by projection.

\begin{proof}
All cup products will be followed by $\switch$, so we suppress $\switch$ from the notation. By definition, $s^* [\alpha, \beta] = \pi_X^* \alpha^{-1} \cdot \pi_Y^* \beta^{-1} \cdot \pi_X^* \alpha \cdot \pi_Y^* \beta$. By the naturality of $\lambda_2$ and Proposition \ref{CartanSuspensions},
\begin{align*}
\Sigmas^* \lambda_2 [\alpha, \beta] = & \pi_X^* \alpha^{-1} \cuppr \pi_Y^* \beta^{-1} + \pi_X^* \alpha^{-1} \cuppr \pi_X^* \alpha + \pi_X^* \alpha^{-1} \cuppr \pi_Y^* \beta + \\
& \pi_Y^* \beta^{-1} \cuppr \pi_X^* \alpha + \pi_Y^* \beta^{-1} \cuppr \pi_Y^* \beta + \pi_X^* \alpha \cuppr \pi_Y^* \beta
\end{align*}
Since $X$ is a suspension, the cup product $\alpha^{-1} \cuppr \alpha$ vanishes by Lemma \ref{lem:cupproduct}(2), whence $\pi_X^* \alpha^{-1} \cuppr \pi_X^* \alpha = 0$. The same reasoning shows that $\pi_Y^* \beta^{-1} \cuppr \pi_Y^* \beta=0$. By Lemma \ref{lem:cupproduct}(1), the cup-product is bilinear and associative, and so we obtain the following formula:
\begin{align*}
\Sigmas^* \lambda_2 [\alpha, \beta] &=  \pi_X^* \alpha \cuppr \pi_Y^* \beta - \pi_X^* \alpha \cuppr \pi_Y^* \beta - \pi_Y^* \beta \cuppr \pi_X^* \alpha +  \pi_X^* \alpha \cuppr \pi_Y^* \beta \\
 &= - \pi_Y^* \beta \cuppr \pi_X^* \alpha +  \pi_X^* \alpha \cuppr \pi_Y^* \beta.
\end{align*}
\end{proof}

\subsubsection*{Hopf invariants of Whitehead products}
Let $Z$ be a pointed simplicial presheaf on $\mathbf{C}$. Consider the maps
\[
[\id_{\Sigma Z}, \id_{\Sigma Z}]: \Sigma Z \wedge Z \longrightarrow \Sigma Z,
\]
and
\[
\Sigma {\mathrm H} [\id_{\Sigma Z}, \id_{\Sigma Z}]: \Sigma^2 Z \wedge Z \longrightarrow \Sigma^2 Z \wedge Z.
\]
Let $e: Z \sma Z \to Z \sma Z$ denote the exchange map, i.e., the map that permutes the two factors.

\begin{prop}
\label{prop:s*H2P}
Let $Z$ be a pointed simplicial presheaf on $\mathbf{C}$ that is a suspension. In the homotopy category, there is an equality
\[
\Sigma {\mathrm H} [\id_{\Sigma Z}, \id_{\Sigma Z}] = - \Sigma^2 e + \Sigma^2 \id_{Z \wedge Z}.
\]
\end{prop}

\begin{proof}
Let $\pi_i : \Sigma (Z \times Z) \to \Sigma Z$ denote the suspension of the $i$th projection for $i=1,2$. Let $\id:
\Sigma Z \to \Sigma Z$ denote the identity. Let $\switch: \Sigma Z \sma \Sigma Z \to \Sigma^2 (Z \sma Z)$ be the permutation that does not change the order of suspensions, and let $\switch^{-1}: \Sigma^2 (Z \sma Z) \to \Sigma Z \sma \Sigma Z$ be its inverse.

Consider the map $\Sigmas: \Sigma^2 (Z \times Z) \to \Sigma^2 (Z \sma Z)$ as introduced above. In that case, Proposition \ref{lambda2WhiteheadProd_suspensions} allows us to conclude the following equality holds:
\[
\Sigmas^*\Sigma {\mathrm H} [\id, \id] = -\switch(\pi_2^* \id \cuppr \pi_1^* \id) + \switch(\pi_1^* \id \cuppr \pi_2^* \id).
\]

Write $\Delta: Z \times Z \to (Z \times Z) \sma (Z \times Z)$ for the reduced diagonal map. Let
\[
\switch': \Sigma^2 (Z \times Z) \wedge (Z \times Z) \longrightarrow \Sigma (Z \times Z) \wedge \Sigma (Z \times Z)
\]
denote the permutation which does not swap the order of the two suspensions in $\Sigma^2 (Z \times Z) \sma (Z \times Z)$. Note that
\[
(\pi_1^* \id \wedge \pi_2^* \id) \circ \switch' \Sigma^2 \Delta : \Sigma^2 (Z \times Z) \longrightarrow \Sigma(Z \times Z) \wedge \Sigma (Z \times Z) \longrightarrow \Sigma (Z) \wedge \Sigma (Z)
\]
is homotopic to $\Sigmas$ followed by the permutation $\switch^{-1}$. Thus, $(\pi_1^* \id \sma \pi_2^* \id) \circ \switch' \Sigma^2 \Delta = \switch^{-1} \Sigmas$ up to homotopy.

By definition, $ \pi_1^* \id \cuppr \pi_2^* \id = (\pi_1^* \id \sma \pi_2^* \id) \circ \switch' \Sigma^2 \Delta$. Thus $ \pi_1^* \id \cuppr \pi_2^* \id = \switch^{-1} \Sigmas$ in the homotopy category. Applying $\switch$ to both sides, we conclude that
\[
\switch (\pi_1^* \id \cuppr \pi_2^* \id) = \Sigmas.
\]

Note that
\[
(\pi_2^* \id \wedge \pi_1^* \id) \circ \switch' \Sigma^2 \Delta : \Sigma^2 (Z \times Z) \longrightarrow \Sigma(Z \times Z) \wedge \Sigma (Z \times Z) \longrightarrow \Sigma (Z) \wedge \Sigma (Z)
\]
is homotopic to $\switch^{-1} \circ \Sigma ^2 e \circ \Sigmas$, whence
\[
\switch(\pi_2^* \id \cuppr \pi_1^* \id)= \Sigmas^* \Sigma^2 e.
\]

It follows that
\[
\Sigmas^* \Sigma {\mathrm H} [\id_{\Sigma Z}, \id_{\Sigma Z}] = \Sigmas^* ( - \Sigma^2 e + \Sigma^2 \id_{Z \wedge Z}).
\]
To conclude, we simply observe that $\Sigmas^*$ is injective. Indeed, this follows from the standard fact that for
simplicial presheaves $X,Y$, after a single suspension, the cofiber sequence $X \vee Y \to X \times Y \to X \sma Y$ is
split by the sum of the projections $\Sigma p_X$ and $\Sigma p_Y$. In that case, the long exact sequence in homotopy
obtained by mapping any suspension of the above cofiber sequence into the space $Z$ splits into a collection of short
exact sequences.
\end{proof}

\subsection{The composite \texorpdfstring{${\mathrm H} {\mathrm P}$}{HP} for a sphere as an element of \texorpdfstring{$\K^{\MW}_0(k)$}{K0(k)}}
\label{ss:HPKmwsubsection}
In this section, we analyze the composite map ${\mathrm H}{\mathrm P}$ for a sphere. Up to this point in the paper, we
have worked either in the context of simplicial presheaves on a site having enough points or in the unstable $\aone$-homotopy theory over a base for which the unstable $\aone$-connectivity property holds. The results in this section differ from those earlier in the paper because they will use finer structure of the $\aone$-homotopy category over a base field $k$ assumed to be perfect (and infinite for those being especially careful). We will try to be clear about exactly which ingredients do not follow from the ``axiomatic" point of view.

Morel shows the sheaf $\bpi_p^{\aone}(S^{p + q \alpha})$ (see \ref{notation:spheres}) is isomorphic to the Milnor--Witt K-theory sheaf $\K^{\MW}_q$ for $p \geq 2$ (or, somewhat exceptionally, for $p = 1$ and $q = 2$) \cite[Theorem 1.23]{MField}. Stringing the EHP exact sequences of Theorem \ref{thm:EHP_range} for different spheres together, one obtains the following diagram:
\[
\xymatrix{
& \bpi_{i+3}^{\aone}(S^{2p+3+2q\alpha}) \ar[d]^-{{\mathrm P}} & & \\
\bpi_{i}^{\aone}(S^{p+q\alpha})\ar[r]^-{{\mathrm E}} & \bpi_{i+1}^{\aone}(S^{p+1+q\alpha}) \ar[r]^-{\mathrm H} \ar[d]^-{{\mathrm E}}& \bpi_{i+1}^{\aone}(S^{2p+1+2q\alpha}) \ar[r]^-{{\mathrm P}} & \bpi_{i-1}^{\aone}(S^{p+q\alpha}) \\
& \bpi_{i+2}^{\aone}(S^{p+2+q\alpha}) & &
}
\]
The composite map ${\mathrm H}{\mathrm P}$ becomes the $E_1$-differential in the EHP spectral sequence.

When $i = 2p$, the composite map ${\mathrm H}{\mathrm P}$ is, by means of Morel's computations, a morphism
\[
{\mathrm H}{\mathrm P}: \K^{\MW}_{2q} \longrightarrow \K^{\MW}_{2q}.
\]
Note that by definition of Milnor--Witt K-theory sheaves, there is a ring homomorphism $\K^{\MW}_0(k) \to \hom(\K^{\MW}_{2q},\K^{\MW}_{2q})$ induced on sections by multiplication; moreover this homomorphism is necessarily injective. Lemma \ref{lem:homfromkmwiscontraction}, combined with the computation of contractions of Milnor--Witt K-theory sheaves (see the discussion after Lemma \ref{lem:contractionisexact}), implies this morphism is an isomorphism if $k$ has characteristic unequal to $2$; that this map is an isomorphism is also true if $k$ has characteristic $2$ but a different proof is required---see the discussion after Lemma \ref{lem:contractionisexact} for more details. In any case, the map ${\mathrm H}{\mathrm P}$ corresponds to an element of $\hom(\K^{\MW}_{2q},\K^{MW}_{2q})$; we will see below that it always lies in the subring $\K^{\MW}_0(k)$.

In order to state the result, we need some more precise information about the structure of the Milnor--Witt K-theory
ring. Recall that $\K^{\MW}_*(k)$ is generated by elements $[a] \in k^*$ of degree $+1$ and an element $\eta$ of degree
$-1$ subject to various relations \cite[Definition 3.1]{MField}. For a unit $a \in k^*$, set $\langle a \rangle := 1 +
\eta [a]$; the identification $\K^{\MW}_0(k) \iso GW(k)$ sends the element $\langle a \rangle$ to the class of the
$1$-dimensional symmetric bilinear form of the same name \cite[Lemma 3.10]{MField}. Following Morel \cite[p. 51]{MField}
or \cite[\S 6.1]{MIntro}, we set $\epsilon := - \langle -1 \rangle$. The class $\epsilon$ is related to the map $\gm{}
\sma \gm{} \to \gm{} \sma \gm{}$ that exchanges the two factors: see \cite[Lemma 6.1.1(2)]{MIntro} for a ``stable"
statement or \cite[Lemma 3.43]{MField} for an ``unstable" statement. Note that $1 - \varepsilon$ is the hyperbolic form
$h$, which intercedes in the definition of Milnor--Witt K-theory.

\begin{thm}
\label{thm:HP}
Assume $k$ is a perfect field and suppose $p,q$ are integers with $p > 1$ and $q \geq 1$. The map
\[
{\mathrm H}{\mathrm P}:\K^{\MW}_{2q} = \bpi_{2p+2}^{\aone} J ((S^{p + 1+ q \alpha})^{\wedge 2} ) \longrightarrow \bpi_{2p} J((S^{p + q \alpha})^{\wedge 2}) = \K^{\MW}_{2q},
\]
is given by $1 - (-1)^{p}\epsilon^q \in \K^{\MW}_0(k)$. Equivalently:
\[
{\mathrm H}{\mathrm P} = \begin{cases}0 & \text{ if } p \text{ and } q \text{ are even}, \\ 2 & \text{ if } p \text{ is odd and } q \text{ is even},\\ h & \text{ if } p \text{ is even and } q \text{ is odd, and} \\ 1+\epsilon & \text{ if } p \text{ and }q \text{ are odd.}\end{cases}
\]
\end{thm}

\begin{rem}
See Remark \ref{rem:charunequalto2} for more details on the situation when $k$ has characteristic $2$. There is a corresponding statement if $q = 0$ as well, but in that case, the composite in question gets identified with an element of $\Z$, not $\K^{\MW}_0(k)$ and the result is the classical computation of $\mathrm{H}\mathrm{P}$.
\end{rem}

\begin{proof}
Under the hypothesis on $k$, the unstable $\aone$-connectivity property holds by Theorem \ref{thm:unstableaoneconnectivityperfectfields}. The hypotheses on $p$ and $q$ are simply those that need to be imposed to appeal to Morel's computations of homotopy sheaves.

We begin by applying Theorem \ref{P=piWhitehead} with ${\mathscr X}= S^{p + 1+ q \alpha}$, ${\mathscr Z} = S^{p+q \alpha}$, $n=p+1$ and $i = 2p+1$ to interpret ${\mathrm P}$ as a Whitehead product. More precisely, since $1 \leq p $, it follows that $i=(2p+1) \leq 3(p+1)-3$, and so we know that ${\mathrm P}$ is induced by the map $[\id_{\Sigma {\mathscr Z}}, \id_{\Sigma {\mathscr Z}}]_*$ in degree $2p+1$. It follows that the composite ${\mathrm H}{\mathrm P}$ is isomorphic to the map obtained by applying $\bpi_{2p+1}^{\aone}$ to ${\mathrm H} [\id_{\Sigma \mathscr Z}, \id_{\Sigma \mathscr Z}]$.

Next, we appeal to our results about James--Hopf invariants of Whitehead products (Proposition \ref{prop:s*H2P}) to produce an explicit formula for ${\mathrm H}{\mathrm P}$ in terms of the exchange map $e: \mathscr{Z} \sma \mathscr{Z} \to \mathscr{Z} \sma \mathscr{Z}$. Indeed, Proposition \ref{prop:s*H2P} yields an equality of the form
\[
\Sigma {\mathrm H}[\id_{\Sigma {\mathscr Z}},\id_{\Sigma {\mathscr Z}}] = - \Sigma^2 e + \Sigma^2 \id_{{\mathscr Z} \wedge {\mathscr Z}}.
\]
However, recall that by Theorem \ref{thm:EHP_range} (combined with Remark \ref{rem:refinesfreudenthal}), the suspension map
\[
\Sigma: [\Sigma({\mathscr Z} \wedge {\mathscr Z}), \Sigma({\mathscr Z} \wedge {\mathscr Z})] \longrightarrow [\Sigma^2({\mathscr Z} \wedge {\mathscr Z}), \Sigma^2({\mathscr Z} \wedge {\mathscr Z})]
\]
is an isomorphism and we conclude that the following equality holds:
\[
{\mathrm H}[\id_{\Sigma {\mathscr Z}},\id_{\Sigma {\mathscr Z}}] = - \Sigma e + \Sigma \id_{{\mathscr Z} \wedge {\mathscr Z}}.
\]
Thus, we see that ${\mathrm H}{\mathrm P}$ is isomorphic to the map induced by applying $\bpi_{2p+1}^{\aone}$ to $-\Sigma e + \Sigma \id_{{\mathscr Z} \sma {\mathscr Z}}$.

Now we identify the homotopy class of the exchange map, $e$. Since it can be effected by pairwise exchanging copies of $S^1$, and $\gm{}$ it suffices to understand the effect of each such exchange on a homotopy class. By \cite[Lemma 3.43]{MField}, in the presence of a single suspension, the exchange of copies of $\gm{}$ contributes a factor of $\epsilon$. It is well-known that the permutation map on $S^1 \sma S^1$ has degree $-1$. Combining these two observations, a straightforward induction argument allows us to conclude that $e$ has degree $(-1)^{p}\epsilon^{q}$. Thus, the map induced by applying $\bpi^{\aone}_{2p+1}(-)$ to $-\Sigma e + \Sigma \id_{{\mathscr Z} \sma {\mathscr Z}}$ is multiplication by $1 - (-1)^{p}\epsilon^{q}$. Since $\epsilon^2 = 1$ in $\K^{\MW}_0(k)$ by \cite[Lemma 3.5]{MField}, the statement of the theorem follows by simply listing the possible cases.
\end{proof}

\begin{rem}
\label{rem:zmod2equivariantinterpretation}
Classically, the composite ${\mathrm H}{\mathrm P}$ computed above is either $2$ or $0$ depending on the parity of the
dimension of the sphere in question (this follows immediately from the definition of the James--Hopf invariant and
symmetry properties of the Whitehead square of the identity). If one invokes real realization \cite[p. 121]{MV}, Theorem
\ref{thm:HP} can be viewed as a direct analog of this classical result. First, observe that $\K^{\MW}_0(\real) \iso
GW(\real) \iso \Z \oplus \Z$; this identification sends a symmetric bilinear form over the real numbers to its rank and signature. Under real realization, the sphere $S^{p+q\alpha}$ is sent to the $\Z/2$-equivariant sphere of the same name. In particular, the $\Z/2$-fixed point locus of $S^{p+q\alpha}$ is simply $S^p$, while the fixed point locus under the trivial group is the sphere $S^{p+q}$. From this point of view, the formula for ${\mathrm H}{\mathrm P}$ from Theorem \ref{thm:HP} simply reflects the relative parities of $p$ and $p+q$: the signature keeps track of the degree on fixed point loci for $\Z/2$, while the rank keeps track of the degree on fixed point loci for the trivial subgroup. For example, when $p$ is even and $q$ is odd, ${\mathrm H}{\mathrm P}$ is multiplication by $h$, which has rank $2$ and signature $0$.
\end{rem}

\section{Applications}
\label{s:applications}
We now collect some computational applications of Theorems \ref{thm:EHP_range} and \ref{thm:lowdegree}. Section
\ref{ss:milnorwittktheory} is of a preliminary nature and contains a number of results about Milnor--Witt K-theory
sheaves that are used elsewhere in the text; some of these facts are certainly well-known, but we could not find good
references. Section \ref{ss:pi4ponesmash3} contains new computations of a family of unstable $\aone$-homotopy sheaves of
motivic spheres: it contains the first computation since Morel's of an $S^1$-stable $\aone$-homotopy sheaf (see Theorems
\ref{thm:pi45ponesmash3} and \ref{thm:nun}). Finally, Section \ref{ss:miscellaneouscomputations} contains results
regarding unstable rationalized motivic homotopy sheaves, and $S^1$-stable homotopy sheaves of Voevodsky's mod $m$
motivic Eilenberg--MacLane spaces (see Theorems \ref{thm:rationalized} and \ref{thm:sonestableemspaces}). While it
should be clear from the referencing, essentially all of the results of this section require finer properties of the
unstable $\aone$-homotopy category than merely the unstable $\aone$-connectivity property.

\subsection{On Milnor--Witt K-theory sheaves}
\label{ss:milnorwittktheory}
In this section we study some properties of the Milnor--Witt K-theory sheaves $\K^{\MW}_n$ \cite[\S 3]{MField}. By \cite[Theorem 3.37]{MField}, the sheaves $\K^{\MW}_n$ are strictly $\aone$-invariant sheaves for any integer $n$. In fact, for any integer $n \geq 1$, $\K^{\MW}_n$ is the free strictly $\aone$-invariant sheaf on the sheaf of pointed sets $\gm{\sma n}$, and $\K^{\MW}_0$ is the free strictly $\aone$-invariant sheaf on $\gm{}/\gm{\times 2}$ by \cite[Theorem 3.46]{MField} (not pointed in this case).

\subsubsection*{Basic properties of Milnor--Witt K-theory sheaves}
If $\mathbf{M}$ is a presheaf of groups (actually, pointed sets suffices), then its contraction $\mathbf{M}_{-1}$ is the presheaf of groups on $\Sm_k$ defined by
\[
\mathbf{M}_{-1}(U) := \ker(\mathbf{M}(\gm{} \times U) \stackrel{(1 \times \Id)^*}{\longrightarrow} \mathbf{M}(U))
\]
where $1: \Spec k \to \gm{}$ is the unit map. The next result summarizes the properties of contractions we will use.

\begin{lem}[{\cite[Lemmas 2.32 and 7.33]{MField}}]
\label{lem:contractionisexact}
The assignment $\mathbf{M} \mapsto \mathbf{M}_{-1}$ defines an endofunctor of the category of strictly (or strongly) $\aone$-invariant sheaves which preserves exact sequences.
\end{lem}

We will freely use the fact that for any pair of integers $n,j$, $(\K^{\MW}_n)_{-j} = \K^{\MW}_{n-j}$ \cite[Lemma 2.9]{AsokFaselSpheres}.\footnote{This identification is due to Morel and appears in several places in \cite{MField} but without a proof. The proof given in \cite{AsokFaselSpheres} requires $k$ to have characteristic unequal to $2$ since it depends on the Gersten conjecture for the sheaves $\mathbf{I}^j$. The result can also be demonstrated in case $k$ has characteristic $2$ if one appeals to Morel's Gersten-Schmid resolution of $\K^{MW}_{n}$. Nevertheless, since we will momentarily restrict to the case where $k$ has characteristic different from $2$ for other reasons.} We write $\mathbf{W}$ for the sheaf of unramified Witt groups, and $\mathbf{I}^n \subset \mathbf{W}$ for the subsheaves of unramified powers of the fundamental ideal in the Witt ring \cite[\S 2.1]{MMilnor}. For any integer $m$, we write $\K^\M_n/m$ for the mod $m$ Milnor K-theory sheaf. The contractions of $\K^\M_m$ are summarized in \cite[Lemma 2.7]{AsokFaselSpheres}. There is a canonical morphism $\K^\M_n/2 \to \mathbf{I}^n/\mathbf{I}^{n+1}$; the Milnor conjecture on quadratic forms, now a theorem, asserts that this morphism is an isomorphism \cite{OVV, MMilnor}.

Suppose $k$ is a base field of characteristic unequal to $2$. Morel established \cite[Th\'eor\`eme 5.3]{MPuissances} under this hypothesis that there is a fiber product presentation of $\K^{\MW}_n$ relating the various sheaves described in the previous paragraph. For any integer $n$, there is a fiber product diagram of the form:\footnote{See \cite{GilleScullyZhong} for some corrections to \cite{MPuissances}.}
\[
\xymatrix{
\K^{\MW}_n \ar[r]\ar[d]& \mathbf{I}^n \ar[d] \\
\K^\M_n \ar[r] & \K^\M_n/2;
}
\]
by convention $\K^\M_n = \K^\M_n/2 = 0$ for $n < 0$, whereas $\mathbf{I}^n \iso\mathbf{W}$ for $n < 0$.

This fiber product presentation yields two fundamental exact sequences:
\begin{equation}
\begin{split}
\label{eqn:fundamentalexactsequences}
0 \longrightarrow \mathbf{I}^{n+1} \longrightarrow &\K^{\MW}_n \longrightarrow \K^\M_n \longrightarrow 0, \text{ and }\\
0 \longrightarrow 2\K^\M_n \longrightarrow &\K^{\MW}_n \longrightarrow \mathbf{I}^n \longrightarrow 0;
\end{split}
\end{equation}
we use these sequences repeatedly in the sequel.

\begin{rem}
\label{rem:charunequalto2}
The assumption that $k$ has characteristic unequal to $2$ is inessential above: the fiber product presentation exists without this condition as one can see by inspecting the proofs, and appealing to the results of Kato \cite{Kato} on the characteristic $2$ version of Milnor's conjecture involving symmetric bilinear forms instead of \cite{OVV}. However, a detailed proof of this generalization does not appear in the literature, and since in later applications we will be restricted to the characteristic unequal to $2$ case anyway, we have not pursued this generalization.
\end{rem}

\subsubsection*{On the structure of contracted sheaves}
\begin{lem}
\label{lem:homfromkmwiscontraction}
Suppose $\mathbf{M}$ is a strictly $\aone$-invariant sheaf.
\begin{enumerate}[noitemsep,topsep=1pt]
\item For any integer $n \geq 1$, there are isomorphisms
\[
\hom(\K^{\MW}_n,\mathbf{M}) \iso \mathbf{M}_{-n}(k).
\]
\item If $n \geq 2$, the evident map $\hom(\K^{\MW}_n,\mathbf{M}) \to \hom(\K^{\MW}_{n-1},\mathbf{M}_{-1})$ induced by contraction is an isomorphism compatible with the identification of \textup{Point (1)}.
\end{enumerate}
\end{lem}

\begin{proof}
Write $\underline{\hom}_*$ for the internal $\hom$ in the category of presheaves of pointed sets on $\Sm_k$.  In that case, unwinding the definitions, there is an identification $M_{-1} = \hom_* (\gm{}, M)$.  A straightforward induction argument combined with the adjunction between $\sma$ and $\hom_*$ then shows $M_{-n} = \underline{\hom}_*(\gm{\sma n}, M)$.

For $n \geq 1$, \cite[Theorem 3.37]{MField} shows that $\K^{\MW}_n$ is the free strictly $\aone$-invariant sheaf of groups on the sheaf of pointed sets $\gm{\sma n}$.  As a consequence, there are functorial identifications
\[
\underline{\hom}_*(\gm{\sma n}, M) \isomt \underline{\hom}(\K^{\MW}_n,\mathbf{M}),
\]
where $\underline{\hom}$ on the right hand side is the internal $\hom$ in the category of presheaves of abelian groups.  To complete the verification of Point (1), simply take sections over $k$.

In light of the discussion of the previous paragraphs, to establish Point (2) one simply observes that, as long as $n \geq 2$, the map in question arises via the following sequence of identifications:
\[
\underline{\hom}(\K^{\MW}_n,\mathbf{M}) \cong \underline{\hom}_{*}(\gm{\sma n},\mathbf{M}) \cong \underline{\hom}_{*}(\gm{\sma n-1},\underline{\hom}_{*}(\gm,\mathbf{M})) \cong \underline{\hom}(\K^{\MW}_{n-1},\mathbf{M}_{-1}).
\].
\end{proof}

\begin{lem}
\label{lem:homoutofkmw0}
Suppose $\mathbf{M}$ is a strictly $\aone$-invariant sheaf.
~\begin{enumerate}[noitemsep,topsep=1pt]
\item There is an isomorphism
\[
\hom(\K^{\MW}_0,\mathbf{M}) \iso \mathbf{M}(k) \times {}_h \mathbf{M}_{-1}(k),
\]
where ${}_h \mathbf{M}_{-1}(k)$ is the $h$-torsion subgroup. The first map is induced by the projection $\K^{\MW}_0 \to \Z$, while the second is induced by a splitting of the map $\mathbf{I} \to \K^{\MW}_0$.
\item For any integer $n \geq 1$, the map $\hom(\K^{\MW}_n,\mathbf{M}) \to \hom(\K^{\MW}_0,\mathbf{M}_{-n})$ induced by contraction has image the factor $\mathbf{M}_{-n}(k)$ of the product described in \textup{Point (1)}.
\end{enumerate}
\end{lem}

\begin{proof}
Note that $\K^{\MW}_0 = \H_0^{\aone}(\gm{}/\gm{\times 2}) = \widetilde{\H}_0^{\aone}(\gm{}/\gm{\times 2}_+)$ by \cite[Theorem 3.46]{MField} and what follows amounts to unwinding the proof of this result.  Note that $\gm{}/\gm{\times 2}$ is pointed by the image of $1 \in \gm{}(k)$ and there is an associated splitting $\gm{}/\gm{\times 2}_+ \cong S^0_k \vee \gm{}/\gm{\times 2}$.

By adjunction, one then obtains identifications of the form
\[
\begin{split}
\underline{\hom}(\K^{\MW}_0,\mathbf{M}) &\cong \underline{\hom}_{\ast}(\gm{}/\gm{\times 2}_+,\mathbf{M}) \\
&\cong \underline{\hom}_{\ast}(S^0_k \vee \gm{}/\gm{\times 2},\mathbf{M}) \\
&\cong \underline{\hom}_{\ast}(S^0_k,\mathbf{M}) \times \underline{\hom}_{\ast}(\gm{}/\gm{\times 2},\mathbf{M}) \\
&\cong \underline{\hom}(\Z,\mathbf{M}) \times \underline{\hom}(\widetilde{\H}_0^{\aone}(\gm{}/\gm{\times 2}),\mathbf{M}) \\
&\cong \mathbf{M} \times \underline{\hom}(\widetilde{\H}_0^{\aone}(\gm{}/\gm{\times 2}),\mathbf{M})
\end{split}
\]
Under this decomposition, the projection map $\K^{\MW}_0 \to \Z$ is precisely the rank map. On the other hand, the splitting $\gm{}/\gm{\times 2}_+ \cong S^0_k \wedge \gm{}/\gm{\times 2}$ corresponds to the splitting $\mathbf{I} \to \K^{\MW}_0$ as described before \cite[Corollary 3.47]{MField}.

Next, there is an exact sequence of Nisnevich sheaves of abelian groups of the form
\[
\gm{} \overset{x\mapsto x^2}{\longrightarrow} \gm{} \longrightarrow \gm{}/\gm{\times 2} \longrightarrow 0.
\]
Taking reduced $\aone$-homology (as the composite of taking the (based) free abelian group functor and the exact functor $\Laone^{ab}$) yields an exact sequence of the form:
\[
\widetilde{\H}_0^{\aone}(\gm{}) \longrightarrow \widetilde{\H}_0^{\aone}(\gm{}) \longrightarrow \widetilde{\H}_0^{\aone}(\gm{}/\gm{\times 2}) \longrightarrow 0.
\]
The map $\widetilde{\H}_0^{\aone}(\gm{}) = \K^{\MW}_1 \to \K^{\MW}_1$ induced by the squaring map on $\gm{}$ is multiplication by $h = \langle 1 \rangle + \langle -1 \rangle$ by \cite[Lemma 3.14]{MField}.  Thus, we conclude that
\[
\hom(\widetilde{\H}_0^{\aone}(\gm{}/\gm{\times 2}),\mathbf{M}) \iso {}_h \mathbf{M}_{-1}(k),
\]
which is what we wanted to show.

For Point (2), we appeal to Lemma \ref{lem:homfromkmwiscontraction}.  Indeed, it suffices by Lemma \ref{lem:homfromkmwiscontraction}(2) and induction to treat the case where $n = 1$.  As in the proof of Lemma \ref{lem:homfromkmwiscontraction}, adjunction yields identifications of the form:
\[
\underline{\hom}(\K^{\MW}_1,\mathbf{M}) \cong \underline{\hom}_*(\gm{},\mathbf{M}) \cong \underline{\hom}_{*}(S^0_k,\underline{\hom}_{*}(\gm{},\mathbf{M})) \cong \underline{\hom}(\Z,\mathbf{M}_{-1}).
\]
In particular, the map
\[
\hom(\K^{\MW}_1,\mathbf{M}) \longrightarrow \hom(\K^{\MW}_0,\mathbf{M}_{-1})
\]
induced by contraction factors through the isomorphism $\hom(\K^{\MW}_1,\mathbf{M}) \cong \hom(\Z,\mathbf{M}_{-1})$.

We treat the universal case: taking $\mathbf{M} = \K^{\MW}_1$, we see that the map induced by contraction factors through a homomorphism $\hom(\Z,\K^{\MW}_0)$.  Such homomorphisms correspond to elements of $\K^{\MW}_0(k)$ via the image of $1 \in \Z$ and under the identification, the identity map $\K^{\MW}_1 \to \K^{\MW}_1$ is sent to the class of $\langle 1 \rangle$.

The identification $(\K^{\MW}_1)_{-1} \cong \K^{\MW}_0$ can be seen in terms of the fiber product presentations $\K^{\MW}_1 \isomt \K^\M_1 \times_{\K^M_1/2} \mathbf{I}$ and $\K^{\MW}_0 \isomt \Z \times_{\Z/2} \mathbf{W}$.  The symbol map $\gm{} \to \K^{\MW}_1$ can be thought of as a set-theoretic splitting of the projection map $\K^{\MW}_1 \to \K^M_1$.  After contraction, this projection is sent to the rank map $\K^{\MW}_0 \to \Z$.  The factorization produced in the previous paragraph thus corresponds to a splitting of the rank map $\K^{\MW}_0 \to \Z$.  On the other hand, the decomposition in Point (1) corresponded with a decomposition $\K^{\MW}_0 \cong \Z \oplus \mathbf{I}$ and under this identification the unit $\langle 1 \rangle$ is sent to $(1,0)$, thus we conclude that the projection onto the other factor is the zero map.
\end{proof}

\begin{lem}
\label{lem:contractionsfactor}
Fix a base field $k$. If $\phi: \K^{\MW}_n \to
\mathbf{M}$ is a morphism of sheaves such that
$\phi_{-j} =0$, then
\begin{enumerate}[noitemsep,topsep=1pt]
\item \label{i:cf1} assuming $n \geq j \geq 0$, the morphism $\phi$ is trivial; and
\item  \label{i:cf2} assuming $0 \leq n < j$, the morphism $\phi$ factors through a morphism $\K^{\MW}_n/\mathbf{I}^{j} \to \mathbf{M}$.
\end{enumerate}
\end{lem}

\begin{proof}
Factor $\phi: \K^{\MW}_n \onto \operatorname{Im}(\phi) \hookrightarrow \mathbf{M}$.  Since the inclusion of the abelian category of strictly $\aone$-invariant sheaves into the abelian category of abelian sheaves is exact (see Lemma \ref{lem:stableaoneconnectivitystrictlyaoneinvariantabeliancategory}) we can assume without loss of generality that $\operatorname{Im}(\phi)$ is strictly $\aone$-invariant.  Thus, it suffices to consider the
case where $\phi: \K^{\MW}_n \onto \mathbf{M}$ is an epimorphism and $\mathbf{M}$ is strictly $\aone$-invariant.

If $\mathbf{M}_{-0} = 0$, then $\mathbf{M} = 0$, and there is nothing to check. Therefore, we can assume without loss of generality that $j \geq 1$. By Lemma \ref{lem:homfromkmwiscontraction}, for any integer $r \geq 1$, $\mathbf{M}_{-r}(k) \iso \hom(\K^{\MW}_r,\mathbf{M})$. For \eqref{i:cf1} we simply observe that if $n \geq j$, then the morphism $\K^{\MW}_n \to \mathbf{M}$ is the trivial map since $\mathbf{M}_{-n}(k) = 0$ as well.

For \eqref{i:cf2} begin by observing that, since $0 \leq n < j$, we can consider the following diagram:
\[
\K^{\MW}_j \longtwoheadrightarrow \mathbf{I}^j \longhookrightarrow \mathbf{I}^{n+1} \longhookrightarrow \K^{\MW}_n \longtwoheadrightarrow \mathbf{M};
\]
reading from the left, the first and third maps are those in the exact sequences in Diagram (\ref{eqn:fundamentalexactsequences}),
the map $\mathbf{I}^j \hookrightarrow \mathbf{I}^{n+1}$ is the standard inclusion (since $j \geq n+1$, this makes sense), and the final epimorphism is the one given by the assumptions. Since $\mathbf{M}_{-j} = 0$ (and $j \geq 1$ by assumption), this composite is trivial, which means the map $\K^{\MW}_n \to \mathbf{M}$ factors through the quotient $\K^{\MW}_n/\mathbf{I}^j$.
\end{proof}

\begin{rem}
This result will be applied below with $\K^{\MW}_n \to \mathbf{M}$ a map to a strictly $\aone$-invariant sheaf with
strictly $\aone$ invariant cokernel.
\end{rem}

\subsubsection*{Some result on \texorpdfstring{$\aone$}{A1}-tensor products}
\begin{prop}
\label{prop:aonetensorproductmilnorwittktheorysheaves}
For any integers $m,n \geq 1$, there is an isomorphism $\K^{\MW}_m \tensor^{\aone} \K^{\MW}_n \isomt \K^{\MW}_{m+n}$.
\end{prop}

\begin{proof}
Morel computed $\widetilde{\H}_{n-1}^{\aone}({\mathbb A}^i \setminus 0) \iso \K^{\MW}_n$ \cite[Theorem 6.40]{MField} (strictly speaking, this result is stated for $n \geq 2$, but the result is true for $n = 1$ as well by unwinding the definitions and appealing to \cite[Theorem 3.37]{MField}). There are identifications
\[
\Sigma {\mathbb A}^m \setminus 0 \wedge {\mathbb A}^n \setminus 0 \weq {\mathbb A}^m \setminus 0 \ast {\mathbb A}^n \setminus 0 \weq {\mathbb A}^{m+n} \setminus 0
\]
(here $\ast$ means join), for any $m,n \geq 1$. Proposition \ref{prop:aonehomologyofproduct} then yields
\[
\K^{\MW}_m \tensor^{\aone} \K^{\MW}_n \iso \widetilde{\H}_{n+m-2}^{\aone}({\mathbb A}^m \setminus 0 \wedge {\mathbb A}^n \setminus 0),
\]
which when combined with the suspension isomorphism $\widetilde{\H}_{n+m-2}^{\aone}({\mathbb A}^m \setminus 0\,\sma\,{\mathbb A}^n \setminus 0) \iso \widetilde{\H}_{n+m-1}^{\aone}({\mathbb A}^m \setminus 0 \ast {\mathbb A}^n \setminus 0) \iso \K^{\MW}_{n+m}$ yields the result.
\end{proof}

\begin{lem}
\label{lem:aonetensorproductmilnorktheorysheaves}
For any integers $m,n \geq 1$, and any integer $r \geq 0$ there are canonical isomorphisms $\K^{\MW}_m \tensor^{\aone}
\K^\M_n/r \iso \K^\M_{m+n}/r$. There are also canonical isomorphisms $\K^\M_m/r \tensor^{\aone} \K^\M_n/r \iso \K^\M_{m+n}/r$.
\end{lem}

\begin{proof}
By Lemma \ref{lem:homfromkmwiscontraction}, for $n \geq 2$, we can identify $\hom(\K^{\MW}_{n+1},\K^{\MW}_n) \iso
\K^{\MW}_{-1}(k) \iso \mathbf{W}(k)$ (the final identification by \cite[Lemma 3.10]{MField}). The group $\K^{MW}_{-1}(k)$ contains the element $\eta$ and we refer to the corresponding map $\K^{\MW}_{n+1} \to \K^{\MW}_n$ using the same notation. Unwinding the definitions, this map corresponds to the composite of the $\K^{\MW}_{n+1} \to \mathbf{I}^{n+1}$ defined on sections by multiplication by $\eta$ and the inclusion map $\mathbf{I}^{n+1} \hookrightarrow \K^{\MW}_{n}$.

By the discussion of the previous paragraph, the first exact sequence of Diagram (\ref{eqn:fundamentalexactsequences}) yields the exact sequence
\[
\K^{\MW}_{n+1} \stackrel{\eta}{\longrightarrow} \K^{\MW}_n \longrightarrow \K^\M_n \longrightarrow 0.
\]
Tensoring this exact sequence with $\K^{\MW}_{m}$ and applying Proposition \ref{prop:aonetensorproductmilnorwittktheorysheaves} we conclude that there is an exact sequence
\[
\K^{\MW}_{m+n+1} \stackrel{\eta}{\longrightarrow} \K^{\MW}_{m+n} \longrightarrow \K^{\MW}_m \tensor^{\aone} \K^\M_{n} \longrightarrow 0.
\]
However, this sequence identifies $\K^{\MW}_m \tensor^{\aone} \K^\M_{n} \iso \K^\M_{m+n}$. Repeating this discussion using the exact sequence $\K^\M_n \to \K^\M_n \to \K^\M_n/r$ we obtain the isomorphism of the statement. Repeating this discussion in the other factor allows us to obtain the final statement.
\end{proof}

\subsubsection*{Rationalized Milnor--Witt K-theory sheaves}
Now, we turn our attention to rationalized Milnor--Witt sheaves $\K^{\MW}_n \tensor \Q$, which will reappear in Section \ref{ss:miscellaneouscomputations}.

\begin{lem}
\label{lem:rationalizedmilnorwitt}
Fix a base field $k$, assumed to have characteristic unequal to $2$.
\begin{enumerate}[noitemsep,topsep=1pt]
\item For every integer $n$, there is a canonical isomorphism $\K^{\MW}_n \tensor \Q \isomt \K^\M_n \tensor \Q \times \mathbf{I}^n \tensor \Q$.
\item If $k$ is not formally real, then $\mathbf{I}^n \tensor \Q$ is trivial, i.e., $\K^{\MW}_n \tensor \Q \isomt \K^\M_n \tensor \Q$.
\end{enumerate}
\end{lem}

\begin{proof}
The first statement follows from the fiber product presentation of $\K^{\MW}_n$ together with the fact that $\K^\M_n/2 \tensor \Q = 0$. For the second statement, since $\mathbf{I}^n$ is an unramified sheaf, it suffices to show that $\mathbf{I}^n(L) \tensor \Q = 0$ for $L$ a finitely generated extension of $k$. If $k$ is not formally real, then any extension field has the same property, and the result follows immediately from the fact that $\mathbf{I}^n(L)$ is a $2$-torsion sheaf if $L$ is not formally real \cite[Proposition 31.4]{ElmanKarpenkoMerkurjev}.
\end{proof}

\begin{cor}
\label{cor:nontrivialityofrationalization}
Fix a base field $k$, assumed to have characteristic unequal to $2$.
\begin{enumerate}[noitemsep,topsep=1pt]
\item For any integer $n \geq 0$, $\K^{\MW}_n \tensor \Q$ is non-trivial.
\item If $k$ is formally real, then for any integer $n$, $\K^{\MW}_n \tensor \Q$ is non-trivial.
\end{enumerate}
\end{cor}

\begin{proof}
Both tensoring with $\Q$ and contraction are exact endofunctors of the category of strictly $\aone$-invariant sheaves
(cf. Lemma \ref{lem:contractionisexact}) and it follows immediately from the definitions that the two constructions
commute, i.e., if $\mathbf{M}$ is strictly $\aone$-invariant, then $(\mathbf{M} \tensor \Q)_{-1} \iso \mathbf{M}_{-1}
\tensor \Q$. Thus, to show $\K^{\MW}_n \tensor \Q$ is non-trivial, it suffices to show that for some $m > 0$,
$(\K^{\MW}_n \tensor \Q)_{-m} = (\K^{\MW}_{n})_{-m} \tensor \Q \iso \K^{\MW}_{n-m} \tensor \Q$ is non-trivial. By Lemma
\ref{lem:rationalizedmilnorwitt}(1), there is a canonical identification $\K^{\MW}_{n-m} \tensor \Q \iso \K^\M_{n-m} \tensor \Q \oplus \mathbf{I}^{n-m} \tensor \Q$.

For (1), take $m = n$, and observe that $\K^\M_{0} \iso \Z$. For (2) observe that if $m > n$, then $\K^{\MW}_{n-m}
\tensor \Q \iso \mathbf{W} \tensor \Q$.  Since $k$ is assumed formally real, we can choose an ordering of $k$
\cite[Proposition 31.20]{ElmanKarpenkoMerkurjev}, and thus find a real closed field $k'$ containing $k$. In that case,
observe that $W(k') \iso \Z$ by Sylvester's law of inertia \cite[Proposition 31.5]{ElmanKarpenkoMerkurjev}. Thus, $\mathbf{W}(k') \tensor \Q$ is non-zero, so the sheaf $\mathbf{W} \tensor \Q$ is non-trivial.
\end{proof}

\begin{rem}
Once again, the assumption that $k$ has characteristic unequal to $2$ is inessential. This assumption only appears by way of our appeal to Morel's fiber square presentation of $\K^{\MW}_n$ (cf. Remark \ref{rem:charunequalto2}).
\end{rem}

\begin{rem}
Rationalized Milnor K-theory sheaves can be quite large. If $L$ is an infinite field, write $L^{alg}$ for an algebraic closure. The Bloch-Kato conjecture \cite{VMod2,VModl} implies that for $n \geq 2$, the groups $K^M_n(L^{alg})$ are (non-trivial) uniquely divisible (these groups are evidently divisible for $n = 1$). By using transfers in Milnor K-theory \cite[\S I.5]{BassTate} it is easy to see that the restriction map $K^M_n(L) \to K^M_n(L^{alg})$ is injective modulo torsion. Any element $\alpha \in K^M_n(L)$ that goes to zero in $K^M_n(L^{alg})$ necessarily goes to zero in a finite extension $L'/L$. In that case, the composite $K^M_n(L) \to K^M_n(L') \to K^M_n(L)$ of restriction with transfer is multiplication by the degree. Thus, $[L':L]\alpha = 0$, i.e., $\alpha$ is torsion. Equivalently, one can use the identification of Milnor K-theory with motivic cohomology \cite[Theorem 5.1]{MVW} and transfers there.
\end{rem}

\subsection{On \texorpdfstring{$\aone$}{A1}-homotopy sheaves of spheres}
\label{ss:pi4ponesmash3}
The goal of this section is to establish Conjecture 5 of \cite{AsokFaselOberwolfach}. The results below depends rather heavily on the results of \cite{AsokFaselA3minus0} and thus we assume throughout this section that $k$ is an infinite perfect field of characteristic unequal to $2$.

\subsubsection*{On the computation of $\bpi_{3+j\alpha}^{\aone}(S^{2+3\alpha})$}
To begin, we recall some results \cite{AsokFaselA3minus0} where we used the notation $\bpi_{3,j}({\mathbb A}^3 \setminus 0)$ for the sheaf in the title. One begins by considering the fiber sequence
\begin{equation}
\label{eqn:fibseqa3minus0}
SL_4/Sp_4 \longrightarrow SL_6/Sp_6 \longrightarrow {\mathbb A}^5 \setminus 0.
\end{equation}
A stable range was described for the homotopy sheaves of $SL_{2n}/Sp_{2n}$ in \cite[Proposition
4.2.2]{AsokFaselA3minus0} in terms of Grothendieck--Witt sheaves (see \cite[\S 3.1 and 3.3]{AsokFaselA3minus0} and the
references there for explication of the notation). Also obtained there was a short exact sequence of sheaves of the form:
\begin{equation}
\label{eqn:exactsequencea3minus0}
\mathbf{GW}^3_5 \longrightarrow \mathbf{K}^{\MW}_5 \longrightarrow \bpi_{3}^{\aone}(S^{2+3\alpha}) \longrightarrow \mathbf{GW}^3_4 \longrightarrow 0.
\end{equation}
The cokernel of morphism $\mathbf{GW}^3_5 \to \K^{\MW}_5$ was called ${\mathbf F}_5$ and a description of
${\mathbf F}_5$ was given in \cite[Theorem 4.4.1]{AsokFaselA3minus0}. Before discussing the structure of this morphism,
we introduce a further convention to simplify the notation.

\begin{convention}
\label{convention:loops}
Write $\Omega(-)$ for the $\aone$-derived loop functor, i.e., $\Omega \Laone (-)$.
\end{convention}

The map $\K^{\MW}_5 \to \bpi_{3}^{\aone}(S^{2+3\alpha})$ is by construction induced by a morphism $\Omega S^{4+5\alpha} \to S^{2+3\alpha}$. The composite map
\[
\delta: S^{3+5\alpha} \longrightarrow \Omega S^{4+5\alpha} \longrightarrow S^{2+3\alpha}
\]
was a shown to be a generator of $\bpi_{3+5\alpha}(S^{2+3\alpha})$ in \cite[Proposition 5.2.1]{AsokFaselA3minus0}. We deduce a few simple consequences of these results now.

\begin{lem}
\label{lem:contractionsofa3minus0}
Suppose $j \geq 6$ is an integer.
\begin{enumerate}[noitemsep,topsep=1pt]
\item There is a canonical isomorphism $\bpi_{3+j\alpha}^{\aone}(S^{2+3\alpha}) \iso \mathbf{W}$.
\item Any $\aone$-homotopy class of maps $S^{3+j\alpha} \to S^{2+3\alpha}$ lifts uniquely along $\delta$ to a map $S^{3+j\alpha} \to S^{3+5\alpha}$.
\end{enumerate}
\end{lem}

\begin{proof}
For Point (1), begin by applying \cite[Theorem 6.13]{MField} to the exact sequence of (\ref{eqn:exactsequencea3minus0}). By \cite[Proposition 3.4.3]{AsokFaselA3minus0}, we observe that $(\mathbf{GW}^3_4)_{-j} = 0$ for $j \geq 5$. By \cite[Lemma 3.4.1]{AsokFaselA3minus0}, if $j = 5$, we conclude that $(\mathbf{GW}^3_5)_{-5} = \mathbf{GW}^2_0 \iso \mathbb{Z}$. Therefore, $(\mathbf{GW}^3_5)_{-j} = 0$ for $j \geq 6$. Thus, we conclude that there is a sequence of isomorphisms
\[
\bpi_{3+j\alpha}^{\aone}(S^{2+3\alpha}) \iso \bpi_{3+j\alpha}^{\aone}(\Omega S^{4+5\alpha}) \iso
\bpi_{4+j\alpha}^{\aone}(S^{4+5\alpha}) \iso (\K^{\MW}_5)_{-j} \iso \mathbf{W}
\]
if $j \geq 6$.

For Point (2), take a map $\phi: S^{3+j\alpha} \to S^{2+3\alpha}$ as in the statement. Mapping $S^{3+j\alpha}$ into the fiber sequence of (\ref{eqn:exactsequencea3minus0}), the argument of Point (1) shows that, for $j \geq 6$, such a map lifts uniquely to a map $S^{3+j\alpha} \to \Omega S^{4+5\alpha}$. Since $S^{3+5\alpha}$ is $\aone$-$2$-connected, the unit of the loop-suspension adjunction $S^{3+5\alpha} \to \Omega S^{4+5\alpha}$ induces an isomorphism on $\aone$-homotopy sheaves in degrees $\leq 4$ by, e.g., Theorem \ref{thm:EHP_range} and Remark \ref{rem:refinesfreudenthal}. Therefore, $\phi$ lifts uniquely along $\delta$.
\end{proof}

\subsubsection*{On the computation of \texorpdfstring{$\bpi_{j+1}^{\aone}(S^{j+3\alpha})$}{pi_(j+1)^(A1)(S^(j+3a))}, \texorpdfstring{$j \geq
    3$}{ j>= 3}}
\begin{prop}
\label{prop:pi46ponesmash3}
If $k$ is a field of characteristic $0$ and containing a quadratically closed subfield, then $\pi_{4+6 \alpha}^{\aone}(S^{3+3\alpha}) = 0$.
\end{prop}

\begin{proof}
Taking $\mathscr{X} = S^{2+3\alpha}$, and observing that $S^{2+3\alpha}$ is $\aone$-$1$-connected, Theorem \ref{thm:lowdegree}, the exactness of contraction and \cite[Theorem 6.13]{MField} yield the following exact sequence:
\begin{equation}
  \label{eq:6}
  \bpi_{5+6\alpha}^{\aone}(S^{3+3\alpha}) \longrightarrow \bpi_{5+6\alpha}^{\aone}(S^{5+6\alpha}) \stackrel{{\mathrm P}}{\longrightarrow} \bpi_{3+6\alpha}^{\aone}(S^{2+3\alpha}) \longrightarrow \bpi_{4+6\alpha}^{\aone}(S^{3+3\alpha}) \longrightarrow 0.
\end{equation}
The sheaf $\bpi_{5+6\alpha}^{\aone}(S^{5+6\alpha}) \iso \K^{\MW}_0$ again by Morel's computations \cite[Theorem 1.23]{MField}.

By Theorem \ref{P=piWhitehead}, the morphism ${\mathrm P}$ is induced by composition with
\[
[\id_{S^{2+3\alpha}},\id_{S^{2+3\alpha}}]: S^{3+6\alpha} \longrightarrow S^{2+3\alpha};
\]
we will refer to this map as composition with the Whitehead square of the identity. By Lemma \ref{lem:contractionsofa3minus0}(2), this map lifts uniquely through $\delta$ to a map
\[
[\id_{S^{2+3\alpha}},\id_{S^{2+3\alpha}}]: S^{3+6\alpha} \longrightarrow S^{3+5\alpha},
\]
which (by \cite[Corollary 6.43]{MField}) can be viewed as an element of $\mathbf{W}(k)$. By Lemma
\ref{lem:contractionsofa3minus0}(1), the exact sequence \eqref{eq:6} becomes
\[
\K^{\MW}_0 \stackrel{{\mathrm P}}{\longrightarrow} \mathbf{W} \longrightarrow \bpi_{4+6\alpha}^{\aone}(S^{3+3\alpha}) \longrightarrow 0.
\]

We claim that each of the sheaves in the above exact sequence are sheaves of $\K^{\MW}_0$-modules in a natural way, and that the morphisms are morphisms of sheaves of $\K^{\MW}_0$-modules.  To see this, it suffices to observe that the portion of the $\aone$-EHP sequence under consideration takes the form $\bpi_{3+6\alpha}(\Omega^2 S^{5 + 6\alpha}) \to \bpi_{3+6\alpha}^{\aone}(S^{2+3\alpha}) \to \bpi_{3+6\alpha}^{\aone}(\Omega S^{3+3\alpha})$, and the $\K^{\MW}_0$-module structure is induced by precomoposition with $\bpi_{3+6\alpha}(S^{3+6\alpha})$.  From these observations it follows that the map $\mathrm{P}$ is determined by an element of $\hom_{\K^{\MW}_0}(\K^{\MW}_0,\mathbf{W}) = \mathbf{W}(k)$.

Assume first that $k$ is a quadratically closed field of characteristic $0$. In that case $\mathbf{W}(k) = \Z/2$,
and, to establish the claim, it suffices to prove that our morphism is non-trivial. To see this, fix an embedding $k \hookrightarrow
\cplx$. Using complex realization (see \cite[\S 3 Lemma 3.4]{MV} or \cite{DuggerIsaksen}), and the fact that complex
realization takes spheres to spheres, it suffices to prove that composition with the Whitehead square of the identity is
non-trivial after taking $\cplx$-points. Serre showed that $\pi_9(S^5) = \Z/2$ and that Whitehead square of the identity
on $S^5$ is a generator, \cite[\S 41]{SerreEM}. Consequently, we conclude that the our morphism ${\mathrm P}$ also corresponds to the non-trivial element of $\mathbf{W}(k)$ and is therefore an epimorphism.

If $L/k$ is an extension field, then the morphism ${\mathrm P}$ in our sequence viewed over the base field $L$ is pulled back from the morphism ${\mathrm P}$ over $k$.  Thus, by appeal to the conclusion of the previous paragraph, we conclude in this case as well that the morphism $\K^{\MW}_0 \to \mathbf{W}$ is necessarily the standard epimorphism, and therefore that $\bpi_{4+6\alpha}^{\aone}(S^{3+3\alpha}) = 0$.
\end{proof}

\begin{rem}
In the preceding proof, the assumption that $k$ has characteristic $0$ can likely be weakened to the assumption that $k$ has characteristic unequal to $2$ via appeal to \'etale realization \cite{Isaksenetalerealization}. As a consequence, the same remark applies to all statements below appealing to Proposition \ref{prop:pi46ponesmash3}. Removing the assumption that $k$ contains a quadratically closed subfield will probably require different techniques. Nevertheless, it seems likely that the ``lifted" map $[\id_{S^{2+3\alpha}},\id_{S^{2+3\alpha}}]: S^{3+6\alpha} \to S^{3+5\alpha}$ is simply a suspension of $\eta$ and the above result can be established without reference to realization of any sort.
\end{rem}

The above vanishing statement has a number of useful consequences.

\begin{thm}
\label{thm:pi45ponesmash3}
If $k$ is a field of characteristic $0$ and containing a quadratically closed subfield, then for every integer $j \geq 3$, there is an exact sequence of the form
\begin{equation}
  \label{eq:9}
  0 \longrightarrow \mathbf{F}_5' \longrightarrow \bpi_{j+1}^{\aone}(S^{j+3\alpha}) \longrightarrow \mathbf{GW}^3_4 \longrightarrow 0,
\end{equation}
together with an epimorphism $\K^\M_5/24 \to \mathbf{F}_5'$ that becomes an isomorphism after $4$-fold contraction.   Moreover, the composite map $\K^\M_5/24 \to \bpi_{4}^{\aone}(S^{3+3\alpha})$ determines an isomorphism $\Z/24 \iso \bpi_{4+5\alpha}^{\aone}(S^{3+3\alpha})$.
\end{thm}

\begin{proof}
We treat the case where $j = 3$, building upon the analysis in the proof of Proposition \ref{prop:pi46ponesmash3}.  Since $S^{3+3\alpha}$ is $\aone$-$2$-connected, the case $j \geq 4$ will follow immediately from this case and the $\aone$-simplicial suspension theorem (Theorem \ref{thm:EHP_range} and Remark \ref{rem:refinesfreudenthal}).

Take $\mathscr{X} = S^{2+3\alpha}$ in Theorem \ref{thm:lowdegree} and consider the map ${\mathrm P}: \K^{\MW}_6 =
\bpi_{5}^{\aone}(S^{5+6\alpha}) \to \bpi_{3}^{\aone}(S^{2+3\alpha})$.  Recall the exact sequence of
\ref{eqn:exactsequencea3minus0}, which appears here as the horizontal line of \eqref{eq:7}
\begin{equation}
  \label{eq:7}
  \xymatrix{ & & \bpi^{\aone}_5(S^{5+6\alpha}) =\K^{\MW}_6 \ar^-{\mathrm{P}}[d] \ar@{..>}[dr]& \\ \mathbf{GW}^3_5 \ar[r] & \K^{\MW}_5 \ar[r]  &
  \bpi_{3}^{\aone}(S^{2+3\alpha}) \ar[r] \ar^-{\mathrm{E}}[d] &  \mathbf{GW}^3_4 \ar[r] &  0 \\
& & \bpi_4^{\aone}(S^{3+3\alpha}) \ar[d] \ar@{-->}[ur]\\ & & 0.}
\end{equation}
The vertical sequence is the EHP sequence applied to $S^{2+3\alpha}$. The dotted diagonal map is an element of
$\hom(\K^{\MW}_6,\mathbf{GW}^3_4) \cong (\mathbf{GW}^3_4)_{-6}(k)$ by Lemma \ref{lem:homfromkmwiscontraction}. On the other
hand \cite[Proposition 3.4.3]{AsokFaselA3minus0} allows us to conclude that $(\mathbf{GW}^3_4)_{-6} = 0$ so this
diagonal map vanishes, and therefore there is an induced epimorphism, denoted by the dashed diagonal arrow in
\eqref{eq:7}, $\bpi_{4}^{\aone}(S^{3+3\alpha}) \to \mathbf{GW}^3_4$, as required by the theorem.
By combining a portion of diagram \eqref{eq:7} with the exact sequence $0 \to \mathbf{I}^6 \to \K^{\MW}_5 \to \K^{\M}_5
\to 0$ we obtain diagram \eqref{eq:9a}, ignoring the dotted arrow for the moment.
\begin{equation}
  \label{eq:9a}
  \xymatrix{ & 0 \ar[d] \\ &\mathbf{I}^6 \ar[d] & & \\ \mathbf{GW}^3_5 \ar[r] & \K^{\MW}_5 \ar[d] \ar[r]  &
  \bpi_{3}^{\aone}(S^{2+3\alpha}) \ar[r] \ar^-{\mathrm{E}}[d] &  \mathbf{GW}^3_4 \ar[r] \ar@{=}[d] &  0 \\
& \K^{\M}_5 \ar@{..>}[r] \ar[d] & \bpi_4^{\aone}(S^{3+3\alpha}) \ar[d]  \ar[r] &\mathbf{GW}^3_4 \ar[r] & 0.\\ & 0 & 0}
\end{equation}
Since we know that $\bpi_4^{\aone}(S^{3+3\alpha})_{-6} = 0$, it follows from Lemma \ref{lem:contractionsfactor} that the
composite map $\K^{\MW}_5 \to \bpi_4^{\aone}(S^{3+3\alpha})$ factors through $\K^{\MW}_5/\mathbf{I}^6 = \K^{\M}_5$,
giving the dotted arrow in diagram \eqref{eq:9a}.

We define $\mathbf{F}_5$, as in \cite{AsokFaselA3minus0}, to be the cokernel of the map $\mathbf{GW}^3_5 \to \K^{\MW}_5$,
and define $\mathbf{F}_5'$ to be the image
of $\mathbf{F}_5$ in $\bpi_{4}^{\aone}(S^{3+3\alpha})$. The exact sequence
\[
0 \longrightarrow \mathbf{F}_5' \longrightarrow \bpi_{4}^{\aone}(S^{3+3\alpha}) \longrightarrow \mathbf{GW}^3_4 \longrightarrow 0.
\]
is an immediate
consequence of this definition. Furthermore, there is a diagram of exact sequences
\begin{equation}
  \label{eq:10}
  \xymatrix{ 0 \ar[r] &  \mathbf{I}^6 \ar@{=}[d] \ar[r] &  \K^{\MW}_5 \ar@{->>}[d] \ar[r] &  \K^{\M}_5  \ar@{->>}[d] \ar[r] & 0
  \\
   &  \mathbf{I}^6 \ar[r] &  \mathbf{F}_5 \ar[r] &  \mathbf{F}'_5 \ar[r] & 0. }
\end{equation}
To determine the behavior of $\mathbf{F}_5'$, we need finer information regarding the sheaf $\mathbf{F}_5$ as described
in \cite[Theorems 4.3.1 and 4.4.1]{AsokFaselA3minus0}. We provide a brief recapitulation of that description here. The
sheaf $\mathbf{F}_5$ is identified there as a quotient of a fibered product as follows.  One defines a sheaf
$\mathbf{S}_5$, the cokernel of a ``Chern class" map $\K^Q_5 \to \K^\M_5$ \cite[Definition 3.6]{AsokFaselSpheres}. The
sheaf $\mathbf{S}_5$ is equipped with a
canonical surjection onto $\K^\M_5/2$ (see \cite[Lemma 3.13]{AsokFaselSpheres} and the subsequent discussion). One then
defines a sheaf $\mathbf{T}_5$ to be the fiber product of $\mathbf{S}_5$ and $\mathbf{I}^5$ over $\K^\M_5/2$
\cite[p. 911]{AsokFaselSpheres}; the maps $\mathbf{I}^5 \to \K^\M_5/2$ and $\K^\M_5 \to \mathbf{S}_5 \to \K^\M_5/2$
coincide with the defining maps in the fiber product presentation in $\K^{\MW}_5$. By \cite[Theorem 4.3.1]{AsokFaselA3minus0} (cf. \cite[Theorem 4.4.1]{AsokFaselA3minus0}), there is an epimorphism
$\mathbf{T}_5 \to \mathbf{F}_5$ and this epimorphism becomes an isomorphism after $4$-fold contraction by \cite[Lemma
5.1.1]{AsokFaselA3minus0}. Assembling all the above, there is a diagram of
short exact sequences, enlarging \eqref{eq:10},
\begin{equation}
  \label{eq:8}
  \xymatrix{ 0 \ar[r] &  \mathbf{I}^6 \ar@{=}[d] \ar[r] &  \K^{\MW}_5 \ar@{->>}[d] \ar[r] &  \K^{\M}_5  \ar@{->>}[d] \ar[r] & 0
  \\
  0 \ar[r] &  \mathbf{I}^6 \ar[r] \ar@{=}[d]  &  \mathbf{T}_5 \ar[r]  \ar^{\phi}@{->>}[d] &  \mathbf{S}_5 \ar[r] \ar@{->>}^{\phi'}[d]  & 0 \\
 &  \mathbf{I}^6 \ar[r] &  \mathbf{F}_5 \ar[r] &  \mathbf{F}'_5 \ar[r] & 0.}
\end{equation}
where $\phi$, and therefore $\phi'$, becomes an isomorphism after $4$-fold contraction.

There is an epimorphism $\K^\M_{5}/24 \onto \mathbf{S}_5$ that becomes an isomorphism after $4$-fold contraction,
\cite[Corollary 3.11]{AsokFaselSpheres}. It follows there is such an epimorphism $\K^\M_{5}/24 \onto \mathbf{F}_5'$ as
well.

Since $(\K^{\M}_5)_{-5} \iso \ZZ$, and $(\mathbf{GW}_4^3)_{-5} = 0$, the latter by \cite[Proposition 3.4.3]{AsokFaselA3minus0}, the $5$-fold contraction of the sequence
\[
0 \longrightarrow \mathbf{F}_5' \longrightarrow \bpi_{4}^{\aone}(S^{3+3\alpha}) \longrightarrow \mathbf{GW}^3_4 \longrightarrow 0
\]
reduces to an isomorphism $\ZZ/24 \iso \bpi_{4+5\alpha}^{\aone}(S^{3+3\alpha})$.
\end{proof}

\begin{rem}
Because of the observation of \cite[Remark 5.1.2]{AsokFaselA3minus0}, we do not know whether the map $\K^\M_5/24 \to \mathbf{F}'_5$ of Theorem \ref{thm:pi45ponesmash3} is an isomorphism after $3$-fold contraction. Nevertheless, it seems likely that this is the case.
\end{rem}

Consider the motivic Hopf map $\nu: S^{3+4\alpha} \to S^{2+2\alpha}$. The standard construction of this map is via the Hopf construction \cite[p. 190]{MField} on the multiplication map $SL_2 \times SL_2 \to SL_2$.

\begin{cor}
\label{cor:nugenerates}
If $k$ is a field having characteristic $0$ and containing a quadratically closed subfield, then for every integer $j \geq 3$, the group $\bpi_{j+1+5\alpha}^{\aone}(S^{j+3\alpha}) \iso \Z/24$ is generated by $\Sigma^{j-2+\alpha} \nu$.
\end{cor}

\begin{proof}
This follows by combining Theorem \ref{thm:pi45ponesmash3} and \cite[Corollary 5.3.1]{AsokFaselA3minus0}.
\end{proof}

\begin{rem}
Since $\eta: S^{1+2\alpha} \to S^{1+\alpha}$ we can consider $\nu \sma \eta: S^{4+6\alpha} \to S^{3+3\alpha}$. Proposition \ref{prop:pi46ponesmash3} then guarantees that $\nu \sma \eta$ and $\eta \sma \nu$ are null-homotopic. Since they remain null-homotopic after suspension, we obtain a purely unstable proof of one of the motivic null-Hopf relations \cite[Proposition 5.4]{DuggerIsaksenHopf}.

Similarly, if $\eta_s: S^3_s \to S^2_s$ is the simplicial Hopf map, then we can the composite map:
\[
\Sigma^{2+2\alpha} \eta \circ \Sigma^{2+3\alpha}\eta \circ \Sigma^{2+4\alpha}\eta \circ \Sigma^{1+6\alpha}\eta_s: S^{4+6\alpha} \longrightarrow S^{3+3\alpha}.
\]
Once again, Proposition $\ref{prop:pi46ponesmash3}$ implies this composition is null-homotopic. Stabilizing with respect to $\pone$-suspension, this implies the relation $\eta^3 \eta_s = 0$ in the motivic stable homotopy ring. This relation is an incarnation of the fact that the topological Hopf map $\eta_{top}$ satisfies $\eta_{top}^4 = 0$.

The existence of such null-homotopies allows us to construct new elements in unstable homotopy sheaves of motivic
spheres using Toda brackets \cite{Toda}. It would be interesting to study such constructions more systematically.
\end{rem}

\subsubsection*{On the structure of \texorpdfstring{$\bpi_{n+1}^{\aone}(S^{n-1+n\alpha})$}{pi_(n+1)^(A1)(S^(n-1+na))}, \texorpdfstring{$n \geq 4$}{n>=4}}
Finally, we can establish \cite[Conjecture 5]{AsokFaselOberwolfach} under the additional hypothesis that our base field contains a quadratically closed field having characteristic $0$.

\begin{thm}
\label{thm:nun}
Suppose $k$ is a field that contains a quadratically closed field having characteristic $0$. For every integer $n \geq 4$, $\nu_n := \Sigma^{(n-2)+(n-2)\alpha}\nu$ induces a non-trivial morphism
\[
(\nu_n)_*: \K^\M_{n+2}/24 \longrightarrow \bpi_{n+1}^{\aone}(S^{n+n\alpha}).
\]
\end{thm}

\begin{proof}
This follows essentially from Corollary \ref{cor:nugenerates}. In more detail, the map $(\nu_n)_*$ determines a morphism $\K^{\MW}_{n+2} \to \bpi_{n+1}^{\aone}(S^{n+n\alpha})$, but by construction, this morphism factors through $\pone$-suspension. In particular, since the map $\K^{\MW}_5 \to \bpi_{4}^{\aone}(S^{3+3\alpha})$ factors through a morphism $\K^\M_5/24 \to \bpi_{4}^{\aone}(S^{3+3\alpha})$, we conclude that for any integer $n \geq 4$, the morphism $\K^{\MW}_{n+2} \to \bpi_{n+1}^{\aone}(S^{n+n\alpha})$ factors through a map $\K^{\MW}_{n-3} \tensor^{\aone} \K^\M_5/24 \to \bpi_{n+1}^{\aone}(S^{n+n\alpha})$. Lemma \ref{lem:aonetensorproductmilnorktheorysheaves} allows us to conclude that $\K^{\MW}_{n-3} \tensor^{\aone} \K^\M_5/24 \iso \K^\M_{n+2}/24$, which is precisely what we wanted to show.
\end{proof}

Recall that in \cite[Theorem 5]{AsokFaselKO}, a morphism
\[
\bpi_{n+1}^{\aone}(S^{n+n\alpha}) \longrightarrow \mathbf{GW}^n_{n+1}
\]
is constructed using ``Suslin matrices". The composite map
\[
\K^{\MW}_{n+2} \longrightarrow \bpi_{n+1}^{\aone}(S^{n+n\alpha}) \longrightarrow \mathbf{GW}^n_{n+1}
\]
is, by means of Lemma \ref{lem:homfromkmwiscontraction}, determined by an element of $(\mathbf{GW}^n_{n+1})_{-n-2}(k)$; since the latter group is trivial by \cite[Proposition 3.4.3]{AsokFaselA3minus0}, we conclude that this composite is trivial. Combining these observations with Theorem \ref{thm:nun} and the connectivity estimate from the $\aone$-simplicial suspension theorem (see Theorem \ref{thm:EHP_range} and Remark \ref{rem:refinesfreudenthal}), we now refine \cite[Conjecture 7]{AsokFaselOberwolfach}.

\begin{conj}
\label{conj:structureofpinplus1}
For any pair of integers $n \geq 4$ and $i \geq 0$, there is an exact sequence of the form
\[
\K^\M_{n+2}/24 \longrightarrow \bpi_{n}^{\aone}(S^{(n-1+i)+n\alpha}) \longrightarrow \mathbf{GW}^n_{n+1};
\]
the right hand map becomes an epimorphism after $(n-3)$-fold contraction, and the sequence becomes a short exact sequence after $n$-fold contraction.
\end{conj}

\begin{rem}
In private communication from 2005, Morel stated a conjecture about the stable $\bpi_1$ sheaf of the motivic sphere
spectrum. Conjecture \ref{conj:structureofpinplus1} can be thought of as an unstable refinement of Morel's conjecture. Morel's conjecture has been verified in various situations. K. Ormsby and P.-A. {\O}stv{\ae}r verified Morel's conjecture after taking sections over fields of small cohomological dimension \cite{OrmsbyOstvaer}. Much more generally, work of P.-A. {\O}stv{\ae}r, O. R{\"o}ndigs and M. Spitzweck has verified Morel's conjecture over fields having characteristic $0$ \cite{OstvaerRondigsSpitzweck} (or, more generally, after inverting the characteristic exponent of the base field). While these results provide evidence for Conjecture \ref{conj:structureofpinplus1}, without a version of the suspension theorem for $\pone$-suspension these stable results do not imply our conjecture.
\end{rem}

\subsection{Other computations}
\label{ss:miscellaneouscomputations}
In this section, we establish non-triviality of unstable rationalized $\aone$-homotopy sheaves of motivic spheres. We then go on to compute the first $S^1$-stable $\aone$-homotopy sheaf of a mod $m$ motivic Eilenberg--MacLane space.

\subsubsection*{Rationalized \texorpdfstring{$\aone$}{A1}-homotopy sheaves of spheres}
Morel's computations of $\aone$-homotopy sheaves of spheres yield isomorphisms $\bpi^{\aone}_{2n-1 } S^{2n-1 + 2q
  \alpha} \iso \K^{\MW}_{2q}$ for $2n-1 \geq 2$ \cite[Theorem 6.40]{MField}. By \cite[Theorem 6.13]{MField} (see also \cite[Corollary 6.43]{MField}), for any integer $j$ there are induced isomorphisms $\pi^{\aone}_{2n-1 + j \alpha} S^{2n-1 + 2q \alpha} \iso (\K^{\MW}_{2q})_{-j}$. In these degrees, the James--Hopf invariant map ${\mathrm H}$ of Section \ref{ss:aoneehp} yields a morphism
\[
{\mathrm H}: \bpi^{\aone}_{2n-1 + j \alpha} S^{n + q \alpha} \longrightarrow \bpi^{\aone}_{2n-1 + j \alpha} S^{2n-1 + 2q \alpha} \iso \K^{\MW}_{2q-j}.
\]
We now study the rationalized version of this map. The next result provides an analog of the fact, due to Hopf, that there is a surjection $\pi_{4n-1}(S^{2n}) \to \Z$.

\begin{thm}
\label{thm:rationalized}
Fix a base field $k$, assumed to be perfect and to have characteristic unequal to $2$. Let $n > 2$, $q \geq 2$ be even integers.
\begin{enumerate}[noitemsep,topsep=1pt]
\item For any integer $j \ge 0$, the sequence of sheaves \[
\bpi^{\aone}_{2n-2+j\alpha} S^{n-1 + q \alpha} \tensor \Q \stackrel{{\mathrm E} \tensor \Q}{\longrightarrow} \bpi^{\aone}_{2n-1+j\alpha} S^{n + q \alpha} \tensor \Q \stackrel{{\mathrm H} \tensor \Q}{\longrightarrow} \K^{\MW}_{2q-j} \tensor \Q \longrightarrow 0
\]
is exact.
\item If $k$ is not formally real, then for any integer $j$ satisfying $0\le j \leq 2q$, the sheaf $\bpi^{\aone}_{2n-1+j\alpha} S^{n + q \alpha} \tensor \Q$ is non-trivial.
\item If $k$ is formally real, then for any integer $j\ge 0$, $\bpi^{\aone}_{2n-1+j\alpha} S^{n + q \alpha} \tensor \Q$ is non-trivial.
\end{enumerate}
\end{thm}

\begin{proof}
Tensoring with $\Q$ and contraction are exact functors on the category of strictly $\aone$-invariant sheaves of abelian groups (cf. Lemma \ref{lem:contractionisexact}). Combining \cite[Theorem 6.13]{MField} with the exact sequence of Theorem \ref{thm:EHP_range} (which applies since $n \geq 3$ by assumption) and then tensoring with $\Q$, we obtain exactness of the above sequence at $ \pi^{\aone}_{2n-1} S^{n + q \alpha} \tensor \Q$. Since $n$ and $q$ are even by assumption, the class $1+ (-1)^{n+q}\langle -1 \rangle^q = 2$. Surjectivity of ${\mathrm H} \tensor \Q$ follows from Theorem \ref{thm:HP}.

Points (2) and (3) follow immediately from Corollary \ref{cor:nontrivialityofrationalization}.
\end{proof}

\begin{rem}
A corresponding statement holds for $q = 0$ as well, but that result follows immediately from the classical computation of non-zero rational homotopy groups of spheres.
\end{rem}

\subsubsection*{Some \texorpdfstring{$S^1$}{S1}-stable \texorpdfstring{$\aone$}{A1}-homotopy sheaves of motivic Eilenberg--MacLane spaces}
Set $K_n:= K(\Z(n),2n)$ and $K_n/m := K(\Z/m(n),2n)$ where for an abelian group $A$, the space $K(A(n),2n)$ is a motivic Eilenberg--MacLane space in the sense of Voevodsky; see, for example, \cite[\S 2]{VRed}. We write $\H^i_{\et}(\mu_m^{\tensor n})$ for the Nisnevich sheafification of the presheaf $U \mapsto H^i_{\et}(U,\mu_m^{\tensor m})$. In the next result, which is an analog of a result appearing in \cite[Example 5.11]{Barcus}, we adhere to Convention \ref{convention:loops}.

\begin{thm}
\label{thm:sonestableemspaces}
Assume $k$ is a field having characteristic exponent $p$. Fix integers $i \geq 1$ and $m,n \geq 2$ and assume $m$ is coprime to $p$.
\begin{enumerate}[noitemsep,topsep=1pt]
\item The space $\Sigma^i K_n/m$ is $\aone$-$(n+i-1)$-connected.
\item If $j$ is an integer satisfying $0 \leq j \leq n-1$, then there are isomorphisms of the form
\[
\bpi_{n+j+i}^{\aone}(\Sigma^i K_n/m) \isomto \H^{n-j}_{\et}(\mu_m^{\tensor n}).
\]
\item There is an exact sequence of the form
\[
\H^0_{\et}(\mu_m^{\tensor n}) \longrightarrow \bpi_{2n+i}^{\aone}(\Sigma^i K_n/m) \longrightarrow \K^\M_{2n}/m \longrightarrow 0,
\]
and $\H^0_{\et}(\mu_m^{\tensor n})$ is killed by a single contraction.
\end{enumerate}
\end{thm}

\begin{proof}
The space $K_n/m$ is $\aone$-$(n-1)$-connected and it is possible to describe all its higher $\aone$-homotopy sheaves. If $m$ is prime to $p$ by the Bloch--Kato conjecture (in Beilinson--Lichtenbaum form) together with $\aone$-representability of mod-$m$ motivic cohomology \cite{VMod2,VModl}, there are isomorphisms of the form $\bpi_{n+r}^{\aone}(K_n/m) \cong \H^{n-r}_{\et}(\mu_m^{\tensor n})$. In particular, $\H^{n-r}_{\et}(\mu_m^{\tensor n})$ is isomorphic to $\K^\M_n/m$ for $r = 0$ and vanishes for $r > n$.

We begin by investigating what happens after a single suspension. By the $\aone$-Freudenthal suspension theorem, the map $\bpi_{r}^{\aone}(K_n/m) \to \bpi_{r+1}^{\aone}(\Sigma K_n/m)$ is an isomorphism for $r \leq 2n-2$. We now show that this map is an isomorphism for $r = 2n-1$ as well.

Theorem \ref{thm:lowdegree} applied with $\mathscr{X} = K_n/m$ yields an exact sequence of the form
\[
\bpi_n^{\aone}(K_n/m) \tensor^{\aone} \bpi_n^{\aone}(K_n/m) \longrightarrow \bpi_{2n-1}^{\aone}(K_n/m) \longrightarrow \bpi_{2n}^{\aone}(\Sigma K_n/m) \longrightarrow 0.
\]
By the discussion of the previous paragraph combined with Lemma \ref{lem:aonetensorproductmilnorktheorysheaves}, we conclude that $\bpi_n^{\aone}(K_n/m) \tensor^{\aone} \bpi_n^{\aone}(K_n/m) \cong \K^\M_n/m \tensor^{\aone} \K^\M_n/m \cong \K^\M_{2n}/m$. Therefore, in the above exact sequence the left hand map is a map $\K^\M_{2n}/m \cong \bpi_n^{\aone}(K_n/m) \tensor^{\aone} \bpi_n^{\aone}(K_n/m) \to \bpi_{2n-1}^{\aone}(K_n/m)$.

One knows $(\H^{i}_{\et}(\mu_m^{\tensor n}))_{-s} \cong \H^{i-s}_{\et}(\mu_m^{\tensor n-s})$ (appeal to \cite[Example
23.3]{MVW} and sheafify).  Since \'etale cohomology vanishes in negative degrees, we conclude that
$\H^{1}_{\et}(\mu_m^{\tensor n}))$ is killed by $2$-contractions. Since there is an epimorphism $\K^{\MW}_{2n} \to
\K^\M_{2n}/m$, and $(\K^\M_1/m)_{-2n} = 0$, by appealing to Lemma \ref{lem:homfromkmwiscontraction}, we may conclude
that left hand morphism in the exact sequence displayed in the previous paragraph is the trivial map. Therefore, $\H^1_{\et}(\mu_m^{\tensor n}) \cong \bpi_{2n-1}^{\aone}(K_n/m) \to \bpi_{2n}^{\aone}(\Sigma K_n/m)$ is an isomorphism.

In light of the discussion above, by reading the exact sequence of Theorem \ref{thm:lowdegree} further to the left, we conclude that there is a short exact sequence of the form
\[
\bpi_{2n}^{\aone}(K_n/m) \longrightarrow \bpi_{2n+1}^{\aone}(\Sigma K_n/m) \longrightarrow \K^\M_{2n}/m \longrightarrow 0
\]
Then, $\bpi_{2n}^{\aone}(K_n/m) \cong \H^0_{\et}(\mu_m^{\tensor n})$ and this sheaf is killed by a single contraction as discussed in the previous paragraph.

For $i \geq 1$, and $0 \leq j \leq n$ the map $\bpi_{n+j+i}(\Sigma^i K_n/m) \to \bpi_{n+j+i+1}^{\aone}(\Sigma^{i+1}K_n/m)$ is an isomorphism by the $\aone$-Freudenthal suspension theorem. Combining these observations establishes the points listed above.
\end{proof}

\begin{rem}
It is possible to treat the case where $m$ is a power of $p$ as well, but the answer is simpler. If $m$ is a power of $p$, then Geisser--Levine computed the homotopy sheaves of $K_n/m$: by \cite[Theorem 8.3]{GeisserLevine}, $\bpi_i^{\aone}(K_n/m)$ is non-vanishing if and only if $i = n$, in which case $\bpi_n^{\aone}(K_n/m)$ may be described as the unramfied Milnor K-theory sheaf $\K^M_n/m$ (see \cite[Theorem 8.1]{GeisserLevine}). In this case, we conclude that $\pi_{n+j+i}^{\aone}(\Sigma^i K_n/m)$ simply vanishes for $1 < i \leq n-1$.
\end{rem}

{\begin{footnotesize}
\raggedright
\bibliographystyle{alpha}
\bibliography{Jamesconstruction}
\end{footnotesize}
}
\Addresses
\end{document}